\documentclass[a4paper,11pt]{amsart}
  \usepackage{amsmath}
  \usepackage{url,graphicx}
  \usepackage{amssymb, color, pstricks-add, manfnt}
  \usepackage{amsmath, amsthm,  amsfonts}
  \usepackage{mathrsfs}

  \theoremstyle{plain}
  \newtheorem{Theorem}{Theorem}[section]
  \newtheorem{Lemma}{Lemma}[section]

  \newtheorem{Corollary}{Corollary}[section]

    \theoremstyle{remark}
  \newtheorem{remark}{Remark}

  \numberwithin{equation}{section}
  \numberwithin{figure}{section}
  \numberwithin{remark}{section}

\renewcommand{\baselinestretch}{1.00}
\begin{document}

\title{Oblique boundary value problems for augmented Hessian equations I}

\author{Feida Jiang}
\address{College of Mathematics and Statistics, Nanjing University of Information Science and Technology, Nanjing 210044, P.R. China; 
Yau Mathematical Sciences Center, Tsinghua University, Beijing 100084, P.R. China}
\email{jfd2001@163.com}

\author{Neil S. Trudinger}
\address{Centre for Mathematics and Its Applications, The Australian National University, Canberra ACT 0200, Australia;
School of Mathematics and Applied Statistics, University of Wollongong, Wollongong, NSW 2522, Australia}
\email{Neil.Trudinger@anu.edu.au}

\thanks{Research supported by  National Natural Science Foundation of China (Nos.11771214,11401306),  Australian Research Council (No.DP1094303), China Postdoctoral Science Foundation (No.2015M571010) and Jiangsu Natural Science Foundation of China (No.BK20140126).}




\abstract {In this paper, we study  global regularity for oblique boundary value problems of augmented Hessian equations for a class of general operators. By assuming a natural convexity condition of the domain together with appropriate convexity conditions on the matrix function in the augmented Hessian, we develop a global theory for classical elliptic solutions by establishing global {\it a priori} derivative estimates up to second order. Besides the known applications for Monge-Amp\`ere type operators in optimal transportation and geometric optics, the general theory here embraces prescribed mean curvature problems in conformal geometry  as well as oblique boundary value problems for augmented $k$-Hessian, Hessian quotient equations and certain degenerate equations.}

\endabstract

\maketitle


\baselineskip=12.8pt
\parskip=3pt
\renewcommand{\baselinestretch}{1.38}

\section{Introduction}\label{Section 1}
\vskip10pt

In this paper we develop the essentials of a general theory of classical solutions of oblique boundary value problems for certain types of fully nonlinear elliptic partial differential equations, which we describe as augmented Hessian equations. Such problems arise in various applications, notably to optimal transportation, geometric optics and conformal geometry  and our critical  domain and augmenting matrix convexity notions  are adapted from those introduced in \cite{MTW2005, Tru2006, TruWang2009} for regularity in optimal transportation. Our main concern here will be with semilinear boundary conditions but we will also cover the  nonlinear case for appropriate subclasses of our general operators. The classical solvability of the Neumann problem for the Monge-Amp\`ere equation was proved by Lions, Trudinger and Urbas in \cite {LTU1986}. Not only was the approach in \cite{LTU1986} special for the Neumann problem, but it follows from the fundamental example of Pogorelov \cite{Pog} that the result cannot be extended to general linear oblique boundary value problems, \cite{Wang1992, Urbas1995}.  On the other hand, the classical Dirichlet problem for basic Hessian equations has been well studied in the wake of fundamental papers by Caffarelli, Nirenberg and Spruck \cite {CNS-Hessian} and Ivochkina \cite{Ivo}, with further key developments  by several authors, including Krylov  in \cite {Kry1995} and related papers and Trudinger in  \cite {Tru1995}; (see also \cite {Guan2014} for a recent account of the resultant theory under fairly general conditions).

Our main concerns in this paper are second derivative estimates under natural ``strict regularity'' conditions on the augmenting matrices, together with accompanying gradient and H\"older estimates, which then lead to classical existence theorems. Our theory embraces a wide class of examples which we also present as well as a key  application to semilinear Neumann problems arising in conformal geometry, where remarkably our adaptation of optimal transportation domain convexity from \cite{Tru2006,TruWang2009} enables us to remove the rather strong umbilic boundary condition for second derivative bounds, assumed in previous work \cite{SSChen2007, JLL2007}. In ensuing papers we consider extensions to weaker matrix convexity conditions as well as the regularity of  weak solutions and the sharpness of our domain convexity conditions. Extensions to the Dirichlet problem for our general class of equations are  treated in \cite{JT-new}. Overall this paper provides a comprehensive framework for studying oblique boundary value problems for a large class of fully nonlinear equations, which embraces the Monge-Amp\`ere type case in Section 4 in \cite{JTX2015} as a special example.

Specifically we study augmented Hessian partial differential equations of the form,
\begin{equation}\label{1.1}
\mathcal{F}[u]:= F[D^2u-A(\cdot,u,Du)]=B(\cdot,u,Du), \quad \mbox{in} \ \Omega,
\end{equation}
subject to boundary conditions
\begin{equation}\label{1.2}
\mathcal{G}[u]:=G(\cdot,u,Du)=0, \quad \mbox{on} \ \partial\Omega,
\end{equation}
where $\Omega$ is a bounded domain in $n$ dimensional Euclidean space $\mathbb{R}^n$ with smooth boundary, $Du$ and $D^2u$ denote the gradient vector and the Hessian matrix of the solution $u\in C^2(\Omega)$, $A$ is a $n\times n$ symmetric matrix function defined on $\Omega \times \mathbb{R} \times \mathbb{R}^n$, $B$ is a  scalar valued function on $\Omega \times \mathbb{R} \times \mathbb{R}^n$ and $G$ is a scalar valued function defined on $\partial \Omega \times \mathbb{R} \times \mathbb{R}^n$. We use $x$, $z$, $p$, $r$ to denote the points in $\Omega$, $\mathbb{R}$, $\mathbb{R}^n$ and $\mathbb{R}^{n\times n}$ respectively.  The boundary condition \eqref{1.2} is said to be oblique, with respect to $u\in C^1(\bar \Omega)$, if
\begin{equation}\label{1.3}
G_p(\cdot, u,Du)\cdot \nu \ge \beta_0, \quad \mbox{on} \ \partial\Omega,
\end{equation}
where $\nu$ is the unit inner normal vector field on $\partial \Omega$ and $\beta_0$ is a positive constant.
If $G_p\cdot \nu > 0$ on all of $\partial \Omega \times \mathbb{R} \times \mathbb{R}^n$, we will simply refer to $G$
(or $\mathcal G$) as oblique. In the context, we shall use either $\mathcal{F}$ or $F$ to denote the general operator in \eqref{1.1}, and either $\mathcal{G}$ or $G$ to denote the boundary operator in \eqref{1.2}.

 Letting $\mathbb{S}^n$ denote the linear space of $n\times n$ symmetric matrices, the function $F$ in \eqref{1.1}  is defined on  an open, convex cone  $\Gamma$ in $\mathbb{S}^n$, with vertex at $0$, containing the positive cone $K^+$.  In order to consider $F$ in a very general setting, we assume that $F\in C^2 (\Gamma)$ satisfies a subset of the following properties.

\begin{itemize}
\item[{\bf F1}:]
$F$ is strictly increasing in $\Gamma$, namely
\begin{equation}\label{F1 inequality}
F_r := F_{r_{ij}} = \left \{ \frac{\partial F}{\partial r_{ij}} \right \} >0, \ {\rm in} \ \Gamma.
\end{equation}

\vspace{0.2cm}

\item[{\bf F2}:]
$F$ is concave in $\Gamma$, namely
\begin{equation}\label{F2 inequality}
\frac{\partial^2 F}{\partial r_{ij}\partial r_{kl}} \eta_{ij} \eta_{kl} \le 0, \ {\rm in} \ \Gamma,
\end{equation}
for all symmetric matrices $\{\eta_{ij}\}\in \mathbb{S}^n$.

\vspace{0.2cm}

\item[{\bf F3}:]
$F(\Gamma)=(a_0, \infty)$ for a constant $a_0\ge -\infty$ with
\begin{equation}
\sup_{r_0\in \partial\Gamma}\limsup_{r\rightarrow r_0} F(r) \le a_0.
\end{equation}

\vspace{0.2cm}

\item[{\bf F4}:]
$F(tr)\rightarrow \infty$ as $t\rightarrow \infty$, for all $r\in \Gamma$.

\vspace{0.2cm}

\item[{\bf F5}:]
For given constants $a$, $b$ satisfying $a_0<a<b$, there exists a constant $\delta_0>0$ such that $\mathscr{T}(r): ={\rm {trace}}(F_r)\ge \delta_0$ if $a< F(r)< b$.

\vspace{0.2cm}

\item[{\bf F5\textsuperscript{+}}:]
$\mathscr{T}(r) \rightarrow \infty$ uniformly for $a \le F(r) \le b$ as $|r|\rightarrow \infty$.
\end{itemize}

We say an operator $\mathcal{F}$ satisfies the above properties, if the corresponding function $F$ satisfies them.
Note that we can take the constant $a_0$ in F2 and F5 to be $0$ or $-\infty$. When $F$ is given as a symmetric function $f$ of the eigenvalues $\lambda_1, \cdots,\lambda_n$ of the matrix $r$, with $\Gamma$ closed under orthogonal transformations, we will refer to $\mathcal F$ as orthogonally invariant. In this case the above conditions are modelled on the conditions introduced for the study of the Dirichlet problem for the basic Hessian equations, with $A=0$, by Caffarelli, Nirenberg and Spruck in \cite{CNS-Hessian}  and Trudinger in \cite{Tru1995}. The standard operators satisfying the above properties are the $k$-Hessian operators, $F_k=(S_k)^{\frac{1}{k}}$, $k=2,\cdots,n$, which satisfy F1-F4, F5\textsuperscript{+} on $\Gamma_k$ with $a_0=0$, and their quotients $F_{k,l}=(\frac{S_k}{S_l})^{\frac{1}{k-l}}$, $n\ge k>l\ge 1$, which satisfy F1-F5 on $\Gamma_k$ with $a_0=0$, where $S_k$ denotes the $k$-th order elementary symmetric function defined by
    \begin{equation}\label{k-Hessian operator}
        S_k[r]:=S_k(\lambda(r))=\sum_{1\leq i_1<\cdots<i_k\leq n}\lambda_{i_1}\cdots \lambda_{i_k}, \ \ k=1,\cdots,n,
    \end{equation}
and  $\Gamma_k$ denotes the cone defined by
\begin{equation}
\Gamma_k = \{r \in \mathbb{S}^n\ | \ S_j[r]>0, \ \forall j=1,\cdots,k \}.
\end{equation}
As usual we set $F_0 = 1$, so that we can also write the standard $k$-Hessian $F_k$ as the quotient $F_{k,0}$. It turns out that the proofs of our results and their underlying ideas are essentially just as complicated for these special cases as for the general situation so that a reader will not miss the main features of our techniques by restricting attention to them.
More generally when $a_0=0$ and $F$ is positive homogeneous of degree one, then properties F1, F2, F3 imply F4 and F5.
Clearly F4 is obvious and to show F5 we have by the concavity F2, for a positive constant $\mu$ and $r\in \Gamma$,
\begin{equation}\label{1.9}
\mathscr T(r)\ge \frac{1}{\mu} \{F(\mu I) - F(r) +r\cdot F_r\} = F(I) > 0,
\end{equation}
where the equality follows from the homogeneity, which implies $r\cdot F_r = F(r)$. (Clearly it is enough to take $\mu = 1$ here but it is convenient to use a general $\mu$ for later use).

We also note in general that  F2 and F3 imply, for $r\in \Gamma$ and finite $a_0$,
\begin{equation}\label{homogeneity}
0 \le r\cdot F_r \le F(r) - a_0.
\end{equation}
By the concavity F2, we have, for $t>0$,
$$ F(tr) - F(r) \le (t-1)r\cdot F_r,$$
from which \eqref {homogeneity} follows by taking $t$ sufficiently large for the first inequality and sufficiently small for the second. If $a_0 \ge -\infty$, then \eqref{homogeneity} clearly holds if also F4 is satisfied, (or more generally
$\liminf F(tr) > -\infty$ as $t\rightarrow \infty$).
From \eqref{1.9} and \eqref{homogeneity}, we then obtain for $F(r) \le b$,
\begin{equation}\label{1.11}
\mathscr T(r)\ge \frac{1}{\mu} \{F(\mu I) - b\} \ge \delta_0>0,
\end{equation}
for some constant $\delta_0$, depending on $F$ and $b$, by taking $\mu$ sufficiently large, so that condition F5 is itself a consequence of F2 and F4.

As for \eqref{homogeneity} and \eqref{1.11}, the condition F4 is typically more than we need in general and can be dispensed with in most of our estimates. When considering the equation \eqref{1.1}, it will be enough to assume instead $F(tr) >  B(\cdot,u,Du)$ for $r\in \Gamma$ and sufficiently large $t$, (depending on $r$).

In our scenario, we call $M[u]:=D^2u-A(\cdot,u,Du)$ the augmented Hessian matrix.
 Usually, we denote the elements of  $M[u]$ and the matrix $F_r$ in F1 by $w_{ij}=D_{ij}u-A_{ij}$ and $F^{ij}$ respectively. A function $u$ is called admissible in $\Omega$ ($\bar \Omega$) if
\begin{equation}\label{admissible condition}
M[u]\in \Gamma, \quad {\rm in} \ \Omega, \ (\bar \Omega),
\end{equation}
so that the operator $\mathcal{F}$ satisfying F1 is elliptic with respect to $u$ in $\Omega$ ($\bar \Omega$) when \eqref{admissible condition} holds. It is also clear that if $M[u] \in \bar\Gamma$ with $B \in F(\Gamma)$ in $\Omega$ ($\bar \Omega$) then \eqref{1.1} is elliptic with respect to $u$ in $\Omega$ ($\bar \Omega$), namely we require $B>a_0$ in $\Omega$ ($\bar \Omega$) for $F$ satisfying F3.

An important ingredient for regularity of solutions to equations involving the augmented matrix $M[u]$ is the co-dimension one convexity (strict convexity) condition on the matrix $A$ with respect to $p$, that is
\begin{equation}\label{1.8}
A_{ij}^{kl}(x,z,p)\xi_i\xi_j\eta_k\eta_l \geq 0, \ (>0 ),
\end{equation}
for all $(x,z,p)\in\Omega\times \mathbb{R}\times\mathbb{R}^n$,
$\xi,\eta\in\mathbb{R}^n$, $\xi\perp\eta$,
where $A_{ij}^{kl}=D^2_{p_kp_l}A_{ij}$. Note that  we use the standard summation convention throughout this paper that repeated indices indicate summation from $1$ to $n$ unless otherwise specified. As in \cite{Tru2006, JTY2013, JTY2014}, we also call the matrix $A$ regular (strictly regular) if $A$ satisfies \eqref{1.8}. These conditions were originally formulated for optimal transportation
problems in the Monge-Amp\`ere case, $k = n$, in \cite{MTW2005} and \cite{TruWang2009}. The strictly regular condition may also be viewed as a supplementary ellipticity.

We now start to formulate the main theorems in this paper. First we state a local/global second derivative estimate which extends the Monge-Amp\`ere case in \cite{MTW2005} and whose global version is needed for our treatment of the boundary condition \eqref{1.2}.

\begin{Theorem}\label{Th1.1}
Let $u\in C^4(\Omega)$ be an admissible solution of equation \eqref{1.1} in $\Omega$. Assume one of the following conditions:
\begin{itemize}
\item[(i):] F1, F2, F3 and F5\textsuperscript{+} hold;
\item[(ii):] F1, F2, F3, F5 hold, and $B$ is convex with respect to $p$.
\end{itemize}
Assume also $A\in C^2(\bar \Omega\times \mathbb{R}\times \mathbb{R}^n)$ is strictly regular, $B>a_0, \in C^2(\bar \Omega\times \mathbb{R}\times \mathbb{R}^n)$. Then for any domains
$\Omega^\prime \subset\subset\Omega_0$ in $\mathbb{R}^n$, we have the estimate
\begin{equation}\label{local global estimate}
\sup_{\Omega\cap\Omega^\prime} |D^2u|\le \sup_{\partial\Omega\cap\Omega_0}|D^2u|+C,
\end{equation}
for $u \in C^2(\bar\Omega\cap\Omega_0)$, where the constant $C$ depends on $n, A, B,F,\Gamma, \Omega\cap\Omega_0, \Omega\cap\Omega^\prime$
 and $|u|_{1;\Omega}$.
\end{Theorem}

The estimate \eqref{local global estimate} in Theorem \ref{Th1.1} in the case ${\Omega^\prime = \Omega}$ provides us a global estimate which reduces the bound for second derivatives to the boundary. When $\Omega_0 = \Omega$ we get the usual form of the interior estimate, which is already formulated for case (i) in \cite{Tru2006}. A more precise version involving cut-off functions will be presented in Section \ref{Section 2}. For the boundary estimates we need to assume appropriate geometric assumptions on the domain $\Omega$. We consider the operator $\mathcal{F}$ in \eqref{1.1} and domains $\Omega\subset \mathbb{R}^n$ with $\partial\Omega\in C^2$ and $\nu$ denoting the unit inner normal to $\partial\Omega$, $\delta=D-(\nu\cdot D)\nu$ the tangential gradient in $\partial\Omega$, and $P=I-\nu\otimes\nu$ the projection matrix onto the tangent space on $\partial\Omega$, where $I$ is the $n\times n$ identity matrix. (Here the tangential gradient operator is the vector operator 
$\delta=(\delta_1, \cdots, \delta_n)$ with $\delta_i = (\delta_{ij}-\nu_i\nu_j)D_j$, where $\delta_{ij}$ is the usual Kronecker delta). Then we introduce the $A$-curvature matrix on $\partial \Omega$,
\begin{equation}\label{1.6}
K_A[\partial\Omega](x,z,p)= -\delta\nu(x) + P(D_pA(x,z,p)\cdot \nu(x))P.
\end{equation}
We call $\partial\Omega$ uniformly $(\Gamma, A, G)$-convex with respect to an interval valued function $\mathcal{I}(x)$ on $\partial\Omega$, if
\begin{equation}\label{uniform convexity}
K_A[\partial\Omega](x,z,p) + \mu \nu(x)\otimes\nu(x) \in \Gamma,
\end{equation}
for all $x\in \partial\Omega$, $z\in \mathcal{I}(x)$, $G(x,z,p)\ge 0$ and some $\mu=\mu (x,z,p)>0$.  For a given function $u_0$, if we take $\mathcal{I}=\{u_0\}$ in the above definition, then $\partial\Omega$ is  called uniformly $(\Gamma, A, G)$-convex with respect to $u_0$. For the cases $\Gamma = \Gamma_k$, corresponding to the $k$-Hessians and their quotients, \eqref{uniform convexity} is equivalent to $K_A[\partial\Omega](x,z,p) \in \Gamma_{k-1}$. Moreover in the Monge-Amp\`ere case, $k=n$, we recover our definitions of uniform $(A,G)$-convexity in \cite{JTX2015}, which extend the notion of uniform $c$-convexity with respect to a target domain $\Omega^*$ in the optimal transportation  case as introduced in \cite{TruWang2009}. When the interval $\mathcal I = \mathbb R$, we will simply call $\partial\Omega$ uniformly $(\Gamma, A, G)$-convex. This includes the case when $A$ and $G$ are independent of $z$ as then the interval $\mathcal I$ becomes irrelevant.

As in  \cite{JTX2015}, we will assume that the function $G\in C^2(\partial \Omega\times \mathbb{R}\times \mathbb{R}^n)$ is concave in $p$,  that is $G_{pp}\le 0$ in $\partial \Omega \times \mathbb{R} \times \mathbb{R}^n$.
This includes the quasilinear  case, when $G_{pp}= 0$,
\begin{equation}\label{quasilinear}
G(x,z,p) = \beta(x, z)\cdot p - \varphi(x,z),
\end{equation}
where $\beta = G_p$ and $\varphi$ are defined on $\partial\Omega\times \mathbb R$. If $G_{pp} (\cdot,u,Du)\le 0$ on $\partial\Omega$ for $u\in C^1(\bar\Omega)$ then we say that $G$ is concave in $p$, with respect to $u$.
Note that we define the obliqueness in \eqref{1.3} with respect to the unit inner normal $\nu$, so that our function $G$ keeps the same sign with those in \cite{JTX2015} and is the negative of that in \cite{Urbas1997, TruWang2009, Tru2008}.  When $G$ is nonlinear in $p$ we will assume a further structural condition on $F$.

\begin{itemize}
\item[{\bf F6}:]
$\mathcal E_2: = F_{r_{ij}} r_{ik}r_{jk} \le o( |r|){\mathscr T}$,  uniformly for $a_0 <a \le F(r) \le b$, as $|r|\rightarrow \infty.$
\end{itemize}
We remark that the Hessian operators $F_k$ ($k=1,\cdots, n$) and the Hessian quotients $F_{n,k}$ ($1\le l\le n-1$) satisfy F6 in the positive cone $K^+$ \cite{Urbas2001}. Further examples are given in Section 4.2. To complete our hypotheses, we will also assume for the second derivative bounds in this paper, (unless F6 is satisfied), that the cone $\Gamma$ lies strictly in a half space in the sense that $r \le$ trace$(r)I$ for all $r\in\Gamma$, that is $ \Gamma \subset \mathcal P_{n-1}$  in accordance with our examples in Section \ref{subsection4.2}. This property is satisfied by the cones $\Gamma_k$ for $k\ge 2$, (but excludes the already well known quasilinear case when $k = 1$). We now state the global second derivative bound which can be viewed as the main result of this paper.

\begin{Theorem}\label{Th1.2}
Let $u\in C^4(\Omega)\cap C^{3}(\bar \Omega)$ be an admissible solution of the boundary value problem \eqref{1.1}-\eqref{1.2} in a $C^{3,1}$  domain $\Omega\subset\mathbb{R}^n$, with boundary $\partial\Omega$  uniformly $(\Gamma,A,G)$-convex with respect to $u$. Assume that F satisfies conditions F1-F5, $A\in C^2(\bar \Omega\times \mathbb{R}\times \mathbb{R}^n)$ is strictly regular in $\bar\Omega$, $B>a_0,\in C^2(\bar \Omega\times \mathbb{R}\times \mathbb{R}^n)$, $G\in C^{2,1}(\partial\Omega\times\mathbb{R}\times\mathbb{R}^n)$ is oblique  and concave in $p$ with respect to $u$
satisfying \eqref{1.3}, and either F5\textsuperscript{+} holds or $B$ is independent of $p$. Assume further that either
$\mathcal G$ is quasilinear and $ \Gamma \subset \mathcal P_{n-1}$ or F also satisfies  F6. Then we have the estimate
\begin{equation}\label{1.10}
\sup\limits_\Omega |D^2 u|\le C,
\end{equation}
where $C$ is a constant depending on $F, A, B, G, \Omega, \beta_0$ and $|u|_{1;\Omega}$.
\end{Theorem}

\begin{remark}\label{Remark1.1} A stronger condition than regularity of the matrix function $A$ is necessary in the above hypotheses as it is known from the Monge-Amp\`ere case that one cannot expect second derivative estimates for  general oblique boundary value problems for $A\equiv 0$, which is a special case of regular $A$ but not strictly regular, see \cite{Urbas1995}, \cite{Wang1992}. We also remark that the alternative condition that $B$ is independent of $p$ may be replaced by $D_pB$ sufficiently small, as well as $B$ convex with respect to $p$, and we will see from our treatment in Section \ref{Section 2} that such a condition is reasonable. Analogously, we may also replace the condition that $G$ is quasilinear by $D^2_pG$ sufficiently small.
\end{remark}

\begin{remark}\label{Remark1.2} It should be noted that the feasibility of our uniform convexity condition depends on an effective relationship between the boundary operator $\mathcal G$  and the curvature matrix $K_A[\partial\Omega]$ to ensure at least that the matrix
$$P(D_pA(x,z,p)\cdot \nu(x))P$$
is uniformly bounded from below, for all $x\in \partial\Omega$, $z\in \mathcal{I}(x)$, $G(x,z,p)\ge 0$. More generally  we need to impose a condition on the gradient $Du$, namely that there exists a sufficiently small boundary neighbourhood  $\Omega_{\rho} = \{x\in \Omega | \ d(x)<\rho\}$, where $d(x)={\rm dist}(x,\partial\Omega)$ and $\rho$ is a small positive constant, such that
\begin {equation}\label{general uniform convexity}
-\delta\nu(x^\prime) + P(D_pA(x,u(x),Du(x))\cdot \nu(x^\prime))P + \mu_0 \nu(x^\prime)\otimes\nu(x^\prime) \in \Gamma,
\end{equation}
for all $x\in \Omega_{\rho}$ and $x^\prime\in \partial \Omega$ satisfying $d(x,x^\prime) = d(x)$ and
$G(x^\prime,u(x),Du(x)) \ge 0$, for a positive constant $\mu_0$. In particular we would then obtain a bound \eqref{1.10} when the curvatures of $\partial \Omega$ are sufficiently large.  Furthermore since the regularity of $A$   on $\partial\Omega$ implies the curvature matrix  $K_A[\partial\Omega](x,z,p)$ is non-decreasing in $p\cdot\nu$, in the Neumann case, \eqref{quasilinear} with $\beta=\nu$, we need only assume that \eqref{uniform convexity} holds for $x\in\partial\Omega$, $z=u(x)$, $p = Du(x)$ and constant $\mu =\mu_0$, that is $x = x^\prime \in\partial\Omega$ in \eqref{general uniform convexity}.
\end{remark}

\begin{remark}\label{Remark 1.3}
We may assume more generally that the matrix function $A$ and scalar function $B$ are only defined and $C^2$ smooth on some open set
$\mathcal U \subset \mathbb{R}^n \times \mathbb{R} \times \mathbb{R}^n$, with $A$ strictly regular and $B > a_0$ in $\mathcal U$. Then Theorems \ref{Th1.1} and \ref{Th1.2}
will continue to hold provided the one jet $J_1 =J_1[u](\Omega) = (\cdot, u, Du)(\Omega)$ is strictly contained in $\mathcal U$, with the constants $C$ in the estimates \eqref{local global estimate} and \eqref{1.10} depending additionally on dist$(J_1, \partial\mathcal U)$. This remark is particularly pertinent to examples arising from optimal transportation or geometric optics where often the resultant Monge-Amp\`ere type equations are subject to constraints on $J_1[u]$ and moreover such constraints may determine an appropriate constant $\mu_0$ in \eqref{general uniform convexity}.
\end{remark}

In order to apply Theorem \ref{Th1.2}, to the existence of smooth solutions to \eqref{1.1}-\eqref{1.2}, we need gradient and solution estimates. Our conditions for gradient estimates are motivated by the case when $F$ is linear and the corresponding conditions for gradient estimates for uniformly elliptic quasilinear equations, as originally introduced by Ladyzhenskaya and Ural'tseva \cite{GTbook, LU1968}. First we need additional conditions on either $A$ or $F$ which facilitate an analogue of uniform ellipticity.

To formulate the condition on the matrix function $A$, we first express the strict regularity condition \eqref{1.8} in the equivalent form,
\begin{equation}\label{1.18}
A_{ij}^{kl}\xi_i\xi_j\eta_k\eta_l \geq \lambda|\xi|^2|\eta|^2 - \bar\lambda(\xi\cdot\eta)^2,
\end{equation}
for all $\xi, \eta \in \mathbb{R}^n$, where $\lambda$ and $\bar\lambda$ are positive functions in
$C^0(\Omega\times \mathbb{R}\times \mathbb{R}^n)$. To derive \eqref{1.18} from \eqref{1.8}, we
set $\eta^\prime = \eta -(\xi\cdot\eta)\xi$, for $|\xi|=|\eta| = 1$, and apply \eqref{1.8} to the orthogonal vectors
$\xi$ and $\eta^\prime$. We then call $A$ uniformly regular in $\Omega$, if $A$ is strictly regular in
 $\Omega$ and
for any $M > 0$, there exist positive constants $\lambda_0$ and $\bar\lambda_0$ such that
\begin{equation}\label{1.19}
\lambda(x,z,p) \ge \lambda_0, \quad \bar\lambda(x,z,p) \le\bar\lambda_0,
\end{equation}
for all $x\in\Omega$, $|z| \le M$, $p\in \mathbb{R}^n$.
The alternative condition on $F$ extends that introduced in the orthogonally invariant case for gradient estimates for curvature equations in \cite{CNS-5}.
\begin{itemize}
\item[{\bf F7}:]
 For  a given constant $a> a_0$, there exists constants $\delta_0, \delta_1>0$ such that
 $F_{r_{ij}}\xi_i\xi_j \ge  \delta_0 + \delta_1 \mathscr T $,  if $a\le F(r)$ and $\xi$ is a unit eigenvector of
 $r$ corresponding to a negative eigenvalue.
 \end{itemize}
We remark that F7 implies F5, with $b=\infty$, and moreover the Hessian quotients $F_{k,l}$, for $0\le l<k \le n$ satisfy F7 in the cone $\Gamma_k$ with constants $\delta_0, \delta_1>0$, depending only on $k,l$ and $n$, \cite{CNS-5,Tru1990}.

We formulate (almost) quadratic growth conditions on $A$ and $B$ as follows.
\begin{equation} \label{quadratic growth}
D_xA, D_xB, D_zA, D_zB = O(|p|^2), \quad D_zA\ge o(|p|^2)I, D_zB\ge o(|p|^2),\quad D_pA, D_pB = O(|p|),
\end{equation}
as $|p| \rightarrow \infty$, uniformly for $x\in\Omega$, $|z|\le M$ for any $M > 0$. Note that in the analogous  natural growth conditions in the uniform elliptic theory, the ``$o$'' lower bounds on $D_z A$ and $D_z B$ in \eqref{quadratic growth} can be dispensed with as a continuity
estimate is available \cite{GTbook, LieTru1986}. Also these  are automatically satisfied under standard uniqueness conditions, namely when $A$ and $B$ are non-decreasing in $z$, that is $D_z A \ge 0$ and $D_z B \ge 0$.

We now state a gradient estimate for oblique semilinear boundary conditions, that is when $\beta$ in \eqref{quasilinear} is independent of $z$ so that \eqref{1.2} may be written in the form
\begin{equation}\label{semilinear}
\mathcal G[u] = G(\cdot,u,Du) =\beta\cdot Du -\varphi(\cdot,u) = 0, \quad \mbox{on} \ \partial\Omega.
\end{equation}
Some variants and extensions, including weaker versions of conditions \eqref{quadratic growth}, local gradient estimates and extensions to nonlinear $G$  will also be considered in conjunction with our treatment in Section \ref{Section 3}.
\begin{Theorem}\label{Th1.3}
Let $u\in C^3(\Omega)\cap C^{2}(\bar \Omega)$ be an admissible solution of the boundary value problem \eqref{1.1}-\eqref{1.2} for an oblique, semilinear  boundary operator $\mathcal G$ in a $C^{2,1}$  domain $\Omega\subset\mathbb{R}^n$. Assume  that  $\mathcal F$ satisfies F1 and F3, $A, B \in C^1(\bar\Omega\times\mathbb{R}\times \mathbb{R}^n)$, satisfy \eqref{quadratic growth}, $b_0:= \inf\limits_\Omega B(\cdot, u,Du) >a_0$  and $\beta\in C^2(\partial\Omega)$,  $\varphi\in C^2(\partial\Omega\times \mathbb{R})$. Assume also one of the following further conditions:
\begin{itemize}
\item[(i):] $A$ is uniformly regular, $\mathcal F$ satisfies F2  and either  {\rm (a)} F5\textsuperscript{+}, with $b=\infty$, or {\rm (b)} F5, with $b=\infty$,  and   $B -p\cdot D_pB\le o(|p|^2)$ in \eqref{quadratic growth};
\item[(ii):]  $\mathcal{F}$ is orthogonally invariant satisfying F7, $A = o(|p|^2)$  in \eqref{quadratic growth} and
 $\beta = \nu$.
\end{itemize}
Then we have the estimate
\begin{equation}\label{1.21}
\sup\limits_\Omega |D u|\le C,
\end{equation}
where $C$ is a constant depending on $F, A, B, \Omega,b_0, \beta, \varphi$ and $|u|_{0;\Omega}$.
\end{Theorem}

As we will show in Section \ref{Section 3} the concavity condition F2 in Theorem \ref{Th1.3} may be removed when $F$ is positive homogeneous of degree one and more generally. Note that the condition on $B$ in case (i) is automatically satisfied when $B$ is convex in $p$. Also when $B$ is bounded, we do not need to take $b=\infty$ in F5, while if the constants $\delta_0$ and $\delta_1$ in conditions F5 and F7 are independent of $a$, the constant $C$ in the estimate \eqref {1.21} does not depend on $b_0$ in  cases (i)(b) and (ii).
Analogously to the situation with uniformly elliptic equations, we obtain gradient estimates in terms of moduli of continuity when ``$o$'' is weakened to ``$O$'' in the hypotheses, \eqref{quadratic growth} and case (ii), of  Theorem \ref{Th1.3}. In particular we will also prove a H\"older estimate for admissible functions in the cones $\Gamma_k$ for $k>n/2$, when $A\ge O(|p|^2)I$, which extends our gradient estimate in the case $k=n$ in \cite{JTX2015}, Lemma 4.1. Taking account of this, as well as Theorems \ref{Th1.2} and \ref{Th1.3}, we have, as an example of our consequent existent results, the following existence theorem for the augmented $k$-Hessian and Hessian quotient equations. In its formulation we will assume the existence of subsolutions and supersolutions to provide the necessary solution estimates and an appropriate interval $\mathcal I$ in our boundary convexity conditions. For this purpose we will say that functions $\underline u$ and $\bar  u$, in $C^2(\Omega)\cap C^1(\bar \Omega)$, are respectively subsolution and supersolution of the boundary value problem \eqref{1.1}-\eqref{1.2} if
\begin{equation}\label{sub super F}
\mathcal F[\underline u] \ge B(\cdot, \underline u, D\underline u), \quad \mathcal F[\bar u] \le B(\cdot, \bar u, D\bar u),
\end{equation}
at points in $\Omega$, where they are admissible, and
\begin{equation}\label{sub super G}
\mathcal G [\underline u] \ge 0, \quad \mathcal G[\bar u] \le 0, \quad \mbox{on} \ \partial\Omega.
\end{equation}

\begin{Theorem}\label{Th1.4}
Let $\mathcal F = F_{k,l}$ for some $0\le l < k \le n$, $A\in C^2(\bar \Omega\times \mathbb{R}\times \mathbb{R}^n)$ strictly regular in $\bar\Omega$, $B > 0, \in C^2(\bar \Omega\times \mathbb{R}\times \mathbb{R}^n)$,
$\mathcal G$ is semilinear and oblique with $G\in C^{2,1}(\partial\Omega\times \mathbb{R}\times\mathbb{R}^n)$. Assume that $\underline u$ and $\bar  u$, $\in C^2(\Omega)\cap C^1(\bar \Omega)$ are respectively an admissible subsolution and a supersolution of \eqref{1.1}-\eqref{1.2} with $\partial\Omega$ uniformly $(\Gamma_k,A,G)$-convex with respect to the interval $\mathcal I = [\underline u, \bar u]$.
Assume also that $A$, $B$ and $\varphi$ are non-decreasing in $z$, with at least one of them strictly increasing, and that $A$ and $B$ satisfy the quadratic growth conditions \eqref{quadratic growth} with $D_pB =0$ if $l>0$. Then if one of the following further conditions is satisfied:
\begin{itemize}
\item[(i):]  $A$ is uniformly regular and $B - p\cdot D_pB\le o(|p|^2)$  in \eqref{quadratic growth};
\item[(ii):]   $\beta=\nu$ and either {\rm (a)} $A = o(|p|^2)$  in \eqref{quadratic growth} or {\rm (b)}  $k>n/2$ and $\Omega$ is convex;
\item[(iii):]  $k= n$ and  $A\ge O(|p|^2)I$ in place of \eqref{quadratic growth},
\end{itemize}
there exists a unique admissible solution $u \in C^{3,\alpha}(\bar \Omega)$ of the boundary value problem \eqref{1.1}-\eqref{1.2}, for any $\alpha<1$.
\end{Theorem}

This paper is organised as follows. In Section \ref{Section 2}, we first prove the local/global second derivative estimate, Theorem \ref{Th1.1}, as well as an extension to non-constant vector fields, Lemma \ref{Lemma 2.1}. Then in Section \ref{subsection2.2}, by delicate analysis of the second derivatives on the boundary, we complete the proof of Theorem \ref{Th1.2} through Lemmas \ref{Lemma 2.2} and \ref{Lemma 2.3} which treat respectively the estimation of non-tangential and tangential second derivatives. In the proof of Lemma \ref{Lemma 2.3} the strict regularity condition is crucial. In Section \ref{Section 3}, we first prove the global gradient estimate, Theorem \ref{Th1.3}, under various more general structural assumptions on $F$, $A$ and $B$. Following this, in Section \ref{subsection3.2}, we prove the analogous local gradient estimates in Theorem \ref{Th3.1}. In Section \ref{subsection3.3} we derive a H\"older estimate for admissible functions in the cones $\Gamma_k$ for $k>n/2$, from which we can infer gradient estimates under natural quadratic growth conditions. In Section \ref{Section 4}, we prove  existence theorems, Theorems \ref{Th4.1} and \ref{Th4.2} for semilinear and nonlinear oblique boundary value problems  based on the {\it a priori} derivative estimates, which include Theorem \ref{Th1.4} as a special case. We then present in Section \ref{subsection4.2} various examples of operators $\mathcal{F}$, matrices $A$, and boundary operators $\mathcal{G}$ along with the application to conformal geometry, where we relax the umbilic boundary  restriction for second derivative estimates in Yamabe problems with boundary as studied in  \cite{SSChen2007, JLL2007}. Furthermore we show in Section \ref{subsection4.3} that our theory can be applied to  degenerate elliptic equations, where $F$ is only assumed non-decreasing in F1; see Corollary \ref{Cor4.1},  and provide a particular example in Corollaries \ref{Cor4.2} and \ref{Cor4.3}. In Section \ref{subsection4.4}, we conclude this paper with some final remarks which also foreshadow further results.

\section{Second derivative estimates}\label{Section 2}
\vskip10pt

We introduce some notation and proceed to the second order derivative estimates for admissible solutions $u$ of \eqref{1.1}-\eqref{1.2}.
We denote the augmented Hessian $M[u]$ by $w=\{w_{ij}\}$, that is
\begin{equation}\label{2.10}
w_{ij}=D_{ij}u-A_{ij}(\cdot,u,Du).
\end{equation}
As usual we denote the first and second partial derivatives of $F$ at $M[u]$ by $F^{ij}$ and $F^{ij,kl}$, namely
\begin{equation}\label{2.12}
F^{ij}=\frac{\partial F}{\partial r_{ij}}(M[u]),\quad  F^{ij,kl}=\frac{\partial^2 F}{\partial r_{ij}\partial{r_{kl}}}(M[u]).
\end{equation}
Then for an admissible $u$, we know from F1 that the matrix $\{F^{ij}\}$ is positive definite and from F2 that
$$F^{ij,kl}\eta_{ij} \eta_{kl} \le 0$$
for all $\{\eta_{ij}\}\in \mathbb{S}^n$. Let us also denote $\mathscr{T}= {\rm trace}(F_r)=\sum\limits_{i=1}^nF^{ii}$ so that
by positivity $|F_r| \le \mathscr{T}$.

It will also be convenient here to use \eqref{1.18} to express the strict regularity of $A$, with respect to $u$, in the form
\begin{equation}\label {strict regularity}
A_{ij}^{kl}(\cdot,u,Du)\xi_i\xi_j\eta_k\eta_l \ge c_0 |\xi|^2 |\eta|^2 - c_1(\xi\cdot\eta)^2,
\end{equation}
for arbitrary vectors $\xi, \eta \in \mathbb{R}^n$, where $c_0$ and $c_1$ are positive constants depending on
$A$ and sup$(|u| +|Du|)$. Then for any positive symmetric matrix  $\{F^{ij}\}$ with eigenvalues $\lambda_1,\cdots, \lambda_n >0$ and corresponding eigenvectors, $\phi^1, \cdots, \phi^n$, we can write
\begin{equation}\label{critical inequality}
\begin{array}{ll}
F^{ij} A_{ij}^{kl}\eta_k\eta_l  \!\!&\!\!\displaystyle= \sum_{s=1}^n \lambda_sA_{ij}^{kl}\phi^s_i\phi^s_j \eta_k\eta_l\\
                                                  \!\!&\!\!\displaystyle\ge c_0\sum_{s=1}^n\lambda_s|\eta|^2 -c_1\sum_{s=1}^n\lambda_s(\phi^s\cdot\eta)^2\\
                                                  \!\!&\!\!\displaystyle= c_0\mathscr{T}|\eta|^2 -c_1F^{ij}\eta_i\eta_j.
 \end{array}
 \end{equation}


\subsection{Local/global second derivative estimates}\label{subsection2.1}
In this subsection, we derive the local and global second derivative estimates for admissible solutions of equation \eqref{1.1}, and give the proof of Theorem \ref{Th1.1}.
We will need to differentiate the equation \eqref{1.1}, with respect to  vector fields $\tau = (\tau^1, \cdots, \tau^n)$ with $\tau^i\in C^2(\bar\Omega), i=1,\cdots,n$. We introduce the linearized operators of the operator $\mathcal F$  and equation \eqref{1.1},
\begin{equation}\label{linearized operators}
 Lv:= F^{ij}[D_{ij}v-A_{ij}^kD_kv], \quad \mathcal{L}v:= Lv-(D_{p_k}B(\cdot,u,Du))D_kv,
\end{equation}
for $v\in C^2(\Omega)$, where $A_{ij}^k=D_{p_k}A_{ij}(\cdot,u,Du)$. For convenience below we shall as usual denote partial derivatives of functions on $\Omega$ by subscripts, that is $u_i = D_i u, u_\tau = D_\tau u = \tau^i u_i,
u_{ij} = D_{ij}u, u_{i \tau} = u_{ij}\tau^j, u_{\tau\tau} = u_{ij}\tau^i\tau^j $ etc.
Differentiating once we now obtain,
\begin{equation} \label{2.6}
\mathcal{L}u_\tau = F^{ij}\tilde D_{x_\tau} A_{ij} + \tilde D_{x_\tau}  B
                                 + F^{ij}(2\tau^k_i u_{jk} +\tau^k_{ij}u_k - A_{ij}^k \tau^l_ku_l) - (D_{p_k}B) \tau^l_ku_l.
\end{equation}
where $\tilde D_{x_\tau} = \tau\cdot\tilde D_x$  and  $\tilde D_x = D_x +DuD_z$.
Differentiating twice, we then obtain
\begin{equation}
\begin{array}{ll}\label{2.7}
\mathcal{L}u_{\tau\tau}= \!&\!\!\displaystyle -F^{ij,kl} D_{\tau}w_{ij}D_{\tau}w_{kl} +F^{ij}[\tau^k\tau^l\tilde D_{x_kx_l}A_{ij}+A_{ij}^{kl}u_{k\tau}u_{l\tau}
+2(\tilde D_{x_\tau}A^k_{ij})u_{k\tau}] \\
  \!&\!\!\displaystyle  + (D_{p_kp_l}B)u_{k\tau} u_{l\tau} +\tau^k\tau^l\tilde D_{x_kx_l} B +  2(\tilde D_{x_\tau} D_{p_k} B)u_{k\tau}\\
  \!&\!\!\displaystyle  + F^{ij}[4\tau^k_i u_{jk\tau} + (\tau^k\tau^l)_{ij} u_{kl}- 2A_{ij}^s \tau^k_s u_{k\tau}]
  -2(D_{p_s}B) \tau^k_su_{k\tau},
 \end{array}
\end{equation}
where $D_p\tilde D_x = \tilde D_x D_p$ is used. To derive the local and global estimates in Theorem \ref{Th1.1}, we only need $\tau$ to be a constant unit vector, in which case the last two terms in \eqref{2.6} and \eqref{2.7} are not present. Setting
\begin{equation}\label{2.8}
v_1 = u_{\tau\tau} + \frac{1}{2}c_1 |u_\tau|^2,
\end{equation}
we then have from the concavity F2 and strict regularity \eqref{strict regularity},
\begin{equation}\label{2.9}
\mathcal{L}v_1 \ge c_0\mathscr{T}|Du_\tau|^2 - C(1+\mathscr{T})(1+|Du_\tau|)  + \lambda_B |Du_\tau|^2,
\end{equation}
where $C$ is a constant depending on  $n,  |A|_{C^2}, |B|_{C^2}$ and $|u|_{1;\Omega}$ and $\lambda_B$ is the minimum eigenvalue of the matrix $D^2_pB$.
Invoking conditions F5\textsuperscript{+}, or F5 and convexity of $B$ in $p$, we then have from the classical maximum principle, under the hypotheses of Theorem \ref{Th1.1},
\begin{equation}\label{2.10}
\sup_{\Omega}u_{\tau\tau}\le \sup_{\partial\Omega}|u_{\tau\tau}| + C,
\end{equation}
which implies a global upper bound for $D^2u$, since $\tau$ can be any unit vector. To get the corresponding
local estimate, we fix a function $\zeta\in C^2(\mathbb R^n)$, satisfying $0\le\zeta\le 1$ and define,
\begin{equation}\label{2.11}
v= \zeta^2 v_1 = \zeta^2(u_{\tau\tau} + \frac{1}{2}c_1|u_\tau|^2).
\end{equation}
From the inequality \eqref{2.9} we obtain the corresponding inequality for $v$ at a maximum point, namely
\begin{equation}\label{2.12}
\mathcal{L}v \ge c_0\mathscr{T}\zeta^2|Du_\tau|^2 - C(1+\mathscr{T})(1+\zeta^2|Du_\tau| + |u_{\tau\tau}|)  + \lambda_B \zeta^2|Du_\tau|^2,
\end{equation}
 where $C$ depends additionally on $|\zeta|_{2;\Omega}$,
so that extending \eqref{2.10}, we have
\begin{equation}\label{upper bound}
\sup_{\Omega}(\zeta^2u_{\tau\tau})\le \sup_{\partial\Omega}(\zeta^2 |u_{\tau\tau}|) + C.
\end{equation}
The lower bounds follow from the concavity F2 since for a fixed matrix $r_0\in \Gamma$, for example $r_0=I$, and positive matrix $a^{ij}_0 = F_{r_{ij}} (r_0)$, we have
\begin{equation}\label{lower second bound}
\begin{array}{ll}
a^{ij}_0u_{ij} \!\!&\!\!\displaystyle = a^{ij}_0 (w_{ij} + A_{ij}) \\
                         \!\!&\!\!\displaystyle \ge F(w) -F(r_0) + \sum\limits_{i=1}^na^{ii}_0 +a^{ij}_0A_{ij} \ge -C
 \end{array}
 \end{equation}
 by virtue of \eqref{1.1}. Taking $\tau$ to be an eigenvector of $\{a^{ij}_0\}$, we infer the full bound from the
 upper bound \eqref{upper bound}. Hence we conclude the estimate
 \begin{equation}
 \sup_{\Omega}(\zeta^2|D^2u|)\le \sup_{\partial\Omega}(\zeta^2 |D^2u|) + C.
 \end{equation}
Theorem \ref{Th1.1} now follows by taking $\zeta\in C^2_0(\Omega_0)$ and $\zeta =1$ on $\Omega^\prime$.

\vskip8pt

The one-sided estimate \eqref{upper bound} can be extended to non-constant  vector fields $\tau$ when $\mathcal F$ is orthogonally invariant. Moreover the relevant calculations will be critical for us in the proof of  tangential boundary estimates when  $G$ is nonlinear in $p$.

\begin{Lemma}\label{Lemma 2.1}
Assume in addition to the hypotheses of Theorem \ref{Th1.1} that the operator $\mathcal F$ is orthogonally invariant.
Then the estimate \eqref{upper bound} holds for any vector field $\tau$ with skew symmetric Jacobian, with the constant $C$ depending additionally on $|\tau|_{1;\Omega}$.
\end{Lemma}
\begin{proof} First we note that the Jacobian $D\tau = \{\tau^i_j\}$ will in fact be a constant skew symmetric matrix so that
$\tau$ itself is an affine mapping. Consequently the second derivatives of $\tau$ in \eqref{2.6} and \eqref{2.7} will vanish.
Our main concern now is to control the third derivatives of $u$ in the last line of \eqref {2.7} and for this we adapt the key identities in the proof of Lemma 2.1 in \cite{ITW2004}, which follow by differentiating $F(P_\alpha r P^t_\alpha)$, for $r\in\Gamma$, with respect to $\alpha$ and setting $\alpha = 0$, where $P_\alpha = {\rm exp}(\alpha D\tau)$ is orthogonal by virtue of the skew symmetry of $D\tau$. Thus taking $r_{ij} = w_{ij}$, we have
\begin{equation}\label{2.16}
F^{ij}\tau^k_iw_{jk} = 0
\end{equation}
and
\begin{equation}\label{2.17}
F^{ij}(\tau^k_i\tau^l_j w_{kl} + \tau^k_i\tau^l_k w_{jl}) + 2 F^{ij,kl}\tau^s_iw_{js}\tau^t_k w_{lt} = 0.
\end{equation}
Differentiating \eqref{2.16} with respect to $\tau$, we have
\begin{equation}\label{2.18}
F^{ij}\tau^k_iD_\tau w_{jk} + F^{ij,kl}\tau^s_i w_{js}D_\tau w_{kl} = 0.
\end{equation}
From \eqref{2.17} and \eqref{2.18}, we then obtain
\begin{equation}
\begin{array}{ll}\label{2.19}
               \!\!&\!\!\displaystyle -F^{ij,kl}D_{\tau}w_{ij}D_{\tau}w_{kl}+2F^{ij}[2(\tau^k_i)D_\tau w_{jk} + \tau^k_i\tau^l_j w_{kl}]  \\
           =   \!\!&\!\!\displaystyle -F^{ij,kl}(D_\tau w_{ij} +2\tau^s_i w_{js})(D_\tau w_{kl} +2\tau^t_k w_{lt}) -2F^{ij}\tau^k_i\tau^l_k w_{jl} \\
           \ge \!\!&\!\!\displaystyle -2F^{ij}\tau^k_i\tau^l_k w_{jl}
\end{array}
\end{equation}
by the concavity F2. Substituting into \eqref{2.7}, using the definition $w_{ij} = u_{ij} - A_{ij}$ and following our previous argument for constant $\tau$, we would obtain  the upper bound \eqref {upper bound}, with constant $C$ replaced by $C(1+\sqrt M_2)$ where  $M_2 =$ sup$_\Omega |D^2u|$.

 In order to get the full strength of Lemma \ref{Lemma 2.1}, we need to control the last term in \eqref{2.19}. Note that this term is nonnegative if $\Gamma = \Gamma_n$ or in the special case when $|D\tau^i| = \tau_0, i=1,\cdots,n$, for a constant $\tau_0$, since
$$-F^{ij}\tau^k_i\tau^l_k w_{jl} \ge (\tau_0)^2 F^{ij}w_{ij} \ge 0$$
from \eqref {homogeneity}.  In general we proceed by calculating
\begin{equation} \label{2.20}
\mathcal L (\tau^l\tau^k_lu_k) =  F^{ij}(2\tau^k_i\tau^l_ku_{jl}  - A^s_{ij}\tau^k_s\tau^l_k u_l) -
(D_{p_s}B)\tau^k_s\tau^l_k u_l +\tau^l\tau^k_l \mathcal L u_k.
\end{equation}
Taking account of \eqref{2.6}, \eqref{2.7}, \eqref{2.19} and \eqref{2.20}, we then obtain the differential inequality \eqref{2.12} with the function $u_{\tau\tau}$  in \eqref{2.11} replaced by the function
$$(u_\tau)_\tau= u_{\tau\tau} +\tau^l\tau^k_lu_k,$$
from which Lemma \ref{Lemma 2.1} follows.
Note that in the process of obtaining the inequality \eqref{2.12}, there is a term $c_1\zeta^2u_\tau F^{ij}\tau_i^kw_{jk}$ which is identically equal to zero by using \eqref{2.16}.
\end{proof}

We remark that for the Monge-Amp\`ere operator, in the form $F(r) =\log(\det r)$, we can take $\tau$ to be any $C^2$ vector field in \eqref{upper bound} with the constant $C$ now depending  additionally on $|\tau|_{2;\Omega}$. This follows from the identity $F^{ik}r_{kj} = \delta_{ij}$.

\vskip8pt


\subsection{Boundary second derivative estimates}\label{subsection2.2}

To prove Theorem \ref{Th1.2}, we  have to establish estimates for second derivatives on the boundary $\partial\Omega$ under the boundary condition \eqref{1.2}. First we will consider
the non-tangential estimates and as in \cite{JTX2015}, the geometric convexity hypotheses on the domain
$\Omega$ in Theorem \ref{Th1.2} are crucial for this stage. We assume that the functions $G(\cdot, z,p)$ and $\nu$ have been extended to $\bar\Omega$, to be constant along normals to $\partial \Omega$ in some neighbourhood $\mathcal N$ of $\partial\Omega$. Differentiating the boundary condition \eqref{1.2} with respect to a tangential vector field $\tau$ we have
\begin{equation}\label{tangen diff the bc}
\tilde D_{x_\tau} G+ (D_{p_k}G) u_{k \tau }  = 0,\quad {\rm on} \ \partial\Omega,
\end{equation}
and hence we have an estimate
\begin{equation}\label{mixed derivative estimate}
|u_{\tau\beta}|\le C,\quad {\rm on} \ \partial\Omega,
\end{equation}
for any unit tangential vector field $\tau$, where $\beta=D_{p}G(\cdot,u,Du)$ and the constant $C$ depends on $G, \Omega$ and $|u|_{1;\Omega}$.  The estimation of the pure second order oblique derivatives $u_{\beta\beta}$ is much more complicated. In general we can only obtain an estimate from above in terms of the tangential derivatives on the boundary. Setting
$$ M_2 = \sup\limits_\Omega |D^2 u|, \quad M^\prime_2 = \sup\limits_{\partial\Omega} \sup\limits_{|\tau| = 1, \tau\cdot\nu=0} |u_{\tau\tau}|, $$
we formulate this as follows.
\begin{Lemma}\label{Lemma 2.2}
Let $u\in C^{3}(\bar \Omega)$ be an admissible solution of the boundary value problem \eqref{1.1}-\eqref{1.2} in a $C^{2,1}$  domain $\Omega\subset\mathbb{R}^n$, with boundary $\partial\Omega$ uniformly $(\Gamma,A,G)$-convex with respect to $u$. Assume that  F  satisfies conditions F1-F5, $A\in C^1(\bar \Omega\times \mathbb{R}\times \mathbb{R}^n)$, $B>a_0,\in C^1(\bar \Omega\times \mathbb{R}\times \mathbb{R}^n)$, $G\in C^2(\partial\Omega\times\mathbb{R}\times\mathbb{R}^n)$ is oblique with respect $u$ satisfying \eqref{1.3}, either F5\textsuperscript{+} holds or $B$ is independent of $p$ and either $\mathcal G$ is quasilinear satisfying \eqref{quasilinear} or  F6 holds. Then  for any $\epsilon >0$,
\begin{equation}\label{pure oblique estimate}
\sup\limits_{\partial\Omega} u_{\beta\beta}\le \epsilon M_2 +C_\epsilon(1+M^\prime_2),
\end{equation}
where $\beta=D_{p}G(\cdot,u,Du)$, and $C_\epsilon $ is a constant depending on $\epsilon, F, A, B, G, \Omega, \beta_0$ and $|u|_{1;\Omega}$. In the case when F6 holds, the estimate \eqref{pure oblique estimate} holds without the dependence on $M^\prime_2$.
\end{Lemma}

For any function $g\in C^2(\bar \Omega\times \mathbb{R}\times \mathbb{R}^n)$ and linearized operator $L$ in \eqref{linearized operators}, by calculation we have
\begin{equation}\label{formula for Lv}
\begin{array}{rl}
Lg(\cdot,u,Du) = \!\!&\!\! \displaystyle  F^{ij}(D_{p_kp_l}g) u_{ik}u_{jl} +  F^{ij}(D_{x_ix_j} - A_{ij}^kD_{x_k})g + (D_zg) Lu +  (D_{p_k}g) Lu_{k} \\
               \!\!&\!\! \displaystyle  +  F^{ij}[2(\tilde D_{x_i}D_{p_k}g)u_{jk} + 2(\tilde D_{x_i}D_zg) u_j - (D_{zz}g)u_iu_j],
\end{array}
\end{equation}
where
\begin{equation}\label{L u}
Lu = F^{ij}\delta_{ik}u_{jk} - F^{ij}A_{ij}^ku_k,
\end{equation}
and
\begin{equation}\label{L uk}
Lu_k = F^{ij}\tilde D_{x_k}A_{ij} + \tilde D_{x_k}B + (D_{p_l}B)u_{kl}.
\end{equation}
Plugging \eqref{L u}, \eqref{L uk} into \eqref{formula for Lv}, we obtain the differential inequality,
\begin{equation}\label{L g}
\begin{array}{ll}
Lg  \!\!&\!\! \displaystyle \ge  F^{ij}(D_{p_kp_l}g) u_{ik}u_{jl} - C \mathscr{T} + F^{ij}\tilde \beta_{ik}u_{jk} + (D_{p_k}g)(D_{p_l}B)u_{kl},
\end{array}
\end{equation}
 in $\Omega$, with
$$\tilde \beta_{ik}:= 2 \tilde D_{x_i}D_{p_k}g+(D_zg)\delta_{ik},$$
where the constant $C$ depends on $n$, $\Omega$, $|g|_{C^2}$, $|A|_{C^1}$, $|B|_{C^1}$ and $|u|_{1;\Omega}$. Note that  when $a_0$ is finite, \eqref{L u} can be estimated directly from \eqref{homogeneity}.

\begin{proof}[Proof of Lemma \ref{Lemma 2.2}.]
For any fixed boundary point $x_0\in \partial \Omega$, we consider the function
\begin{equation}\label{bar v}
\bar v=G(\cdot,u,Du) + \frac{a}{2} |Du-Du(x_0)|^2,
\end{equation}
where $G$ is the boundary function in \eqref{1.2}, and $a\le1$ is a positive constant. We consider the quasilinear case of $\mathcal G$, \eqref{quasilinear}, namely $G_{pp}=0$. We also consider the case when F5\textsuperscript{+} holds. By \eqref{L g}, Cauchy's inequality and \eqref{L uk}, we have
\begin{equation}\label{L bar v}
\begin{array}{ll}
L \bar v \!\!&\!\!\displaystyle \ge - C \mathscr{T}  + F^{ij}\tilde \beta_{ik}u_{jk} + (D_{p_k}G)(D_{p_l}B)u_{kl} + a F^{ij}u_{ik}u_{jk}+ a(u_k-u_k(x_0))Lu_k\\
              \!\!&\!\!\displaystyle \ge - C(1+a)\mathscr{T} - \frac{1}{a} F^{ij}\tilde \beta_{ik}\tilde \beta_{jk} + [D_{p_k}G+a(u_k-u_k(x_0))](D_{p_l}B)u_{kl} \\
              \!\!&\!\!\displaystyle \ge - [\frac{C}{a}+ (\epsilon_1 M_2+C_{\epsilon_1})]\mathscr{T},\quad\quad   {\rm in} \ \Omega,
\end{array}
\end{equation}
for any $\epsilon_1 >0$, where $g$ is replaced by $G$ in $\tilde \beta_{ik}$, F5\textsuperscript{+} is used in the last inequality, the constants $C$ depend on $n$, $\Omega$, $|G|_{C^2}$, $|A|_{C^1}$, $|B|_{C^1}$ and $|u|_{1;\Omega}$, and the constant $C_{\epsilon_1}$ depends on $\epsilon_1, F$ and $B$.

We shall construct a suitable upper barrier for $\bar v$ at the point $x_0$. We employ a function of the form
\begin{equation}\label{upper barrier}
\bar \phi = \phi + \frac{b}{2}|x-x_0|^2, \quad {\rm in}\ \Omega_{\rho},
\end{equation}
with
\begin{equation}
\phi= c(d-td^2),
\end{equation}
where $d=d(x)={\rm dist}(x,\partial\Omega)$, $\Omega_{\rho} = \{x\in \Omega | \ d(x)<\rho\}$, $b$, $c$, $t$ and $\rho$ are positive constants to be determined.
Then, by calculation, we have
\begin{equation}\label{L barrier 1}
\begin{array}{ll}
\displaystyle \frac{1}{2td-1}L \phi \!\!&\!\!\displaystyle = F^{ij}[c(-D_{ij}d+A^k_{ij}D_kd+\frac{2t}{1-2td}D_idD_jd)]\\
                 \!\!&\!\!\displaystyle \ge F^{ij}[c(-D_{ij}d+A^k_{ij}D_kd+2t D_idD_jd)], \quad\quad  {\rm in}\ \  \Omega_{\rho},
\end{array}
\end{equation}
provided $t\rho\le 1/4$. For convenience, we denote $h_{ij}=-D_{ij}d+A^k_{ij}D_kd+2t D_idD_jd$ in \eqref{L barrier 1}. By the uniform $(\Gamma, A, G)$-convexity of $\partial\Omega$, since $|u|$ and $|Du|$ are bounded in $\Omega$, there exists a small positive constant $\sigma$ such that
\begin{equation}
K_A[\partial\Omega](x,u,Du) + \mu_0 \nu(x)\otimes \nu(x) -2 \sigma I \in \Gamma,
\end{equation}
for all $x\in \partial\Omega$ satisfying $G(x,u(x),Du(x))\ge 0$. Reversing the projection onto the tangent space of $\partial\Omega$, we then have
\begin{equation}
-D\nu(x)+D_pA(x,u,Du)\cdot \nu(x) + \tilde \mu_0 \nu(x)\otimes \nu(x) -2\sigma I \in \Gamma,
\end{equation}
for all $x\in \partial\Omega$, for a  larger constant $\tilde \mu_0$, which implies $(h_{ij}-\sigma \delta_{ij})\in \Gamma$ for $x\in \partial\Omega$ provided $t\ge \tilde \mu_0$. By choosing $\rho$ sufficiently small and then $t$ sufficiently large, we have $(h_{ij}-\sigma \delta_{ij})\in \Gamma$ in
$\Omega_{\rho}^+ = \Omega_{\rho} \cap \{G(\cdot,u,Du) > 0\}$. Note that the constants $\rho$ and $t$ should be chosen under the restriction $t\rho\le 1/4$. Then the constants $\rho$ and $t$ are now fixed. Then, by the concavity F2, we have from \eqref{L barrier 1},
\begin{equation}\label{L barrier 2}
\begin{array}{ll}
\displaystyle \frac{1}{2t d-1}L \phi \!\!&\!\!\displaystyle \ge F^{ij}[c(h_{ij})]\\
   \!\!&\!\!\displaystyle = c\sigma\mathscr{T} + F^{ij}[c(h_{ij}-\sigma \delta_{ij})]\\
   \!\!&\!\!\displaystyle \ge c\sigma\mathscr{T} + F(c(h_{ij}-\sigma \delta_{ij}))-F(w_{ij})+F^{ij}w_{ij}\\
   \!\!&\!\!\displaystyle \ge c\sigma\mathscr{T} + F(c(h_{ij}-\sigma \delta_{ij}))-B(\cdot,u,Du),\quad\quad {\rm in}\ \  \Omega_{\rho}^+,
\end{array}
\end{equation}
where \eqref{1.1} and \eqref{homogeneity} are used in the last inequality. By using F4 with sufficiently large $c$, we have from \eqref{L barrier 2}
\begin{equation}\label{L barrier 3}
\begin{array}{ll}
\displaystyle L \phi \!\!&\!\!\displaystyle \le -\frac{1}{2}c\sigma\mathscr{T}, \quad\quad {\rm in}\ \  \Omega_{\rho}^+,
\end{array}
\end{equation}
where $2td-1\le -1/2$ in $\Omega_{\rho}^+$ is used.
Thus, we obtain, from \eqref{L bar v},
\begin{equation}\label{Lv ge Lphi}
L \bar \phi \le (-\frac{1}{2}c\sigma+Cb) \mathscr{T} \le  L \bar v, \quad\quad {\rm in}\ \  \Omega_{\rho}^+,
\end{equation}
provided $c\ge 2[C(b+\frac{1}{a}) + (\epsilon_1 M_2+C_{\epsilon_1})]/\sigma$.

Next, we examine $\bar v$ and $\bar \phi$ on the boundary of $\Omega_{\rho}$. For $x\in \partial \Omega$, we have
\begin{equation}
\begin{array}{ll}
|Du(x)-Du(x_0)| \!\!&\!\! \displaystyle \le C (\sup_{\partial\Omega}\sup_{|\tau|=1,\tau\cdot \nu=0} |D_\tau Du|) |x-x_0| \\
                \!\!&\!\! \displaystyle \le C (1+M^\prime_2)|x-x_0|,
\end{array}
\end{equation}
where the mixed derivative estimate \eqref{mixed derivative estimate} and the strict obliqueness \eqref{1.3} are used in the second inequality, so the constant $C$ depends also on $\beta_0$. For $x\in \Omega_{\rho}$ and $x^\prime$ the closest point on $\partial \Omega$, we then obtain,
\begin{equation}
\begin{array}{ll}
|Du(x)-Du(x_0)|^2 \!\!&\!\! \displaystyle \le 4(\sup|Du|)M_2d(x) + 2|Du(x^\prime)-Du(x_0)|^2\\
                \!\!&\!\! \displaystyle \le C (1+(M^\prime_2)^2 + M_2) (|x-x_0|^2 +d),
\end{array}
\end{equation}
so that
\begin{equation}\label{boundary inequality 1}
\bar v \le \frac{1}{2}a C (1+(M^\prime_2)^2 + M_2)(|x-x_0|^2 + d) \le \bar \phi, \quad {\rm on} \ \bar\Omega_{\rho}\cap\{G(\cdot,u,Du) = 0\},
\end{equation}
by choosing $b=aC (1+(M^\prime_2)^2 +M_2)$ and $c \ge b$. On the inner boundary, by choosing $c\ge C/\rho$, we have
\begin{equation}\label{boundary inequality 2}
\bar v \le \bar \phi, \quad {\rm on} \ \partial\Omega_{\rho} \cap \Omega.
\end{equation}
Now from \eqref{Lv ge Lphi}, \eqref{boundary inequality 1} and \eqref{boundary inequality 2}, by the comparison principle, we have
\begin{equation}
\bar v \le \bar \phi, \quad {\rm in} \ \Omega_{\rho}^+.
\end{equation}
Since $\bar v(x_0)= \bar \phi(x_0)=0$, we have
\begin{equation}
D_\beta \bar v(x_0) \le D_\beta \bar \phi(x_0),
\end{equation}
which implies
\begin{equation}\label{double oblique at x0}
u_{\beta\beta} (x_0) \le \beta^0 c+C,
\end{equation}
where $\beta^0:=\sup\limits_{\partial\Omega}(G_p(\cdot, u, Du)\cdot \nu)\ge \beta_0$.
We can fix the constant $c$ so that
\begin{equation}\label{choice of c}
\begin{array}{ll}
c \!\!&\!\! \displaystyle \le\frac{2[C(b+\frac{1}{a})+(\epsilon_1 M_2+C_{\epsilon_1}) ]}{\sigma} + b + \frac{C}{\rho}\\
  \!\!&\!\! \displaystyle \le C [(\epsilon_1+a)M_2+a(M^\prime_2)^2 +\frac{1}{a}] + C_{\epsilon_1},
\end{array}
\end{equation}
where $C$ now depends on $F, A, B, G, \Omega, \beta_0$ and $|u|_{1;\Omega}$, and $C_{\epsilon_1}$
depends additionally on $\epsilon_1$. For any $\epsilon >0$, taking $a=\frac{1}{1+\epsilon_1M_2}$
 and $\epsilon_1= \frac{\epsilon}{\beta^0C}$ for a further constant $C$ in \eqref{choice of c}, we then get
\begin{equation}\label{pure oblique at x0}
u_{\beta\beta} (x_0) \le \epsilon M_2 +C_\epsilon(1+M^\prime_2)
\end{equation}
from \eqref{double oblique at x0} and \eqref{choice of c}. Since $x_0$ is any boundary point, we can take the supremum of \eqref{pure oblique at x0} over $\partial\Omega$ to arrive at the desired estimate \eqref{pure oblique estimate}. Therefore, we have proved Lemma \ref{Lemma 2.2} in the case when $\mathcal{G}$ is quasilinear and F5\textsuperscript{+} holds. While in the case $\mathcal{G}$ is quasilinear and only F5 holds with $B$ independent in $p$, the last term in the second line of \eqref{L bar v} does not appear. So we still arrive at the same estimate \eqref{pure oblique estimate} and Lemma \ref{Lemma 2.2} is thus proved in the quasilinear case.

Next, we turn to the case that $F$ satisfies F6. Here we may simply take $a=0$ in \eqref{bar v}  and $b = 0$ in
\eqref{upper barrier}  so that from \eqref{L g}, F6 and Cauchy's inequality
\begin{equation}
\begin{array}{rl}
L \bar v \ge \!\!&\!\! \displaystyle  F^{ij}(D_{p_kp_l}G)u_{ik}u_{jl}- C \mathscr{T}   +  (D_{p_k}G)(D_{p_l}B)u_{kl} + 2  F^{ij}\tilde \beta_{ik}u_{jk}\\
                   \ge \!\!&\!\! \displaystyle  -(\epsilon_1M_2 + C_{\epsilon_1})\mathscr{T},\quad\quad   {\rm in} \ \Omega,
\end{array}
\end{equation}
for any $\epsilon_1 >0$, where $\epsilon_1$ now comes from the use of both  F5\textsuperscript{+} and F6, the constant $C$ depends on $n$, $\Omega$, $|G|_{C^2}$, $|A|_{C^1}$, $|B|_{C^1}$ and $|u|_{1;\Omega}$, and the constant $C_{\epsilon_1}$ depends also on ${\epsilon_1}$ and $F$.  We can then derive the desired estimate \eqref{pure oblique estimate}, without the dependence on $M^\prime_2$, for both F5\textsuperscript{+} and
$B$ independent of $p$.
\end{proof}

\begin{remark}\label{Remark 2.1}
The proof of Lemma \ref{Lemma 2.2} readily gives us  a local boundary estimate when we only assume $\partial\Omega\cap B$  is uniformly $(\Gamma,A,G)$-convex with respect to $u$ for some ball $B = B_R(x_0)$ of radius $R$, centred at $x_0$. Under the hypotheses of Lemma \ref{Lemma 2.2}, with $\Omega$ replaced by $\Omega\cap B$ and $\partial\Omega$ replaced by
$\partial\Omega\cap B$, we then obtain, in place of  \eqref{pure oblique estimate},
\begin{equation}\label{pure local oblique estimate}
 u_{\beta\beta}(x_0)\le \epsilon M_{2;\Omega\cap B} +C_\epsilon(1 + R^{-2}+M^\prime_{2;\partial\Omega\cap B})  ,
\end{equation}
where now $C_\epsilon $ is a constant depending on $\epsilon, F, A, B, G, \Omega, \beta_0, $ and $|u|_{1;\Omega\cap B}$
and
$$ M_{2;\Omega\cap B} = \sup\limits_{\Omega\cap B} |D^2 u|, \quad M^\prime_{2;\partial\Omega\cap B} = \sup\limits_{\partial\Omega\cap B} \sup\limits_{|\tau| = 1, \tau\cdot\nu=0} |u_{\tau\tau}|.$$
 \end{remark}

It now remains to estimate the pure tangential derivatives on the boundary. In this part, the strictly regular condition on the matrix $A$ is crucial.  We can formulate the pure tangential derivative estimates as follows.

\begin{Lemma}\label{Lemma 2.3}
Let $u\in C^2(\bar \Omega)\cap C^4(\Omega)$ be an admissible solution of the boundary value problem \eqref{1.1}-\eqref{1.2} in a $C^{2,1}$ domain $\Omega \subset \mathbb{R}^n$. Assume that $F, A$ and $B$ satisfy the hypothesis of Theorem \ref{Th1.1} and $G\in C^2(\partial\Omega\times \mathbb{R}\times \mathbb{R}^n)$ is oblique and concave in $p$ with respect to $u$ satisfying \eqref{1.3}, and either $\mathcal{G}$ is quasilinear or $\mathcal F$ is
orthogonally invariant or $F$ also satisfies condition F6. Then for any  tangential vector field $\tau$, $|\tau|\le 1$ and constant
$\epsilon > 0$, we have the estimate
\begin{equation}\label{pure tangential}
M_2^+(\tau) \le \epsilon M_2 + C_\epsilon,
\end{equation}
where $M_2^+(\tau)=\sup\limits_{\partial\Omega}u_{\tau\tau}$, and $C_\epsilon$ is a constant depending on $\epsilon, F, A, B, G, \Omega, \beta_0$ and $|u|_{1;\Omega}$.
\end{Lemma}

\begin{proof}
As usual we extend $\nu$ and $G$ smoothly to all of $\bar\Omega$ so that $\nu$ and $G(\cdot, z, p)$ are constant along normals to $\partial\Omega$ in some neighbourhood of $\partial\Omega$. Suppose that the function
\begin{equation}\label{2.50}
v_\tau = u_{\tau\tau} + \frac{c_1}{2} |u_\tau|^2
\end{equation}
takes a maximum over $\partial\Omega$ and tangential vectors $\tau$, such that $|\tau|\le 1$, at a point $x_0 \in \partial\Omega$ and vector $\tau=\tau_0$, where $c_1$ is the constant in the strict regularity condition \eqref{strict regularity}. Without loss of generality, we may assume $x_0=0$ and $\tau_0=e_1$. Setting
\begin{equation}\label{b}
b=\frac{\nu_1}{\beta\cdot \nu}, \quad\quad \tau=e_1 - b\beta,
\end{equation}
we then have, at any point in $\partial\Omega$,
\begin{equation}
v_1 = v_\tau + b(2u_{\beta\tau}+c_1u_\beta u_\tau)+ b^2 (u_{\beta\beta}+\frac{c_1}{2}u_\beta^2),
\end{equation}
with $v_1(0)=v_\tau (0)$, $b(0)=0$ and $\tau(0)=e_1$. From \eqref{tangen diff the bc}, $u_{\beta\tau}=-\tilde D_{x_\tau}G$ on $\partial\Omega$ so that setting
$$g=\frac{1}{\beta\cdot \nu} (2u_{\beta\tau}+c_1u_\beta u_\tau),$$
we have
$$|g-g(0)|\le C(1+M_2)|x|,\quad \quad {\rm on}\ \partial\Omega,$$
where $C$ is a constant depending on $G, \Omega, \beta_0$ and $|u|_{1;\Omega}$. Accordingly, there exists a further constant $C_1$ depending on the same quantities, such that the function,
\begin{equation}\label{2.51}
\tilde v_1 = v_1 - g(0)\nu_1 - C_1(1+ M_2)|x|^2
\end{equation}
satisfies
\begin{equation}\label{2.52}
\begin{array}{ll}
\tilde v_1 \!\!&\!\!\displaystyle \le |\tau|^2 \tilde v_1(0)\\
           \!\!&\!\!\displaystyle \le f\tilde v_1(0),\quad \quad {\rm on}\ \partial\Omega,
\end{array}
\end{equation}
where $f$ is any non-negative function in $C^2(\bar \Omega)$ satisfying $f\ge |\tau|^2$ on $\partial\Omega$, $f(0)=1$. In the case when $\mathcal{G}$ is quasilinear, that is $\beta=\beta(\cdot,u)$, we may simply estimate
\begin{equation} \label{tau}
\begin{array}{ll}
|\tau|^2   \!\!&\!\!\displaystyle \le 1-2b\beta_1 + b^2\sup |\beta|^2\\
           \!\!&\!\!\displaystyle \le 1-2\frac{\beta_1}{\beta\cdot \nu}(0)\nu_1 + C_1|x|^2 : = f.
\end{array}
\end{equation}
Now differentiating \eqref{1.2} twice in a tangential direction $\tau$, with $\tau(0)=e_1$, we obtain using the concavity of $G$ and \eqref{mixed derivative estimate},
\begin{equation}\label{2.56}
D_\beta u_{11}(0) \ge - C_1 (1+ M_2).
\end{equation}
Consequently for a sufficiently large constant $K$ depending on the same quantities as $C_1$, the function
\begin{equation}\label{2.57}
v= \tilde v_1 - f\tilde v_1(0) - K(1+ M_2)\phi
\end{equation}
must take an interior maximum in $\Omega$, where $\phi \in C^2(\bar \Omega)$ is a negative defining function for $\Omega$ satisfying $\phi=0$ on $\partial\Omega$, $D_\nu \phi=-1$ on $\partial\Omega$. This effectively reduces our argument to the proof of Theorem \ref{Th1.1}. In the proof of Theorem \ref{Th1.1}, by replacing the function $v_1$ in \eqref{2.8} with the function $v$ in \eqref{2.57}, we obtain at a maximum point $x_0$, using \eqref{strict regularity},
\begin{equation} \label {interior est}
u_{11}(x_0) \le C \sqrt{1+ M_2},
\end{equation}
where $C$ depends additionally on $A$ and $K|\phi|_{2;\Omega}$. The estimate \eqref{pure tangential} then follows by fixing $\phi$ and the constant $C_1$ in \eqref{tau} so that $\phi\ge -\frac{\epsilon}{4K}$ and $f\ge\frac{1}{2}$  in $\Omega$. Instead of adjusting $\phi$ we can alternatively maximize $v$ in a sufficiently small strip $\Omega_{\delta_0}$ around $\partial\Omega$ and apply the interior estimate in Theorem \ref{Th1.1}.

When $G$ is nonlinear in $p$, the coefficient $\beta_1$ in  the expansion \eqref {tau} of $|\tau|^2$ depends on $Du$ and cannot be controlled by the argument above. For orthogonally invariant $\mathcal F$, this is overcome by using a first order approximation to the tangent vector $e_1$ at $0$. Fixing the $x_n$ coordinate in the direction of $\nu$ at $0$, we then replace $e_1$ by the vector field
\begin{equation}\label{tang approx}
\xi = e_1 + \sum_{1\leq k< n}\delta_k\nu_1(0)(x_ne_k -x_ke_n).
\end{equation}
Then in place of \eqref{b}, we have
$$ b=\frac{\xi\cdot\nu}{\beta\cdot \nu}, \quad\quad \tau=\xi - b\beta,$$
so that, both $b(0)= 0$, $\delta b(0) = 0$. Accordingly we then have, in place of \eqref{2.51} and \eqref{2.52},
\begin{equation}\label{2.60}
\tilde v_1: = v_\xi  - C_1(1+ M_2)|x|^2 \le \tilde v_1(0)(1+C_1|x|^2),\quad \quad {\rm on}\ \partial\Omega,
\end{equation}
where $C_1$ again denotes constants depending on  $G, \Omega, \beta_0$ and $|u|_{1;\Omega}$. Comparing  the forms of 
$\tilde v_1$ in \eqref{2.51} and \eqref{2.60}, since $b(0)= 0$ and $\delta b(0) = 0$, here we can avoid the term $g(0)\nu_1$ in \eqref{2.60}. The inequality in \eqref{2.60} is obtained by estimating $|\tau|^2 = |\xi|^2-2b(\xi\cdot\beta)+b^2|\beta|^2$, 
with $|\xi|^2 \le 1 + 2 \delta_1\nu_1(0)x_n + \sum\limits_{1\le k<n}(\delta_k \nu_1(0))^2 |x|^2$ on $\partial\Omega$,  
$-2b(\xi\cdot\beta)+b^2|\beta|^2 \le C_1|x|^2$ on $\partial\Omega$, (since $b(0)=0$ and $\delta b(0)=0$), and using 
$\nu(0) = e_n$  to estimate $x_n$ on $\partial\Omega$.  Since the vector field $\xi$ has skew symmetric Jacobian $D\xi$, we can then reduce to the argument of the proof of  Lemma \ref{Lemma 2.1} when $\mathcal F$ is orthogonally invariant. 
The reduction is achieved by replacing the function $v$ in \eqref{2.11}, (where $u_{\tau\tau}$ is replaced by $(u_\tau)_\tau$ in the proof of Lemma \ref{Lemma 2.1}), with the function $v$ in \eqref{2.57}, (where $\tilde v_1$ is defined in \eqref{2.60}, and $f$ is any non-negative function in $C^2(\bar\Omega)$ satisfying $f\ge 1+ C_1|x|^2$ on $\partial\Omega$).
Combining \eqref{2.19} with \eqref{2.6} and \eqref{2.7},
we would then, following our argument above, arrive at the estimate \eqref{interior est} at an interior maximum point
$x_0$, with $e_1$ replaced by $\xi$  and from there get the corresponding estimate for $u_{\xi\xi}(0) =u_{11}(0)$ and
conclude \eqref{pure tangential}, as before.

Finally if F6 is satisfied,  we first obtain  from the formulae \eqref{formula for Lv}, \eqref{L u} and \eqref{L uk}, the complementary inequality to \eqref{L g}
\begin{equation}
\begin{array}{ll}
\mathcal L g \!\!&\!\!\displaystyle \le C(\mathcal E_2 +\mathscr T +1)\\
                       \!\!&\!\!\displaystyle \le C (\epsilon M_2 + 1) (1+\mathscr T),
\end{array}
\end{equation}
for any $\epsilon >0$ and provided $|D^2u| \ge C_\epsilon$, where the terms $F^{ij}\tilde \beta_{ik}u_{jk}+(D_{p_k}g)(D_{p_l}B)u_{kl}$ in $\mathcal L g$ are treated in the same way as in \eqref{L bar v}. Here the constant $C$ depends on $\Omega$, $|g|_{C^2}$, $|A|_{C^1}$, $|B|_{C^1}$ and $|u|_{1;\Omega}$ while $C_\epsilon$ depends on $\epsilon, F$ and $|u|_{1;\Omega}$. Now taking $f = f(\cdot,u,Du)$ satisfying $f \ge|\tau|^2$ on $\partial\Omega$, $f\ge \frac{1}{2},\in C^2(\bar \Omega\times\mathbb{R}\times \mathbb{R}^n)$, we then obtain at the interior maximum point $x_0$ of the function $v$ in \eqref{2.57}, corresponding to \eqref{2.9},
\begin{equation}
\mathcal{L}v \ge c_0\mathscr{T}|Du_1|^2 - C(1+\mathscr{T})(1+|Du_1| + \epsilon M^2_2)  + \lambda_B |Du_1|^2,
\end{equation}
if $|u_{11}(x_0)| \ge C_\epsilon$. A suitable function $f$ for example is obtained by taking $f = |\tau|^2 +C_1|x|^2 $ in
$\Omega$ with $\beta\cdot\nu$ in \eqref{b} replaced by $\zeta(\beta\cdot\nu)$ where $\zeta\in C^2(\mathbb{R})$ satisfies $\zeta^\prime, \zeta^{\prime\prime} \ge 0$, $\zeta(t) = t$ for $t\ge 3\beta_0/4$, $\zeta(t) = \beta_0/2$ for $t\le \beta_0/4$ and $C_1$ is a large enough constant, depending on $G, \Omega, \beta_0$ and $|u|_{1;\Omega}$,  to ensure $f\ge \frac{1}{2}$. Also using F6 in conjunction with \eqref{strict regularity}, we only need take $c_1= 0$ in \eqref{2.50}. By suitably adjusting $\epsilon$, we then infer the estimate \eqref{pure tangential}.
\end{proof}

With the local/global second derivative estimate in Theorem \ref{Th1.1}, and the boundary estimates in \eqref{mixed derivative estimate}, Lemma \ref{Lemma 2.2} and Lemma \ref{Lemma 2.3}, we are now ready to prove the global second derivative estimate \eqref{1.10} in Theorem \ref{Th1.2}.
\begin{proof}[Proof of Theorem \ref{Th1.2}]
In the quasilinear case we have by hypothesis that the sum of any $n-1$ eigenvalues of $w = M[u]$  is nonnegative which implies that the quantities $M^\prime_2$ and $M^+_2$ are equivalent. Combining the estimates \eqref{mixed derivative estimate}, \eqref{pure oblique estimate} and \eqref{pure tangential}, we then get an estimate
\begin{equation}
\sup_{\partial\Omega}u_{\xi\xi}\le \epsilon M_2 + C_\epsilon,
\end{equation}
for any constant unit vector $\xi$ and constant  $\epsilon >0$ where $C_\epsilon$ depends on $\epsilon,F, A, B, G, \Omega, \beta_0$ and $|u|_{1;\Omega}$. Then using the concavity of $F$, as in the proof of Theorem \ref{Th1.1}, or the above property of $w$, we get the full boundary estimate
\begin{equation}\label{full boundary}
\sup_{\partial\Omega}|D^2u| \le \epsilon M_2 + C_\epsilon
\end{equation}
and the desired estimate \eqref{1.10} follows from the global second derivative bound in Theorem \ref{Th1.1} with
$\epsilon$ chosen sufficiently small. If F6 is satisfied, the term in $M^\prime _2$ does not occur in the estimate  \eqref{pure oblique estimate} and we obtain \eqref{full boundary} directly from the concavity of $F$, as in the proof of Theorem \ref{Th1.1}. This completes the proof of Theorem \ref{Th1.2}.
\end{proof}

\begin{remark}\label{Remark2.1}
Using the equivalence of $M^\prime_2$ and $M^+_2$ as above,  we may replace $M_2$ by $u_{11}(0)$ in \eqref{2.56} and \eqref{2.57}. Taking $\epsilon = 1$ in our adjustment of $\phi$ in the proof of Lemma \ref{Lemma 2.3}, we then obtain, in the cases when $\mathcal G$ is quasilinear or $\mathcal F$ is orthogonally invariant, a more precise version of the tangential estimate \eqref{pure tangential}
\begin{equation}
M_2^+(\tau) \le C(1+\sqrt{M_2}),
\end{equation}
where $C$ is a constant depending on $ F, A, B, G, \Omega, \beta_0$ and $|u|_{1;\Omega}$.
\end{remark}

\section{Gradient estimates}\label{Section 3}
\vskip10pt

In this section, we prove various gradient estimates for admissible solutions $u$ of the oblique problem \eqref{1.1}-\eqref{1.2}.
We mainly consider the case when the oblique boundary operator $\mathcal{G}$ is semilinear and in particular give the proof of Theorem \ref{Th1.3}. We also derive corresponding local gradient estimates as well as an estimate for nonlinear $\mathcal{G}$. As mentioned in the introduction, our conditions on either the matrix $A$ or the operator $F$ enable an analogue of  uniform ellipticity. Accordingly we will employ improvements of the methods for uniformly elliptic equations in \cite{LieTru1986} with a critical adjustment used to supplement the tangential gradient terms in \cite{LieTru1986},  which is similar to that used for gradient estimates in the conformal geometry case in \cite{JLL2007}. We also prove a H\"older estimate
for $\Gamma = \Gamma_k$ for $k >n/2$, from which we infer gradient estimates under natural quadratic growth conditions.

\subsection{Global gradient estimates}\label{subsection3.1}

We consider the case of oblique semilinear $\mathcal{G}$ in \eqref {semilinear} and normalise $G$ by dividing by $\beta\cdot \nu \ge \beta_0$, so that we can write
\begin{equation}\label{semilinear Section 4}
G(x,z,p) = \nu\cdot p  + \beta^\prime\cdot p^\prime  - \varphi (x,z) ,
\end{equation}
where now $\beta^\prime \in C^2(\partial\Omega)$, $\varphi \in C^2(\partial\Omega\times\mathbb{R})$ and $p^\prime = p - (p\cdot \nu)\nu$. For convenience, we still use $\varphi$ to denote its normalised form here. The boundary condition \eqref{1.2} can thus be written in the form
\begin{equation}\label{bc}
\mathcal G[u] = G(\cdot,u,Du) = D_\nu u  + \beta^\prime\cdot \delta u -\varphi (\cdot,u) = 0, \quad \mbox{on} \ \partial\Omega.
\end{equation}

We begin with a preliminary calculation to estimate  $\mathcal{L}g$ from below, for $g$ given by
\begin{equation}\label{quadratic g}
g(x,u,Du)= a_{kl}(x)u_ku_l +  b_k(x,u)u_k + c(x,u),
\end{equation}
where $\{a_{kl}\}$ is a nonnegative  matrix function on $\Omega$, $b(x,u)=(b_1(x,u),\cdots,b_n(x,u)) $ is a vector valued function on $\Omega \times \mathbb{R}$, $c(x,u)$ is a scalar function on $\Omega\times \mathbb{R}$, $a_{kl}\in C^2(\bar\Omega)$, $b_k, c \in C^2(\bar\Omega\times \mathbb{R})$, and $u\in C^3(\Omega)$ is an admissible solution of \eqref{1.1}. For this estimation we may assume more general growth conditions than \eqref{quadratic growth}, namely
\begin{equation} \label{general growth}
D_xA, D_xB, D_zA, D_zB = o(|p|^3), \quad D_zA\ge o(|p|^2)I, D_zB\ge o(|p|^2),\quad D_pA, D_pB = o(|p|^2),
\end{equation}
as $|p| \rightarrow \infty$, uniformly for $x\in\Omega$, $|z|\le M$ for any $M > 0$.

Now we calculate, using \eqref{formula for Lv},
\begin{equation}\label{mathcal L g}
\begin{array}{rl}
\mathcal{L}g = \!\!&\!\!\displaystyle  2F^{ij}a_{kl}u_{ik}u_{jl} + u_ku_l\mathcal{L}a_{kl} + u_l [F^{ij}D_{x_ix_j}b_l-(F^{ij}A_{ij}^k+D_{p_k}B)D_{x_k}b_l]  \\
               \!\!&\!\!\displaystyle   + [F^{ij}D_{x_ix_j}c-(F^{ij}A_{ij}^k+D_{p_k}B)D_{x_k}c]+ (u_kD_zb_k + D_zc)\mathcal{L}u \\
               \!\!&\!\!\displaystyle  + [(2a_{kl}u_l+b_k) \mathcal{L}u_k+ 2F^{ij}(2u_l D_ia_{kl}+ \tilde D_{x_i}b_k)u_{jk}] \\
               \!\!&\!\!\displaystyle + F^{ij}[2(u_k\tilde D_{x_i}D_zb_k + \tilde D_{x_i}D_zc)u_j - (u_kD_{zz}b_k +D_{zz}c)u_iu_j],
\end{array}
\end{equation}
where $\tilde D_{x_i} = D_{x_i} +u_iD_z$ and, corresponding to \eqref{L u} and \eqref{L uk},
\begin{equation}\label{mathcal L u}
\mathcal Lu = F^{ij}\delta_{ik}u_{jk} - (F^{ij}A_{ij}^k + D_{p_k}B)u_k
\end{equation}
and
\begin{equation}\label{mathcal Luk}
\mathcal Lu_k = F^{ij}\tilde D_{x_k}A_{ij} + \tilde D_{x_k}B.
\end{equation}
From the growth conditions \eqref{general growth}, we then estimate
\begin{equation}\label{3.9}
\mathcal{L}g\ge 2F^{ij}a_{kl}u_{ik}u_{jl} + F^{ij}\tilde \beta_{ik}u_{jk} - C(\mathscr{T} +1)(\omega(|Du|)|Du|^4 +1),
\end{equation}
where
$$\tilde \beta_{ik}= 2(2u_l D_ia_{kl}+ \tilde D_{x_i}b_k) +(u_l D_zb_l+ D_zc)\delta_{ik}$$
and $\omega$ is a positive decreasing function on $[0,\infty)$ tending to $0$ at infinity, depending on $A,B$ and
$M_0 =\sup_\Omega|u|$,  and $C$ is a constant depending on $a_{kl}, b_k, c, \Omega, A, B$ and $M_0$. For our approach here we will assume $\{a_{kl}\}$ is positive definite so that $\{a_{kl}\}\ge a_1 I$ for some positive constant $a_1$. By Cauchy's inequality and the positivity of $\{a_{kl}\}$, we then obtain
\begin{equation} \label{g estimate}
\mathcal{L}g\ge a_1 \mathcal E_2^\prime  - C(\mathscr{T} +1)(\omega|Du|^4 +1),
\end{equation}
where $\mathcal{E}_2^\prime = F^{ij}u_{ik}u_{jk}$.

In our proof of Theorem \ref{Th1.3}, we will specifically choose
\begin{equation}\label{g}
g = |\delta u|^2 + |\mathcal G[u]|^2,
\end{equation}
where $\nu$ and $G$ given by \eqref{semilinear Section 4} are appropriately extended to all of $\bar \Omega$. For our purposes a convenient extension
will be as usual to take $\nu$ constant along normals to $\partial \Omega$ in a sufficiently small strip $\{d<\rho\}$ where the distance function $d$ is $C^2$ smooth and to extend
$\beta^\prime \in C^2(\bar\Omega)$ and  $\varphi \in C^2(\bar\Omega\times\mathbb{R})$
to vanish for $d\ge \rho$ so that $g=|Du|^2$ for $d\ge\rho$. Using Cauchy's inequality, we can estimate  here
\begin {equation}
\{a_{kl}\} \ge (1+\beta^\prime_0)^{-2} I
\end{equation}
 where $\beta^\prime_0 = \sup|\beta^\prime|$ so that \eqref{g estimate} holds with $a_1 = (1+\beta^\prime_0)^{-2}$ and $C$ depending on $\beta^\prime, \varphi, \Omega, A, B$ and $M_0$, with $\Omega \in C^3$. By further use of Cauchy's inequality, we also obtain
\begin{equation}\label{3.12}
\frac{a_1}{4} |Du|^2 - |\varphi|^2 \le g \le 2 |Du|^2 + \frac{2}{a_1} |\varphi|^2,
\end{equation}
so that the estimation of $Du$ is equivalent to that of the function $g$.

With these preparations, we give the proofs of the global gradient estimates.
\vskip8pt
\begin{proof}[Proof of Theorem \ref{Th1.3}.]
We employ auxiliary functions of the form
\begin{equation}\label{3.13}
v:= g + M^2_1 (\alpha\eta -\kappa \phi),
\end{equation}
in $\Omega$, where $g$ is given by \eqref{g},  $\eta=e^{+(-)K(u_0-u)}$, for a constant $u_0$ and a positive constant $K$ to be determined, $\phi\in C^2(\bar\Omega)$ is a negative defining function for $\Omega$ satisfying $\phi=0$ on $\partial\Omega$ and $D_\nu \phi =-1$ on $\partial \Omega$, $M_1=\sup\limits_\Omega |Du|$,  $\alpha$ and $\kappa$ are positive constants to be determined.
We assume the function $v$ attains its maximum over $\bar \Omega$ at some point $x_0$. If $x_0 \in \partial\Omega$, then we have $D_{\beta}v(x_0) \le 0$, where $\beta= \nu + \beta^\prime$. From the construction of $g$, we have
on $\partial\Omega$,
\begin{equation}\label{g Dg property on boundary}
g= |\delta u|^2 \quad {\rm and}\ \ Dg = D (|\delta u|^2) .
\end{equation}
Then we have on $\partial\Omega$,
\begin{equation}\label{3.15}
\begin{array}{ll}
D_\beta v \!\!&\!\!\displaystyle =  D_{\beta}|\delta u|^2 + M^2_1(\alpha D_\beta \eta - \kappa D_\beta \phi) \\
          \!\!&\!\!\displaystyle \ge 2 \delta_k u D_\beta\delta_k u + \kappa M^2_1 -  \alpha K\eta |\varphi|(\cdot,u)M^2_1,
          \end{array}
\end{equation}
using \eqref{g Dg property on boundary} and the boundary condition \eqref{bc}. Next by
 tangential differentiation of \eqref{bc}, as in \eqref {tangen diff the bc},  we have
\begin{equation}\label{3.16}
D_\beta \delta_k u =  \delta_{x_k} \varphi  + (D_z\varphi)\delta_k u  -  (\delta_k\beta_i) D_i u  -  \beta_i\nu_k(D_i\nu_l)D_lu - \beta_i (D_i\nu_k)D_\nu u.
\end{equation}
Plugging \eqref{3.16} into \eqref{3.15}, by Cauchy's inequality and the fact that $|\delta u|\le |Du|$,
we then have at $x_0$,
\begin{equation}\label{3.17}
\begin{array}{ll}
D_\beta v \!\!&\!\!\displaystyle \ge - C(1+|Du|^2) +  M^2_1[\kappa -\alpha K\eta|\varphi|(\cdot,u)] \\

          \!\!&\!\!\displaystyle \ge \frac{1}{2}\kappa M^2_1 - C,
\end{array}
\end{equation}
provided $\kappa$ is chosen large enough so that $\kappa\ge C(1+\alpha K\sup\limits_{\partial\Omega} \eta) $, where $C$ is a constant depending on $\beta,\varphi,\Omega$ and  $M_0$. With the constant $\kappa$  fixed, from \eqref{3.17} and $D_\beta v(x_0)\le 0$, we obtain
\begin{equation}\label{regular A conclusion 1}
M^2_1 \le \frac{2C}{\kappa }.
\end{equation}

We next consider the case that the maximum of $v$ occurs at a point $x_0\in\Omega$.
We now take the constant $\alpha$  sufficiently small and fix the defining function $\phi$ such that
\begin{equation}\label{key restrictions}
\alpha\le a_1/16\mathop{\rm osc}\limits_\Omega \eta, \quad\quad
\phi\ge -a_1/16\kappa \quad {\rm in} \  \Omega.
\end{equation}
Taking \eqref{3.12} into account, these restrictions in \eqref{key restrictions} will enable us to proceed from estimating $Du(x_0)$ to an estimate for $M_1$. Note that these conditions ultimately depend on the independent choice of the constant $K$. Since $Dv(x_0)=0$ and $D^2v(x_0) \le 0$, we have
\begin{equation}\label{4.11}
\mathcal{L}v =F^{ij}D_{ij}v - (F^{ij}A^k_{ij}+ D_{p_k}B)D_k v = F^{ij}D_{ij}v \le 0, \quad {\rm at} \ x_0,
\end{equation}
where $\mathcal{L}$ is the linearized operator defined in \eqref{linearized operators}. Our estimations then reduce to getting an appropriate lower bound for $\mathcal L \eta$ and for this we separate cases (i) and (ii) in Theorem \ref{Th1.3}.

\vskip8pt

Case (i): $A$ uniformly regular.

Here we take the ``$+$'' sign in $\eta$,  that is  $\eta=e^{K(u_0-u)}$, and for convenience set
$u_0 = M^+_0 =\sup\limits_\Omega u$ so that $\eta\ge 1$.  Our estimation of $L\eta$ is motivated by the barrier constructions in \cite{JT2014, JT-oblique-II, JTY2013, JTY2014} for regular $A$, where the constant  $u_0$ is replaced by an admissible function $\bar u$. In particular the reader is referred to Section 2 in \cite{JT-oblique-II} for the extension to general operators. First by Taylor's formula, we have
\begin{equation}\label{3.20}
\begin{array}{rl}
\mathcal{L}(u_0-u) = \!\! & \!\!\displaystyle  F^{ij}[-w_{ij} -A_{ij}(\cdot,u,0) \\
              \!\! & \!\!\displaystyle  + A_{ij}(\cdot,u,0) - A_{ij}(\cdot,u,Du) + A_{ij}^k(\cdot,u,Du)u_k] + (D_{p_k}B)u_k\\
            = \!\! & \!\!\displaystyle  \frac{1}{2}F^{ij}A_{ij}^{kl}(\cdot,u,\hat p)u_ku_l - F^{ij}( w_{ij} +A_{ij}(\cdot,u,0)) + (D_{p_k}B)u_k,
\end{array}
\end{equation}
where $\hat p=\theta Du$ with $\theta\in (0,1)$. Since $A$ is uniformly regular, we can estimate
\begin{equation}\label{3.21}
\frac{1}{2}F^{ij}A_{ij}^{kl}(\cdot,u,\hat p)u_ku_l \ge \frac{\lambda_0}{2} |Du|^2 \mathscr{T} - \frac{\bar \lambda_0}{2} F^{ij}u_iu_j.
\end{equation}
From \eqref{3.20}, \eqref{3.21}, we have
\begin{equation}\label{3.22}
\begin{array}{rl}
\mathcal{L}\eta \!\!&\!\!\displaystyle = K\eta[\mathcal{L}(u_0-u) + KF^{ij}u_iu_j] \\
                \!\!&\!\!\displaystyle \ge K\eta[\frac{\lambda_0}{2} |Du|^2 \mathscr{T} + (K-\frac{\bar \lambda_0}{2})F^{ij}u_iu_j
  - F^{ij}( w_{ij} +A_{ij}(\cdot,u,0))  +  (D_{p_k}B)u_k] \\
                \!\!&\!\!\displaystyle \ge K\eta[\frac{\lambda_0}{2} |Du|^2 \mathscr{T} +\frac{K}{2}F^{ij}u_iu_j- F^{ij} w_{ij} -\mu \mathscr{T} +  (D_{p_k}B)u_k],
\end{array}
\end{equation}
by choosing $K\ge \bar\lambda_0$, where $\mu$ is a constant depending on $A$. At this point we introduce a more general condition than the concavity F2 which also includes the homogeneous case. Namely we assume for any constant $a > a_0$, there exist non-negative constants $\mu_0$ and $\mu_1$ such that
\begin{equation} \label{3.23}
r\cdot F_r \le F + \mu_0 + \mu_1\mathscr{T}
\end{equation}
whenever $F \ge a$. From \eqref{1.9}, we see that  F2 implies \eqref{3.23},
 with $\mu_0 = \max\{0, - F(\mu I)\}$ and $\mu_1 = \mu$ for any $\mu >0$. Note that when
$a_0 > -\infty$, then \eqref{3.23}, with  $\mu_0 =\max\{0, - a_0\}$ and $\mu_1 = 0$, is immediate from \eqref{homogeneity}.
Using \eqref {3.23} in \eqref{3.22} we thus obtain
\begin{equation}\label{3.24}
\mathcal{L}\eta \ge K\eta[\frac{\lambda_0}{2} |Du|^2 \mathscr{T} + \frac{K}{2}F^{ij}u_iu_j- \mu (1+\mathscr{T}) - B +(D_{p_k}B)u_k],
\end{equation}
where now $\mu$ depends on $A$, $\mu_0$ and $\mu_1$.

Assuming now F5, with $b = \infty$, so that
$\mathscr{T} \ge \delta_0$ for $B\ge b_0 >a_0$ and supplementing the growth conditions \eqref{general growth}
by
\begin{equation}\label{3.25}
B-p\cdot B_p \le o(|p|)^2,
\end{equation}
we then have from \eqref{3.24}, with $\omega$ sufficiently small,
\begin{equation}\label{3.26}
\begin{array}{rl}
\mathcal{L}\eta \ge \!\!&\!\!\displaystyle K\eta\mathscr{T}[\frac{\lambda_0}{2} |Du|^2  -\omega|Du|^2 -C]\\
\ge \!\!&\!\!\displaystyle \frac{K\lambda_0}{4}\mathscr{T}|Du|^2
\end{array}
\end{equation}
provided $|Du|\ge C_1$ for some sufficiently large constant $C_1$, depending on $F$, $A$, $B$ and $M_0$.
Combining \eqref{g estimate}, \eqref{3.26}, and also choosing $K \ge \frac{4}{\lambda_0}$, we then obtain
\begin{equation}\label{3.27}
\mathcal{L}v \ge  \alpha M_1^2 |Du|^2\mathscr{T} - C\mathscr{T} (\omega|Du|^4 +\kappa M_1^2
\omega |Du|^2 + 1).
\end{equation}
On the other hand if F5\textsuperscript{+} is satisfied, with $b=\infty$, and ``$o$'' is weakened to ``$O$'' in \eqref{3.25}, that is
\begin{equation}\label{3.28}
B-p\cdot B_p \le O(|p|)^2,
\end{equation}
then we have from \eqref{3.24},
\begin{equation}\label{3.29}
\mathcal{L}\eta \ge K\eta[\frac{\lambda_0}{2} |Du|^2 \mathscr{T} + \frac{K}{2}F^{ij}u_iu_j -C(1+|Du|^2)],
\end{equation}
and we arrive again at the inequality \eqref{3.27}, with large enough $K$, by using the positivity of $F_r$ to estimate
 $$ F^{ij}u_iu_j \ge \delta_1|Du|^2$$
 for some positive constant $\delta_1$ when $|r| \le C$ and $F(r)\ge b_0$.

Now since  $v$ is maximised at $x_0$ and  $\omega$ can be made arbitrarily small for large enough $|Du|$, we must have $|Du(x_0)|\le C$, noting that $\alpha$ and $\kappa$ are now fixed by our choice of $K$ above. From
\eqref{3.12}, \eqref{3.13} and our restrictions on $\alpha$ and $\phi$, we obtain
$$v\le C + \frac{\alpha}{8} M^2_1, \quad \mbox{in} \ \Omega,$$
so that again using \eqref{3.12}, we conclude $M_1 \le C$, where $C$ now depends on $F, A, B, \Omega, b_0, \beta,
\varphi$ and $M_0$. This completes the proof of case (i) in Theorem \ref{Th1.3}, with conditions \eqref{quadratic growth} and F2 weakened to \eqref{general growth},  \eqref{3.23} and \eqref{3.28}.

\vskip8pt

Case (ii): F7 holds, $\beta=\nu$.

We take the ``$-$'' sign in $\eta$ so that $\eta=e^{K(u-u_0)}$ and  set $u_0 = M^-_0 = \inf \limits_\Omega u$.
Assuming growth conditions,
\begin{equation}\label{3.30}
p\cdot A_p \le O(|p|^2) I, \quad  p\cdot B_p \le O(|p|^2),
\end{equation}
as $|p| \rightarrow \infty$, uniformly for $x\in\Omega$, $|z|\le M$ for any $M > 0$,
we then have from \eqref{mathcal L u}, in place of \eqref{3.22},
\begin{equation}\label{3.31}
\begin{array}{rl}
\mathcal{L}\eta \!\!&\!\!\displaystyle = K\eta[\mathcal{L}(u-u_0) + KF^{ij}u_iu_j] \\
                \!\!&\!\!\displaystyle =  K\eta[ KF^{ij}u_iu_j +F^{ij}u_{ij} - (F^{ij}A_{ij}^k + B_{p_k})u_k] \\
                \!\!&\!\!\displaystyle \ge K\eta[ KF^{ij}u_iu_j +F^{ij}u_{ij} - \mu(\mathscr{T} + 1)(1+|Du|^2)],
\end{array}
\end{equation}
where $\mu$ is a constant depending on $M_0$.  At this stage, anticipating our use of F7, we can fix the constant $K$
so that $K \ge 2n(1+2\mu)/\min\{\delta_0,\delta_1\}$. In order to get a lower bound for $\mathcal{L}\eta$, similar to
\eqref{3.26}, at the point $x_0\in\Omega$ where $v$, given by  \eqref{3.13}, is maximised, we also need to impose our key restrictions in the hypotheses of case (ii) of Theorem \ref{Th1.3}, namely $\mathcal F$ is orthogonally invariant,
$\beta = \nu$, that is $\beta^\prime = 0$,  and
\begin{equation}\label{3.32}
A = o(|p|^2),
\end{equation}
as $|p| \rightarrow \infty$, uniformly for $x\in\Omega$, $|z|\le M$ for any $M > 0$.

By choosing coordinates, we can assume the augmented Hessian $M[u]=D^2u-A=\{w_{ij}\}$ is diagonal at $x_0$. Then we have, at $x_0$,
\begin{equation}\label{diagonal equality}
u_{ii} = w_{ii}+ A_{ii},\quad \forall\ i,\quad \quad{\rm and}\quad\quad  u_{ij}=A_{ij}, \quad {\rm for} \ i\neq j.
\end{equation}
From now on, all the calculations will be made at the maximum point $x_0$.
Since $Dv(x_0)=0$, we have
\begin{equation}\label{D i H = 0 F7}
v_i  = g_i + M^2_1(\alpha K\eta u_i  - \kappa \phi_i ) =0, \ \ {\rm for}\ i=1,\cdots, n,
\end{equation}
that is
\begin{equation}\label{D i v 1}
g_i= M^2_1(- \alpha K\eta u_i + \kappa \phi_i ), \ \ {\rm for}\ i=1,\cdots, n.
\end{equation}
On the other hand, we have
\begin{equation}\label{D i v 2}
\begin{array}{ll}
g_i \!\!&\!\!\displaystyle =D_i [|\delta u|^2 + (D_\nu u - \varphi)^2] \\
     \!\!&\!\!\displaystyle =D_i [|Du|^2 -2\varphi D_\nu u  + \varphi^2] \\
     \!\!&\!\!\displaystyle = 2 (u_k - \varphi \nu_k)u_{ki} - 2\varphi u_kD_{i}\nu_k + 2(\varphi - D_\nu u)\tilde D_{x_i}\varphi,
\end{array}
\end{equation}
for $i=1,\cdots,n$, where $\tilde D_x = D_x +DuD_z$ and in accordance with our extension of $G$ to $\Omega$,
$\varphi = 0$ for $d \ge \rho$. Combining \eqref{D i v 1}, \eqref{D i v 2} and \eqref{diagonal equality}, we have
\begin{equation}\label{before largest Du}
\begin{array}{rl}
(u_i -\varphi \nu_i)w_{ii}
    = \!\!&\!\!\displaystyle  \frac{1}{2}M^2_1(-\alpha K \eta u_i + \kappa \phi_i) + \varphi  u_k D_i\nu_k -(\varphi - D_\nu u)\tilde D_{x_i} \varphi \\
      \!\!&\!\!\displaystyle -(u_i - \varphi \nu_i)A_{ii}-\sum_{k\neq i}(u_k -\varphi \nu_k)A_{ki},
\end{array}
\end{equation}
for $i=1,\cdots,n$. Without loss of generality, we can further choose our coordinates so that
\begin{equation}\label{largest Du}
u_1(x_0) \ge \frac{1}{\sqrt n}|Du(x_0)|.
\end{equation}
By assuming
\begin{equation}\label{u1aass1}
u_1(x_0)>2\max \{\sup\limits_\Omega |\varphi |, \kappa \sup\limits_\Omega (|D\phi|/\alpha K)\},
\end{equation}
we also have
\begin{equation}\label{u1aass2}
\frac{u_1}{u_1-\varphi\nu_1}\ge \frac{2}{3},
\end{equation}
at $x_0$. From \eqref{before largest Du}, \eqref{u1aass1} and \eqref{u1aass2}, we then obtain
\begin{equation}\label{push a negative w11}
\begin{array}{rl}
w_{11} = \!\!&\!\!\displaystyle  \frac{1}{u_1 -\varphi \nu_1} [\frac{1}{2}M^2_1(-\alpha K \eta u_1 + \kappa \phi_1) + \varphi  u_k D_1\nu_k \\
         \!\!&\!\!\displaystyle  -(\varphi - D_\nu u)\tilde D_{x_1} \varphi -\sum_{k> 1}(u_k -\varphi \nu_k)A_{k1}] - A_{11}\\
      \le\!\!&\!\!\displaystyle  \frac{1}{u_1 -\varphi \nu_1} [- \frac{1}{4}\alpha K \eta M^2_1 u_1 + |\varphi u_k D_1\nu_k| \\
         \!\!&\!\!\displaystyle  + |(\varphi - D_\nu u)\tilde D_{x_1} \varphi|  + |\sum_{k>1}(u_k -\varphi \nu_k)A_{k1}|] +  |A_{11}|\\
      \le\!\!&\!\!\displaystyle  -\frac{1}{6}\alpha K \eta M^2_1 + C (|\varphi D_1\nu_k| + |\tilde D_{x_1}\varphi| + \sum_{k>1}|A_{k1}|) + |A_{11}|\\
      \le\!\!&\!\!\displaystyle  -\frac{1}{6}\alpha K M^2_1 + C(\omega|Du|^2+1) ,
\end{array}
\end{equation}
at $x_0$, where $\omega = \omega(|Du|)$ approaches $0$ as $|Du|\rightarrow \infty$ and the growth condition \eqref{3.30} is used in the last inequality. It then follows that $w_{11}(x_0) < 0$ provided $|Du(x_0)| \ge C_1$ for some constant $C_1$,  depending on $F, A, B, \Omega, \varphi$ and $M_0$. Since $w_{11}$ is the eigenvalue of $M[u]$ corresponding to the eigenvector $e_1$ and the matrix $F_r$ is  diagonal at $M[u]$, by virtue of the orthogonal invariance of $\mathcal F$, we then obtain from F7 and \eqref{largest Du} that
\begin{equation}\label{3.40}
F^{ij}u_iu_j \ge F^{11}u^2_1 \ge \frac{1}{n}(\delta_0 +\delta_1\mathscr{T})|Du|^2,
\end{equation}
at $x_0$. From our choice of $K$ we then obtain, from \eqref{3.31} and \eqref{3.40},
\begin{equation}\label{3.41}
\mathcal{L}\eta \ge K \eta [2(1+\mathscr{T})|Du|^2 + F^{ij}u_{ij}],
\end{equation}
at $x_0$. Note that if $F^{ij}w_{ij} \ge 0$, as in \eqref{homogeneity}, then we can absorb the term $F^{ij}u_{ij}$ in the last term in \eqref{3.31} so that it is not present in \eqref{3.41}. Furthermore if $F$ is positive homogeneous of degree one, we can replace $p\cdot B_p$ in \eqref{3.30} by $p\cdot B_p - B$.

Assuming also the growth conditions \eqref{general growth}, we now combine \eqref{3.41} with \eqref{g estimate}, with $a_1=1$, and in general, (when $F^{ij}w_{ij}$ may be unbounded from below), use Cauchy's inequality to control the term $F^{ij}u_{ij}$. Accordingly  we obtain, at the maximum point $x_0$,
\begin{equation}\label{combining together}
\begin{array}{rl}
0 \ge \mathcal{L}v \ge \!\!&\!\!\displaystyle   \alpha M^2_1K[2(1+\mathscr{T})|Du|^2 -
\alpha M^2_1K\eta\mathscr{T}] \\
                   \!\!&\!\!\displaystyle  - C(1+\mathscr{T}) (\omega|Du|^4 +\kappa M_1^2
\omega |Du|^2 + 1) \\
                   \ge \!\!&\!\!\displaystyle   (1+\mathscr{T})[ \alpha M^2_1|Du|^2 - C(\omega|Du|^4 +\kappa M_1^2
\omega |Du|^2 + 1)],
\end{array}
\end{equation}
provided, taking account of \eqref{3.12} and \eqref{3.13}, we further restrict $\alpha$  so that $\alpha K\eta \le \frac{1}{16}$. As in case (i), since $\omega$ can be made arbitrarily small for large enough $|Du|$  and, $\alpha$ and $\kappa$ are fixed by our choice of $K$ above, we obtain an estimate $|Du(x_0)| \le C$ and hence $M_1 \le C$, where $C$ depends on $F, A, B, \Omega, b_0, \varphi$ and $M_0$. This completes the proof of case (ii) in Theorem \ref{Th1.3}, with conditions \eqref{quadratic growth} weakened to \eqref{general growth} and \eqref{3.30}.
\end{proof}

\begin{remark}\label{Remark3.1}
When $\beta = \nu$, so that $a_{ij} = \delta_{ij}$ in \eqref{quadratic g}, we can further weaken the general growth conditions \eqref{general growth}, in both cases of the above proof, to
\begin{equation}\label{more general growth}
\begin{array}{c}
p\cdot D_xA +|p|^2 D_zA\ge o(|p|^4)I, \quad p\cdot D_xB +|p|^2 D_zB\ge o(|p|^4), \\
D_xA, D_xB  = o(|p|^4), \quad D_zA, D_zB = o(|p|^3), \quad D_pA, D_pB = o(|p|^3),
\end{array}
\end{equation}
as $|p| \rightarrow \infty$, uniformly for $x\in\Omega$, $|z|\le M$ for any $M > 0$, while if also $\varphi = 0$ so that
$g=|Du|^2$, we need only assume, in place of \eqref{general growth},
\begin{equation}\label{3.44}
p\cdot D_xA +|p|^2D_zA \ge o(|p|^4) I, \quad p\cdot D_xB +|p|^2D_zB  \ge o(|p|^4).
\end{equation}
Discarding the boundary condition \eqref{bc} so that $g=|Du|^2$ and $\kappa = 0$ in \eqref{3.13}, we then have a global gradient estimate for admissible solutions $u\in C^3(\Omega)\cap C^1(\bar\Omega)$ in terms of the gradient on the boundary,
namely
\begin{equation}\label{3.45}
\sup\limits_\Omega |Du|\le C(1+ \sup\limits_{\partial\Omega} |Du|),
\end{equation}
where $C$ is a constant depending on $F, A, B, \Omega,b_0$ and $|u|_{0;\Omega}$. The estimate \eqref{3.45} thus holds when $F$, $A$ and $B$ satisfy the hypotheses of Theorem \ref{Th1.3}, but more generally we can replace \eqref {quadratic growth} by \eqref{3.44} with cases (i) and (ii) replaced respectively by
\begin{itemize}
\item[(i'):] $A$ is uniformly regular, $F$ satisfies \eqref{3.23}  and either  (a) F5\textsuperscript{+}, with $b=\infty$ and \eqref{3.28}, or
(b) F5, with $b=\infty$,  and   \eqref{3.25};
\item[(ii'):]  $\mathcal{F}$ is orthogonally invariant satisfying F7, $A$ and $B$ satisfy \eqref{3.30} and \eqref{3.32}.
\end{itemize}
Using our barrier constructions in Section 2 of \cite{JT-oblique-II} in the  proof of case (i') also enables some alternative conditions to uniform regularity which would include the case when $A$ is independent of $p$.
\end{remark}

\begin{remark}\label{Remark3.2}
From the estimate \eqref{3.45} in Remark \ref{Remark3.1}, we can infer, under the same hypotheses, a global gradient bound for admissible solutions $u\in C^3(\Omega)\cap C^1(\bar\Omega)$ of the  boundary value problem \eqref{1.1}-\eqref{1.2} for nonlinear $\mathcal G$, when $G$ is uniformly concave in $p$, that is
\begin{equation}\label{uniformly concave G}
D_{p_ip_j}G(x,z,p) \xi_i\xi_j \le - \sigma_0 |\xi|^2,
\end{equation}
for all $x\in \partial\Omega$, $|z|\le M$, $p\in \mathbb{R}^n$, any unit vector $\xi$, and some positive constant $\sigma_0$, depending on the constant $M$. By virtue of the global bound \eqref{3.45}, we only need to estimate $Du$ on $\partial\Omega$. Using Taylor's expansion, with $\theta\in (0,1)$, we have on $\partial\Omega$,
\begin{equation}
\begin{array}{rl}
0=G(x,u,Du) \!\!&\!\!\displaystyle = G(x,u,0) + D_{p_i}G(x,u,0)D_iu + \frac{1}{2} D_{p_ip_j}G(x,u,\theta Du) D_iuD_ju \\
            \!\!&\!\!\displaystyle \le G(x,u,0) +  D_{p_i}G(x,u,0)D_iu - \frac{1}{2}\sigma_0 |Du|^2,
\end{array}
\end{equation}
which leads to
\begin{equation}\label{boundary gradient for uniformly concave G}
|Du|\le C, \quad \quad {\rm on}\ \partial\Omega,
\end{equation}
and hence
\begin{equation}
|Du|\le C, \quad \quad {\rm on}\ \Omega,
\end{equation}
where $C$ depends on $F, A, B, \Omega,b_0, \sigma_0$ and $|u|_{0;\Omega}$.
\end{remark}

\begin{remark}\label{Remark3.3}
We remark that it is not necessary to restrict $\beta=\nu$ in case (ii) of Theorem \ref{Th1.3} and we can assume more generally
\begin{equation}\label{relaxed beta}
|\beta - \nu| = |\beta^\prime|< 1/\sqrt n.
\end{equation}
Replacing $\nu$ by $\beta$ in \eqref{D i v 2}, we still obtain $w_{11} (x_0) < 0$, if $|Du(x_0)| \ge C_1$, under condition \eqref{relaxed beta},  where now $C_1$ depends also on $\sup_{\partial\Omega} |\beta-\nu|$.
\end{remark}

\subsection{Local gradient estimates}\label{subsection3.2}

In this subsection we prove local and interior versions of Theorem \ref{Th1.3} and unlike the global gradient estimates in the previous section we will need the full strength of the growth conditions in \eqref{quadratic growth} with respect to the
$p$ variables. The local estimates will also provide us with estimates in terms of moduli of continuity of solutions under weaker growth conditions analogous to the uniformly elliptic case in \cite{LieTru1986}. For the latter estimates we also need to assume in case (ii) a complementary condition to \eqref{3.23}, namely
 that there exist non-negative constants $\mu_0$, $\mu_1$ and $\mu_2$ such that for any $r\in \Gamma$,
\begin{equation}\label{3.053}
-r\cdot F_r \le \mu_0 + \mu_1 \mathscr{T}(r) + \mu_2 |F(r)|.
\end{equation}
Clearly \eqref{3.053} is satisfied trivially for positive homogeneous $F$  or if F2 and either $a_0 > -\infty$ or F4 are satisfied  by \eqref{homogeneity}.

We summarise the results in the following theorem, where for convenience we use balls rather than the domains $\Omega_0$ and $\Omega^\prime$ in  Theorem \ref{Th1.1}.

\begin{Theorem}\label{Th3.1}
Let $u\in C^3(\Omega)$ be an admissible solution of equation \eqref{1.1} in $\Omega$ and assume
 that  $\mathcal F$ satisfies F1 and F3, $A, B \in C^1(\bar\Omega\times\mathbb{R}\times \mathbb{R}^n)$, satisfy \eqref{quadratic growth}, $b_0:= \inf\limits_\Omega B(\cdot, u,Du) >a_0$ together with
one of the following further conditions:
\begin{itemize}
\item[(i):] $A$ is uniformly regular, $F$ satisfies \eqref{3.23}   and either  {\rm (a)} F5\textsuperscript{+}, with $b=\infty$, or
{\rm (b)} F5, with $b=\infty$,  and  \eqref{3.28};
\item[(ii):]  $\mathcal{F}$ is orthogonally invariant satisfying F7  and $A$ satisfies \eqref{3.32}.

\end{itemize}
Then for any $y\in\bar\Omega$, $0< R < 1$ and ball $B_R = B_R(y)$,
we have the estimate
\begin{equation}\label{3.50}
|Du(y)| \le C (\frac{1}{R}+ \sup\limits_{\partial\Omega \cap B_R} |Du|),
\end{equation}
for $u\in C^1(\bar\Omega\cap B_R)$, where $C$ is a constant depending on $F, A, B, \Omega,b_0$ and
$|u|_{0;\Omega\cap B_R}$. Furthermore if $y\in\partial \Omega$, $\partial\Omega\cap B_R\in C^{2,1}$ and
$\mathcal G[u] = 0$ on $\partial\Omega$, for an oblique semilinear boundary operator $\mathcal G$, given by
\eqref{semilinear}, with  $\varphi\in C^2(\partial\Omega\cap B_R\times \mathbb{R})$, $\beta\in C^2(\partial\Omega\cap B_R)$  in case {\rm (i)} and  $\beta=\nu$  in case {\rm (ii)}, then we have the estimate
\begin{equation}\label{3.51}
|Du(y)| \le \frac{C}{R},
\end{equation}
for $u\in C^2(\bar\Omega\cap B_R)$, where $C$ depends additionally on $\beta$ and $\varphi$. If  ``$o$'' is extended to ``$O$'' in \eqref{quadratic growth} and \eqref{3.32}, then there exists a positive constant $\theta$ depending on the same quantities as
$C$, such that the estimates \eqref{3.50} and \eqref{3.51} continue to hold provided
$\mathop{\rm osc}\limits_{ \Omega \cap B_R} u < \theta$ and $\mathcal{F}$ satisfies \eqref{3.053} in case {\rm (ii)} with $C$ depending additionally on $\mu_0$, $\mu_1$ and $\mu_2$.
\end{Theorem}

\begin{proof}
Theorem \ref{Th3.1} follows by modification of the proof of Theorem \ref{Th1.3}.
First we fix a function $\zeta\in C^0(\bar B_R)\cap C^2(S)$, satisfying $0\le\zeta\le 1$, where $S=S_\zeta\subset B_R$, denotes the support of $\zeta$. From \eqref{quadratic growth}, \eqref{g estimate} and Cauchy's inequality, we now obtain in place of \eqref{g estimate},
\begin{equation}\label{3.52}
\begin{array}{rl}
\mathcal{L}(\zeta^2 g) =   \!\!&\!\!\displaystyle  \zeta^2 \mathcal{L} g + (\mathcal {L}\zeta^2) g + 2F^{ij}D_i\zeta^2D_jg\\
                       \ge \!\!&\!\!\displaystyle  a_1\zeta^2  \mathcal E_2^\prime  - C\zeta^2 (\mathscr{T} +1)(\omega|Du|^4 +1)\\
                           \!\!&\!\!\displaystyle  - C(\mathscr{T}+1)(|Du|^2+1)[|D\zeta|^2 + (|Du|+1)\zeta|D\zeta | + \zeta|D^2\zeta |]\\
                        \ge \!\!&\!\!\displaystyle  a_1\zeta^2  \mathcal E_2^\prime  - C(\mathscr{T} +1)(|Du|^2+1) (\omega\zeta^2|Du|^2 +
                        \zeta|Du||D\zeta| + |D\zeta|^2 +|D^2\zeta| + 1),

\end{array}
\end{equation}
in $\Omega\cap S$, where $\omega=\omega(|Du|)$ is a positive decreasing function on $[0,\infty)$ tending to $0$ at infinity. With $g$ defined by \eqref{g}, we consider now in place of \eqref{3.13}, auxiliary functions of the form
\begin{equation}\label{3.53}
v:= \zeta^2 g + \tilde M^2_1 (\alpha\eta -\kappa \phi),
\end{equation}
where $\tilde M_1 = \sup\limits_{\Omega\cap B_R} \zeta |Du|$, $\alpha, \kappa, \eta$ and $\phi$ are as before, except that  $\Omega$ is replaced by $ \Omega \cap B_R$, in the definitions of $M^+_0$ and $M^-_0$. For the estimate \eqref{3.50}, which is the local version of the estimate \eqref{3.45} in Remark \ref{Remark3.1}, we take as there $g=|Du|^2$, $\kappa = 0$ and choose
\begin{equation}\label{zeta}
\zeta(x)=1-\frac{|x-y|^2}{R^2},
\end{equation}
so that
\begin{equation}\label{zeta prop}
\zeta(y) = 1, \  |D\zeta|\le c/R,  \  |D^2\zeta|\le c/R^2,
\end{equation}
for some constant $c$, and
\begin{equation}\label {3.58}
\mathcal{L}(\zeta^2 g) \ge a_1 \zeta^2 \mathcal E_2^\prime - C(\mathscr{T}+1)(\omega \zeta^2 |Du|^4 + \frac{1}{R}\zeta |Du|^3)
\end{equation}
at the maximum point $x_0$ of $v$ in $\Omega\cap S$, provided $\zeta(x_0) |Du(x_0)|> 1/R$ and $|Du(x_0)|>1$.
For the estimate \eqref{3.51}, we need to first take $R$ sufficiently small so that there exists a cut-off function $\zeta\in C^0(\bar B_R)\cap C^2(S_\zeta)$, $0\le\zeta\le 1$ satisfying  \eqref{zeta prop} together with the boundary condition
\begin{equation}\label{zeta boundary prop}
D_\beta \zeta = 0 \quad {\rm on} \ \partial\Omega\cap S_\zeta.
\end{equation}
We show how to construct such a function $\zeta$ from the function \eqref{zeta} at the end of the proof.

From the property \eqref{zeta boundary prop}, we now obtain in place of \eqref{3.15},
\begin{equation}
\begin{array}{ll}
D_\beta v \!\!&\!\!\displaystyle = 2\zeta (D_\beta \zeta) g + D_{\beta}|\delta u|^2 + \tilde M^2_1(\alpha D_\beta \eta - \kappa D_\beta \phi) \\
          \!\!&\!\!\displaystyle \ge 2 \zeta^2 (\delta_k u) D_\beta\delta_k u + \kappa \tilde M^2_1 -  \alpha K\eta |\varphi|(\cdot,u)\tilde M^2_1,
          \end{array}
\end{equation}
on $\partial\Omega\cap S$. With these modifications, the estimates \eqref{3.50} and \eqref{3.51} follow from the proof of Theorem \ref{Th1.3}, with $M_1$ replaced by $\tilde M_1$. In case (ii), we obtain in place of the estimate
\eqref{push a negative w11},
\begin{equation}\label{local negative w11}
\zeta^2w_{11} \le -\frac{1}{6}\alpha K M^2_1 + C[\zeta^2(\omega|Du|^2+1)  +  \frac{1}{R}\zeta |Du|],
\end{equation}
so that we obtain again $w_{11}(x_0) < 0$ provided $\zeta(x_0) |Du(x_0)|> C_1/R$, for some constant $C_1$,  depending on $F, A, B, \Omega, \varphi$ and $M_0$.

If we replace ``$o$'' by ``$O$'' in the structure conditions \eqref{quadratic growth}, we  obtain \eqref{3.58} with  $\omega = 1$. Similarly if we replace ``$o$'' by ``$O$'' in \eqref{3.32}, we obtain \eqref{local negative w11} with $\omega = 1$. Accordingly we may still arrive at our desired gradient estimates, \eqref{3.50} and \eqref{3.51}, if $\alpha$ can be chosen sufficiently large, in which case we can still satisfy
 \eqref{key restrictions} for $\alpha$, provided
\begin{equation}
 \mathop{\rm osc}\limits_{ \Omega \cap B_R} u =M^+_0 - M^-_0 \le\theta : =  \frac{1}{K} \log (1+ \frac{a_1}{16\alpha}).
\end{equation} Note that in case (ii), we cannot satisfy the further restriction $\alpha K\eta \le \frac{1}{16}$, for large
$\alpha$, so here we use condition \eqref{3.053} to control  the term $F^{ij}u_{ij}$ in \eqref{3.41}. We remark that when $A$ is regular such a control can be alternatively achieved through a barrier \cite{JT-oblique-II}.

To end the proof of Theorem \ref{Th3.1}, we give the key construction of the cut-off function at boundary.

{\it Construction of cut-off function at boundary.} We fix a point $y \in \partial\Omega$, which we may take to be the origin, and a coordinate system so that $\nu(0)=e_n$. Suppose that in some ball $B_\rho=B_\rho (0)$, $\Omega\cap B_\rho = \{x_n >h(x')\}$, $\partial\Omega \cap B_\rho = \{x_n=h(x')\}$, where $h\in C^3(\bar B_\rho^\prime)$, $B_\rho^\prime=\{|x'|<\rho\}$ and $x'=(x_1,\cdots, x_{n-1})$. By taking $\rho$ sufficiently small, we can assume,
\begin{equation}
\chi (x):= \beta_n(x) -\sum_{i=1}^{n-1} \beta_i (x) D_ih(x') \ge 1-\delta
\end{equation}
for any fixed $\delta>0$, since $\beta_n(0)=1$, $Dh(0)=0$. Now we consider a coordinate transformation $x\rightarrow \tilde x=\psi(x)$, where
\begin{equation}\label{cut-off diffeomorphism}
\begin{array}{rl}
\tilde x_i \!\!&\!\!\displaystyle = x_i - \beta_i (x) \tilde x_n, \quad \quad i=1,\cdots, n-1,\\
\tilde x_n \!\!&\!\!\displaystyle = \frac{1}{\chi(x)}(x_n -h(x')).
\end{array}
\end{equation}
Again with $\rho$ sufficiently small, we have $\psi(0)=0$, $\psi(\partial\Omega\cap B_\rho) = \{\tilde x_n=0\}$, $D\psi (0)=I$ and $\det D\psi >0$, so that in particular $\psi$ is invertible in $B_\rho$. Furthermore, if $\zeta \in C^1(\bar \Omega\cap B_\rho)$, $\tilde \zeta = \zeta \circ \psi^{-1}\in C^1(\psi(\bar \Omega \cap B_\rho))$, we obtain by calculation
\begin{equation}\label{cut-off property}
\begin{array}{rl}
D_{\tilde x_n} \tilde \zeta \circ \psi \!\!&\!\! \displaystyle = \sum_{i=1}^n\frac{\partial\zeta}{\partial x_i} \frac{\partial x_i}{\partial \tilde x_n}\\
\!\!&\!\! \displaystyle = \sum_{i=1}^{n-1}\beta_i(x)\frac{\partial\zeta}{\partial x_i} + (\chi(x)+ \sum_{i=1}^{n-1}\beta_i(x)D_ih(x'))\frac{\partial\zeta}{\partial x_n}\\
\!\!&\!\! \displaystyle = \sum_{i=1}^{n-1}\beta_i(x)\frac{\partial\zeta}{\partial x_i} + \beta_n(x) \frac{\partial\zeta}{\partial x_n}=D_\beta \zeta, \quad \quad {\rm on}\ \partial\Omega \cap B_\rho.
\end{array}
\end{equation}
Hence if $\tilde \zeta \in C^{1}_{0}(\psi(\bar\Omega \cap B_\rho))$ satisfies $D_{\tilde x_n}\tilde \zeta (\tilde x',0)=0$, we have $D_\beta \zeta =0$ on $\partial\Omega\cap B_\rho$.

With the help of \eqref{cut-off property}, a boundary cut-off function $\zeta\in C^0(\bar B_R)\cap C^2(S_\zeta)$, $0\le\zeta\le 1$ satisfying \eqref{zeta prop} and \eqref{zeta boundary prop} can be constructed. For a fixed point $y\in \partial\Omega$, which we may take to be the origin, we make the coordinate transformation $x\rightarrow \tilde x=\psi(x)$ as in \eqref{cut-off diffeomorphism}. In the $\tilde x$-coordinate system, we can choose the function
\begin{equation}
\tilde \zeta = 1-\frac{|\tilde x|^2}{R^2},
\end{equation}
then the function $\zeta=\tilde \zeta \circ \psi$ is the desired cut-off function satisfying the above properties \eqref{zeta prop} and \eqref{zeta boundary prop} as we expected.
\end{proof}

\begin{remark}\label{Remark3.4}
Note that  when $\beta=\nu$, \eqref{g} is similar to the corresponding function used for the gradient estimate of Neumann problems in \cite{JLL2007}, (and more recently for the $k$-Hessian equations in \cite{MX2014}). In our proof, we use the auxiliary functions \eqref{3.13} and \eqref{3.53}, which are modifications of the auxiliary functions used in Section 3 of \cite{LieTru1986} for uniformly elliptic equations and for interior gradient bounds for $k$-Hessian equations in \cite{Tru1997}. We remark that we can use alternative  functions; in particular  functions of the form $v = g \exp{(\alpha \eta- \kappa\phi)}$ and $v=\zeta^2 g \exp{(\alpha \eta- \kappa\phi)}$, with appropriately chosen positive constants $\alpha$ and $\kappa$, in place of \eqref{3.13} and \eqref{3.53} respectively.

\end{remark}

\begin{remark}\label{Remark3.5}
If we assume the matrix $A$ satisfies the following condition
\begin{equation}\label{3.68}
D_z A(x,z,p) \ge c |p|^2I,
\end{equation}
as $|p|\rightarrow\infty$, $x\in \Omega$, $|z|\le M$, $p\in \mathbb{R}^n$ for any $M>0$ and for some $c>0$, we can dispense with the uniformly regular assumption on $A$ in Theorem \ref{Th1.3} and Theorem \ref{Th3.1} and the proof is much simpler. More generally we can replace the exponent $2$ on the right hand side of \eqref{3.68} by $m$ for $0\le m \le 2$ provided the powers of $|p|$ in the growth conditions \eqref{quadratic growth} are reduced by $2-m$. When the constant $c$ is sufficiently large, ``$o$'' in \eqref{quadratic growth} can be weakened to ``$O$''.
\end{remark}

\begin{remark}\label{Remark3.6}
Using the cut-off function $\zeta$ constructed in the proof of Theorem \ref{Th3.1}, we can also extend the global second derivative estimates in Theorem \ref{1.2} to local estimates in the case of semilinear boundary operators
$\mathcal G$. As in Remark \ref{Remark 2.1} we need only assume $\partial\Omega\cap B$  is uniformly $(\Gamma,A,G)$-convex with respect to $u$ for some ball $B = B_R(y)$ of radius $R$, centred at $y\in\bar\Omega$. Under the hypotheses of Theorem \ref{1.2}, with $\Omega$ replaced by $\Omega\cap B$ and $\partial\Omega$ replaced by
$\partial\Omega\cap B$, we then obtain for semilinear  $\mathcal G$, in place of  \eqref{1.10},
\begin{equation}\label{1.10 local}
 |D^2 u(y)|\le C(1 + R^{-2}),
\end{equation}
where $C$ is now a constant depending on $F, A, B, G, \Omega, \beta_0$ and  $|u|_{1;\Omega\cap B}$.

\end{remark}

\begin{remark}\label{Remark3.7}
We remark that the uniformly regular definition \eqref{1.18} \eqref{1.19} of the matrix $A$, can also be equivalently formulated as follows, namely
\begin{equation}\label{alter lambda}
\lambda(x,u,p) = \inf_{|\xi|=|\eta|=1, \atop \xi\cdot \eta=0}  A_{ij}^{kl}(x,u,p)\xi_i\xi_j\eta_k\eta_l \ge \lambda_0 >0,
\end{equation}
and
\begin{equation}\label{alter Lambda}
\Lambda(x,u,p) = \sup_{|\xi|=|\eta|=1} |A_{ij}^{kl}(x,u,p)\xi_i\xi_j\eta_k\eta_l| \le \mu_0 \lambda(x,u,p),
\end{equation}
for $x\in \Omega$, $|u|\le M$, for positive constants $\lambda_0$ and $\mu_0$, depending on $M$. Then the estimates \eqref{1.21}, \eqref{3.50} and \eqref{3.51} also hold for $A$ satisfying \eqref{alter lambda}, \eqref{alter Lambda}, $\mathcal{F}$ and $B$ satisfying (i) of Theorems \ref{Th1.3} and \ref{Th3.1}.
\end{remark}

\subsection{H\"older estimates}\label{subsection3.3}

In this subsection, we will prove a H\"older estimate for admissible functions $u$ of the augmented Hessian equation \eqref{1.1} in the cones $\Gamma_k$ for $k>n/2$, when $A\ge O(|p|^2)I$. For $M[u]\in \Gamma_k$ for $n/2 < k \le n$, we have
\begin{equation}\label{k-Hessian equation}
F_{k}(M[u])>0, \quad {\rm in}\  \Omega,
\end{equation}
where the operator $F_{k}=({S_k})^{\frac{1}{k}}$.
Here the condition $A\ge O(|p|^2)I$ is interpreted as the natural quadratic structure condition from below, namely
\begin{equation}\label{quadratic structure of A}
A(x,z,p) \ge -\mu_0 (1+|p|^2)I,
\end{equation}
for all $x\in \Omega$, $|z|\le M$, $p\in \mathbb{R}^n$, and some positive constant $\mu_0$ depending on the constant $M$. The one-sided quadratic condition \eqref{quadratic structure of A} has already been used for the gradient estimate in the Monge-Amp\`ere case, for the Dirichlet problem in \cite{JTY2013} and the Neumann and oblique problems in \cite{JTX2015}.

\begin{Lemma}\label{Lemma 3.1}
Let $u\in C^2(\Omega)$ satisfy $M[u]\in \Gamma_k$ for $n/2<k\le n$ where the matrix $A$ satisfies  \eqref{quadratic structure of A}.
 Then for any ball $B_R = B_R(y)$, with centre $y \in \bar\Omega$, $x\in \Omega_R:=B_R\cap\Omega$ and $\alpha = 2-n/k$, we have the estimate
\begin{equation}\label{Holder 1}
|u(x) - u(y)|\le C |x-y|^\alpha (R^{-\alpha}\mathop{\rm osc}_{\Omega_R}u  +1),
\end{equation}
provided one of the following holds:
\begin{itemize}
\item[(i):] $B_R\subset\Omega$, $u\in C^0(\bar \Omega_R)$ and the constant $C$ depends on $n,k,\mu_0$, $\mathop{\rm osc}\limits_{\Omega_R}u$ and {\rm diam} $\Omega$;
\item[(ii):] $\Omega \in C^2$ is convex, $u\in C^1(\bar \Omega_R)$ and $C$ depends additionally on $\Omega$ and $\inf\limits_{B_R\cap\partial\Omega} D_\nu u$;
\item[(iii):] $u\in C^0(\bar\Omega_R)\cap C^{0,\alpha}(\bar B_R\cap \partial\Omega)$ and $C$ depends additionally on  $[u]_{\alpha; B_R\cap\partial\Omega}$.
\end{itemize}

\end{Lemma}

\begin{proof}
First we consider the interior case (i). For any ball $B_R=B_R(y)\subset \Omega$, we need to compare the following two functions in $B_R$,
\begin{equation}\label{v function}
v(x)=e^{K(u(x)-u(y))} + \frac{a}{2} |x-y|^2-1,
\end{equation}
and
\begin{equation}\label{phi function}
\Phi(x)= c |x-y|^\alpha, \quad \alpha = 2-\frac{n}{k},
\end{equation}
where $n/2<k\le n$, $K$, $a$ and $c$ are positive constants to be determined. By direct calculation, we first observe that
\begin{equation}
F_{k}(D^2\Phi)=0, \quad {\rm for} \ x\neq y.
\end{equation}
Denoting $\tilde \Phi= \Phi + \frac{\epsilon}{2}|x-y|^2$ for some positive constant $\epsilon<1$, then the perturbation function $\tilde \Phi$ of $\Phi$ satisfies $D^2\tilde \Phi \in \Gamma_k$ in $B_R$. Fixing a constant $\rho < \min\{R,1\}$, it is readily checked that $F_k(\tilde \Phi)$ is a strictly decreasing function with respect to $|x-y|$ for $x\in B_R\backslash B_\rho$, and hence
\begin{equation}\label{Fk decreasing in r}
F_k(D^2\tilde \Phi)< F_k(D^2 \tilde \Phi)|_{x\in \partial B_\rho}\le C(n,k,c)\left(\frac{\epsilon}{\rho^{{n(k-1)}/{k}}}\right)^{\frac{1}{k}}, \quad {\rm for}\  x\in B_R\backslash B_\rho,
\end{equation}
where $B_\rho:=B_\rho(y)$, $C(n,k,c)$ is a positive constant depending on $n$, $k$ and $c$. By introducing
\begin{equation}
S^{ij}_{k}:=S^{ij}_{k}(D^2\tilde \Phi)=\frac{\partial S_{k}(D^2\tilde \Phi)}{\partial\tilde \Phi_{ij}}, \quad {\rm and} \ \  F^{ij}_{k}:=F^{ij}_{k}(D^2\tilde \Phi)=\frac{\partial F_{k}(D^2\tilde \Phi)}{\partial\tilde \Phi_{ij}},
\end{equation}
we have
\begin{equation}\label{relation between Sk and Fk}
S_{k}^{ij}=kF_{k}^{k-1}(D^2\tilde \Phi)F_{k}^{ij}.
\end{equation}
We also denote $\mathscr{T}_{S_{k}} = {\rm trace}(S_{k}^{ij})$ = $(n-k+1)S_{k-1}$. Since $D^2\tilde \Phi \in \Gamma_k$ in $B_R$, we have $\{F_{k}^{ij}\}>0$ and $\{S_{k}^{ij}\}>0$ in $B_R$.

For our desired comparison of the functions $v$ in \eqref{v function} and $\Phi$ in \eqref{phi function} on $B_R$, we shall first compare them on $B_R\backslash B_\rho$ for a fixed $\rho$, and then send $\rho$ to $0$ in the end.
For convenience of later discussion, we now introduce some notation and fix some constants in advance. We denote
\begin{equation}\label{tilde K M tilde K m}
\omega_R = \mathop{\rm osc}\limits_{B_R}u, \quad \tilde K_M = Ke^{K\omega_R},\quad \tilde K_m = Ke^{-K\omega_R},
\end{equation}
and
\begin{equation}\label{delta n k}
\delta(n,k)=\left\{
\begin{array}{ll}
{1}/{(n-k+1)}, & {n}/{2}<k<n,\\
{1}/{n},     & k=n.
\end{array}
\right.
\end{equation}
We can fix the constant $K$ large such that $K> \mu_0/\delta(n,k)$, and fix the constant $a$ such that $a> 1 +\mu_0\tilde K_M$. For fixed $K$, by choosing 
$$c\ge R^{\frac{n}{k}}[R^{-2}(e^{K\omega_R-1)}+\frac{a}{2}],$$
we have $v-\Phi\le 0$ on $\partial B_R$. Now the constant $c$ has been fixed as well.
For fixed $\rho$ and $c$, by choosing $\epsilon$ sufficiently small such that the quantity on the right hand side of \eqref{Fk decreasing in r} is sufficiently small, we can have
\begin{equation}\label{perturbed Hessian}
F_{k}(D^2\tilde \Phi)<  \tilde K_m\inf_{\Omega}F_{k}(M[u]), \quad {\rm for} \ x\in B_{R}\backslash B_\rho,
\end{equation}
where $\tilde K_m$ is the constant defined in \eqref{tilde K M tilde K m}.
If $v-\Phi$ attains its maximum over $\overline {B_R\backslash B_\rho}$ at a point $x_0\in B_R \backslash B_\rho$, then we have $D(v-\Phi)=0$ and $D^2(v-\Phi) \le 0$ at $x_0$, namely
\begin{equation}\label{max interior equ 1}
u_i (x_0)={\tilde K}^{-1}[-a+\alpha c|x_0-y|^{-\frac{n}{k}}][(x_0)_i -y_i],\ {\rm for}\ i=1,\cdots,n,
\end{equation}
and
\begin{equation}\label{max interior equ 2}
S_{k}^{ij}D_{ij}(v-\Phi)\le 0,\quad {\rm at} \ x_0,
\end{equation}
where $\tilde K= Ke^{K(u(x_0)-u(y))}$.
Without loss of generality, by rotation of the coordinates, we can assume $x_0-y = ((x_0)_1 -y_1, 0, \cdots, 0)$. From \eqref{max interior equ 1}, we have $Du(x_0)=(u_1(x_0), 0, \cdots, 0)$. By calculation, we have
\begin{equation}
D^2\Phi(x_0) = \alpha c|x_0-y|^{-\frac{n}{k}}{\rm diag} (1-\frac{n}{k}, 1, \cdots, 1),
\end{equation}
which has a negative eigenvalue $\Phi_{11}(x_0)$ when $n/2<k< n$ and has a null eigenvalue $\Phi_{11}(x_0)$ when $k=n$. Correspondingly, for small $\epsilon$, the perturbed Hessian $D^2\tilde \Phi(x_0)=D^2\Phi(x_0)+\epsilon I$ is diagonal, and has a negative eigenvalue $\tilde \Phi_{11}(x_0)$ when $n/2<k< n$ and has a least positive eigenvalue $\tilde \Phi_{11}(x_0)$ when $k=n$. Note also that the matrices $\{F_{k}^{ij}\}$ and $\{S_{k}^{ij}\}$ are diagonal at $x_0$. From the properties of the $k$-Hessian operator and the Monge-Amp\`ere operator, we have
\begin{equation}\label{F7 property for Sk}
S_k^{11} \ge \delta(n,k) \mathscr{T}_{S_k},
\end{equation}
where $\delta(n,k)$ is defined in \eqref{delta n k}.

By our choices of the constants $K$ and $a$, from $w_{ij}=u_{ij}-A_{ij}$, \eqref{quadratic structure of A}, \eqref{perturbed Hessian}, \eqref{relation between Sk and Fk}, \eqref{F7 property for Sk}, and the concavity and homogeneity of $F_{k}$, we have, at $x_0$,
\begin{equation}\label{max interior equ 3}
\begin{array}{rl}
  S_k^{ij}D_{ij}(v-\Phi)=    \!\!&\!\! \displaystyle S_k^{ij}[K\tilde K u_iu_j + \tilde K A_{ij} + (a-\epsilon)\delta_{ij} + \tilde K w_{ij}-\tilde \Phi_{ij}] \\
  \ge  \!\!&\!\! \displaystyle K\tilde K S_k^{11}u_1^2 + S_k^{ij}[-\mu_0 \tilde K(1+|Du|^2)\delta_{ij} + (a-\epsilon)\delta_{ij}] \\
       \!\!&\!\! \displaystyle + kF_k^{k-1}(D^2\tilde \Phi)F_k^{ij}[\tilde K w_{ij}-\tilde \Phi_{ij}] \\
  \ge  \!\!&\!\! \displaystyle \tilde K_m (K\delta(k,n)-\mu_0)\mathscr{T}_{S_k}|Du|^2 + (a-\epsilon-\mu_0\tilde K_M)\mathscr{T}_{S_k} \\
       \!\!&\!\! \displaystyle + kF_k^{k-1}(D^2\tilde \Phi)[\tilde K_m F_k(M[u])-F_k(D^2\tilde \Phi)] \\
  >    \!\!&\!\! \displaystyle 0.
\end{array}
\end{equation}
The contradiction from \eqref{max interior equ 2} and \eqref{max interior equ 3} shows that $v-\Phi$ must take its maximum over $\overline {B_R\backslash B_\rho}$ at the boundary $\partial B_R$ or $\partial B_\rho$. Therefore, we have
\begin{equation}\label{comparison on B_R - B_rho}
v-\Phi \le \max\{0,\sup_{B_\rho}(v-\Phi)\}, \quad {\rm for}\  x\in \overline {B_R\backslash B_\rho}.
\end{equation}
We observe that we can choose $\epsilon$ as small as we want in \eqref{max interior equ 3}. Letting $\epsilon\rightarrow 0$ in \eqref{max interior equ 3}, the inequality \eqref{comparison on B_R - B_rho} can hold for $\rho$ as small as we want. Sending $\rho\rightarrow 0$ and using the forms of $v$ in \eqref{v function} and $\Phi$ in \eqref{phi function}, the right hand side of \eqref{comparison on B_R - B_rho} tends to $0$. Correspondingly, we have from \eqref{comparison on B_R - B_rho},
\begin{equation}
v \le \Phi, \quad {\rm for} \ x\in \bar B_R,
\end{equation}
namely
\begin{equation}\label{interior comparison}
u(x)-u(y)\le \frac{1}{K} \log (1+c|x-y|^\alpha), \quad {\rm for}\   x\in\bar B_R,
\end{equation}
and hence assertion (i) follows from the estimate, $c\le R^{-\alpha} \omega_R \tilde K_M$.
\begin{remark}\label{Remark3.7}
In the above argument, we use the perturbation $\tilde \Phi=c|x-y|^\alpha+ \frac{\epsilon}{2}|x-y|^2$ of the function $\Phi=c|x-y|^\alpha$. We remark that there are alternative perturbations that can be used here.  For instance, we can choose a perturbation in the form, $\tilde \Phi= c(|x-y|^2+\epsilon^2)^{\frac{\alpha}{2}}$ for small $\epsilon$.
\end{remark}
If we consider  admissible functions $u$ of equation \eqref{1.1} in the cones $\Gamma_k$ for $k>n/2$ satisfying various boundary conditions, we can also have similar comparison in \eqref{interior comparison} locally near the boundary. For the Neumann case (ii), we suppose $B_R$ intersects $\partial\Omega$, with $\Omega_R$ convex and  $D_\nu u > - \kappa$ on $\partial\Omega\cap B_R$, where $\kappa$ is a nonnegative constant.
 Defining $\tilde u= u -\kappa\phi$, where $\phi$ is a negative defining function of $\Omega$, as in Section \ref{subsection3.1} and \ref{subsection3.2}, satisfying $\phi<0$ in $\Omega$, $\phi=0$ on $\partial\Omega$ and $D_\nu \phi=-1$ on $\partial\Omega$, then we have $D_\nu \tilde u > 0$ on $\partial\Omega\cap B_R$. Using $\tilde u$ in place of $u$ in \eqref{v function}, we then compare the replaced function $v$ in \eqref{v function} and the function $\Phi$ in \eqref{phi function} on $\bar \Omega_R$. Similar to the interior case, we begin with our discussion on $\overline {(B_R\backslash B_\rho)\cap \Omega}$ for a fixed $\rho<\min\{R,1\}$, where $B_\rho:=B_\rho(y)$. If the maximum of $v-\Phi$ takes its maximum over $\overline {(B_R\backslash B_\rho)\cap \Omega}$ at a point $x_0\in (B_R\backslash B_\rho)\cap \Omega$, by choosing a larger constant $a$ depending additionally on $\kappa$, $|\phi|_{2;\Omega}$, we can obtain the same inequality as in \eqref{max interior equ 3} and get a contradiction with \eqref{max interior equ 2}. Therefore, the possibilities that maximum point $x_0$ of $v-\Phi$ occurs are on $(B_R\backslash B_\rho)\cap \partial\Omega$, $\partial B_R \cap \Omega$ or $\partial B_\rho\cap \Omega$. If $x_0\in (B_R\backslash B_\rho)\cap \partial\Omega$, from the convexity of $\Omega_R$, we have $(x_0-y)\cdot \nu(x_0) \le 0$. Then we have, at $x_0$,
\begin{equation}\label{contradiction on the boundary}
\begin{array}{rl}
0 \!\!&\!\!\displaystyle \ge D_\nu (v-\Phi)(x_0) \\
  \!\!&\!\!\displaystyle =   Ke^{K(\tilde u(x_0)-\tilde u(y))}D_\nu \tilde u(x_0) + [a-\alpha c|x_0-y|^{-\frac{n}{k}}](x_0-y)\cdot \nu(x_0) \\
  \!\!&\!\!\displaystyle > 0,
\end{array}
\end{equation}
by using $D_\nu \tilde u> 0$ on $\partial\Omega$, and choosing the constant $c$ large such that $c\ge aR^{\frac{n}{k}}/\alpha$. Then \eqref{contradiction on the boundary} leads to a contradiction and excludes the case that the maximum of $v-\Phi$ occurs at $(B_R\backslash B_\rho)\cap \partial\Omega$. By fixing the defining function $\phi$ such that $\phi>-1/\kappa$, we have $\kappa(\phi(x)-\phi(y))>-1$ for $x\in\partial B_R\cap\Omega$ and $y\in\partial\Omega$. With this property of the defining function, now by choosing $c$ larger again such that 
$$c\ge R^{\frac{n}{k}}[R^{-2}(e^{K(\omega_R+\kappa \mathop{\rm osc}\limits_{\Omega_R}\phi)}-1)+\frac{a}{2}],$$
 we have $v-\Phi\le 0$ on $\partial B_R\cap\Omega$. Similarly to \eqref{comparison on B_R - B_rho} of the interior case, with $\mu_0$ and $K$ appropriately adjusted, we now have
\begin{equation}
v-\Phi \le \max \{0, \sup_{B_\rho\cap \Omega}(v-\Phi)\}, \quad {\rm for}\ x\in \overline {(B_R\backslash B_\rho)\cap \Omega}.
\end{equation}
Therefore, by successively passing $\epsilon$ and $\rho$ to $0$, the same inequality \eqref{interior comparison} holds on $\bar \Omega_R$ and hence assertion (ii) is proved.

Finally for  Dirichlet  boundary values, as in case (iii), we suppose again that $B_R$ intersects
$\partial\Omega$ and  $u\in C^0(\bar \Omega_R) \cap C^{0,\alpha}(\bar B_R\cap \partial\Omega)$ so that
\begin{equation}\label{Holder x y}
|u(x) - u(y)|\le \kappa |x-y|^\alpha,
\end{equation}
for all $x,y\in B_R\cap \partial\Omega$ for some non-negative constant $\kappa = [u]_{\alpha;B_R\cap\partial\Omega}$. Assume first that the centre  $y\in\partial \Omega$. Then proceeding as in the previous case we need to compare $v$ and $\Phi$ on $(B_R\backslash B_\rho)\cap \partial\Omega$. Accordingly we now obtain
\begin{equation}\label{Dirichlet analyse for Holder}
\begin{array}{rl}
v-\Phi \!\!&\!\!\displaystyle = e^{K(u(x)-u(y))}-1 + \frac{a}{2}|x-y|^2 - c|x-y|^\alpha \\
       \!\!&\!\!\displaystyle \le \kappa Ke^{K\omega^\prime_R} |x-y|^\alpha+ \frac{a}{2}|x-y|^2 - c|x-y|^\alpha\\
       \!\!&\!\!\displaystyle \le 0, \quad\quad\quad\ {\rm for} \ x\in(B_R\backslash B_\rho)\cap \partial\Omega,
\end{array}
\end{equation}
by taking $c$ larger such that 
$$c\ge \kappa Ke^{K\omega^\prime_R}+
 \frac{a}{2}R^{2-\alpha}, $$
where $\omega^\prime_R= \mathop{\rm osc}\limits_{\partial\Omega\cap B_R}u$. With $K$ and $c$ also chosen as in case (i), with $B_R$ replaced by $\Omega_R$,   we arrive again at   \eqref{interior comparison} on $\bar \Omega_R$.

The general case, $y\in \Omega$, in case (iii), now follows by combining the case, $y\in\partial\Omega$ with the interior estimate, \eqref {Holder 1} in case (i), as in Theorem 8.29 in \cite{GTbook}.
\end{proof}

For convex domains, Lemma \ref{Lemma 3.1} extends the gradient bound, Lemma 3.2 in \cite{JTX2015}, for the case $k = n$. More generally it provides a modulus of continuity estimate for  solutions of \eqref{1.1} that are  admissible in $\Gamma_k$ for $k>n/2$. Combining with the local gradient estimate in Theorem \ref{Th3.1}, the estimate \eqref{3.50} can hold by extending ``$o$'' to ``$O$'' in \eqref{quadratic growth} and \eqref{3.32}. For a convex domain $\Omega$, the estimate \eqref{3.51} can still hold for the semilinear Neumann problems in case (ii) of Theorem \ref{Th3.1} by extending ``$o$'' to ``$O$'' in \eqref{quadratic growth} and \eqref{3.32}.


\section{Existence and applications}\label{Section 4}
\vskip10pt

In this section, we present some  existence results for classical solutions based on our first and second derivative {\it a priori} estimates for admissible solutions for the oblique boundary value problem \eqref{1.1}-\eqref{1.2}. We  also give various examples  of equations and boundary conditions satisfying our conditions and also show that our theory can be extended to embrace
$C^{1,1}$ solutions of degenerate equations.

\subsection{Existence theorems}\label{subsection4.1}

We assume that $\underline u$ and $\bar u$ in $C^2(\Omega)\cap C^1(\bar \Omega)$ are respectively an admissible subsolution and supersolution of the boundary value problem \eqref{1.1}-\eqref{1.2}, satisfying the inequalities \eqref{sub super F} and \eqref{sub super G}, with  $\mathcal F$ satisfying F1 and  $\mathcal G$ oblique. Under the assumptions $A$, $B$ and $G$ are non-decreasing in $z$, with at least one of them strictly increasing, if $u \in C^2(\Omega)\cap C^1(\bar \Omega)$ is an admissible solution of the problem \eqref{1.1}-\eqref{1.2}, by the comparison principle, we have
\begin{equation}\label{solution bound}
\underline u \le u \le \bar u, \quad \ {\rm in}\ \bar\Omega.
\end{equation}
For our purposes here we note that the comparison principle, as formulated in Lemma 3.1 in \cite{JTX2015}, extends automatically to operators $\mathcal F$ satisfying F1.
Then \eqref{solution bound} provides the solution bound and the interval $\mathcal{I}=[\underline u, \bar u]$ for the convexity definition \eqref{uniform convexity}.

With the {\it a priori} estimates up to second order, we can formulate  existence results for the classical admissible solutions of the oblique boundary value problems \eqref{1.1}-\eqref{1.2}. We consider first the case when the matrix $A$ is strictly regular and the boundary operator $\mathcal{G}$ is semilinear.
\begin{Theorem}\label{Th4.1}
Assume that $F$ satisfies conditions F1-F4 in the cone $\Gamma\subset \mathcal{P}_{n-1}$, $\Omega$ is a $C^{3,1}$ bounded domain in $\mathbb{R}^n$, $A\in C^2(\bar \Omega\times \mathbb{R}\times \mathbb{R}^n)$ is strictly regular in $\bar\Omega$, $B > a_0, \in C^2(\bar \Omega\times \mathbb{R}\times \mathbb{R}^n)$, $\mathcal G$ is semilinear and oblique with $G\in C^{2,1}(\partial\Omega\times \mathbb{R}\times\mathbb{R}^n)$, $\underline u$ and $\bar  u$, $\in C^2(\Omega)\cap C^1(\bar \Omega)$ are respectively an admissible subsolution and a supersolution of the oblique boundary value problem \eqref{1.1}-\eqref{1.2} with $\Omega$ uniformly $(\Gamma,A,G)$-convex with respect to the interval $\mathcal I = [\underline u, \bar u]$.
Assume also that $A$, $B$ and $\varphi$ are non-decreasing in $z$, with at least one of them strictly increasing, and that $A$ and $B$ satisfy the quadratic growth conditions \eqref{quadratic growth}. Assume either F5\textsuperscript{+} holds or $B$ is independent of $p$. Then if one of the following further conditions is satisfied:
\begin{itemize}
\item[(i):]  $A$ is uniformly regular and either {\rm (a)} F5\textsuperscript{+}, with $b=\infty$, or {\rm (b)} F5, with $b=\infty$ and $B -p\cdot D_pB\le o(|p|^2)$ in \eqref{quadratic growth};
\item[(ii):] $\beta=\nu$, $\mathcal{F}$ is orthogonally invariant and satisfies F7 and either (a) $A = o(|p|^2)$  in \eqref{quadratic growth} or (b) $\Gamma\subset\Gamma_k$ with $k>n/2$ and $\Omega$ is convex,
\end{itemize}
there exists a unique admissible solution $u \in C^{3,\alpha}(\bar \Omega)$ of the boundary value problem \eqref{1.1}-\eqref{1.2} for any $\alpha<1$.
\end{Theorem}
\begin{proof}
Under the assumptions of this theorem, we have  solution bounds \eqref{solution bound},  gradient estimates \eqref{1.21} from Section \ref{Section 3} and  second derivative estimates \eqref{1.10} by Theorem \ref{Th1.2}. Note that the gradient estimate when $\Gamma=\Gamma_k$ with $k>n/2$ and $\Omega$ is convex in case (ii) is obtained by combining the local gradient estimate in Theorem \ref{Th3.1} and the H{\" o}lder estimate in Lemma \ref{Lemma 3.1}, (see the last paragraph of Section \ref{Section 3}). From the uniformly elliptic theory, Theorem 3.2 in \cite{LT1986} or Theorem 1.1 in \cite{LieTru1986}, we can derive a global second derivative H\"older estimate
\begin{equation}\label{2nd Holder}
|u|_{2,\alpha;\Omega}\le C,
\end{equation}
for admissible solutions $u\in C^4(\Omega)\cap C^3(\bar \Omega)$ of the semilinear oblique boundary value problem \eqref{1.1}-\eqref{1.2} for $\alpha\in (0,1)$. With the $C^{2,\alpha}$ estimate, by choosing the subsolution $\underline u$ as an initial solution, we can employ the classical method of continuity, Theorem 17.22 and Theorem 17.28 in \cite{GTbook} to derive the existence of an admissible solution $u\in C^{2,\alpha}(\bar \Omega)$. Here, in order to preserve our subsolution and supersolution inequalities and guarantee the uniform {\it a priori} estimates, we need to consider the family of problems:
\begin{equation}\label{homotopy family}
\begin{array}{rl}
\mathcal{F}[u]=  \!\!&\!\!B(\cdot,u,Du) +(1-\sigma)\{\mathcal F[\underline u] - B(\cdot,\underline u,D\underline u)\} , \quad \mbox{in} \ \Omega,\\
\mathcal{G}[u] =\!\!&\!\! (1-\sigma)\mathcal G[\underline u] , \quad \mbox{on} \ \partial\Omega,
\end{array}
\end{equation}
for $0\le\sigma\le 1$. Further regularity follows from the Schauder approach and approximations for the standard linear elliptic theory in \cite[Chapter 6]{GTbook}, or the Aleksandrov-Bakel'man maximum principles and the $L^p$ regularity for the standard elliptic linear theory in \cite[Chapter 9]{GTbook}. The uniqueness readily follows from the comparison principle.
\end{proof}

With the above existence result for general operators, the existence for semilinear oblique problem \eqref{1.1}-\eqref{1.2} of the $k$-Hessian and Hessian quotient equations, $\mathcal{F}=F_{k,l}$ for $0\le l<k\le n, k>1$ in Theorem \ref{Th1.4}, is just a special case. The conditions in cases (i), (ii) in Theorem \ref{Th1.4} agree with those in (i), (ii) in Theorem \ref{Th4.1}, respectively. For case (iii) of Theorem \ref{Th1.4}, the gradient estimate follows from Lemma 3.2 in \cite{JTX2015}, while second derivative estimate is from Theorem \ref{Th1.2}. In the special case when $k = 1$, equation \eqref{1.1} reduces to a quasilinear Poisson equation, 
as the matrix $A$ can then be absorbed in the scalar $B$ and considerably more general results for arbitrary smooth domains $\Omega$ follow from the classical Schauder theory \cite{GTbook}. In particular we need only assume the quadratic growth, $B= O(|p|^2)$ as $|p| \rightarrow \infty$, uniformly for $x\in\Omega$, $|z|\le M$ for any $M > 0$, and under reduced smoothness hypotheses, $B\in C^{0,\alpha}(\bar\Omega\times \mathbb{R}\times \mathbb{R}^n)$, $\partial\Omega \in C^{1,\alpha}$, $\beta, \varphi \in C^{1,\alpha}(\partial\Omega)$, we infer the existence of a unique classical solution $u\in C^{2,\alpha}(\bar\Omega)$ of the semilinear oblique problem \eqref{1.1}-\eqref{1.2}.

Using  Lemma 4.1 in \cite{JTX2015} and the nonlinear case in Theorem \ref{Th1.2}, we can extend Theorem \ref{Th4.1} to cover nonlinear boundary operators in the case where $\Gamma$ is the positive cone $\Gamma_n$. For this we also need to assume that $\mathcal G$ is uniformly oblique in the sense that
\begin{equation}\label{G uniformly oblique}
G_p(x,z,p)\cdot\nu \ge \beta_ 0,\quad |G_p(x,z,p)| \le \sigma_0, \quad \mbox{on } \  \partial\Omega,
\end{equation}
for all $x\in \Omega$, $|z|\le M$, $p\in \mathbb{R}^n$ and  positive constants $\beta_0$ and $\sigma_0$, depending on the constant $M$. The following existence result, which is proved similarly to Theorem \ref{Th4.1}, extends
the Monge-Amp\`ere case, Theorem 4.2 in \cite{JTX2015} as well as case (iii) in Theorem \ref{Th1.4}.

\begin{Theorem}\label{Th4.2}
Assume that $F$ satisfies conditions F1-F4 and F6 in the positive cone $\Gamma_n$, $\Omega$ is a $C^{3,1}$ bounded domain in $\mathbb{R}^n$, $A\in C^2(\bar \Omega\times \mathbb{R}\times \mathbb{R}^n)$ is strictly regular in $\bar\Omega$, $B > a_0, \in C^2(\bar \Omega\times \mathbb{R}\times \mathbb{R}^n)$,  $G\in C^{2,1}(\partial\Omega\times \mathbb{R}\times\mathbb{R}^n)$ is concave with respect to $p$ and uniformly oblique in the sense of \eqref{G uniformly oblique}, $\underline u$ and $\bar  u$, $\in C^2(\Omega)\cap C^1(\bar \Omega)$ are respectively an admissible subsolution and a supersolution of the oblique boundary value problem \eqref{1.1}-\eqref{1.2} with $\Omega$ uniformly $(A,G)$-convex with respect to the interval $\mathcal I = [\underline u, \bar u]$.
Assume also that $A$, $B$ and $-G$ are non-decreasing in $z$,  with at least one of them strictly increasing and  $A$ satisfies the quadratic growth condition \eqref {quadratic structure of A}.  Assume either F5\textsuperscript{+} holds or $B$ is independent of $p$. Then there exists a unique admissible solution $u \in C^{3,\alpha}(\bar \Omega)$ of the boundary value problem \eqref{1.1}-\eqref{1.2} for any $\alpha<1$.
\end{Theorem}

\begin{remark}\label{Remark4.1}
Without the monotonicity, subsolution and supersolution hypotheses in Theorems \ref{Th4.1} and \ref{Th4.2} we can still obtain the existence of possibly non-unique, admissible solutions of the oblique boundary value problem \eqref{1.1}-\eqref{1.2}, using topological fixed point theorems, (cf. Theorem 11.6 in \cite{GTbook}), or  degree theory, (as in \cite{Urbas1995}),
 instead of the method of continuity, provided we have  {\it a priori} bounds for solutions of appropriate families such as 
 \eqref{homotopy family}, so that their ranges lie in fixed intervals $\mathcal I$, where $\partial\Omega$ is uniformly $(\Gamma,A,G)$-convex. Note also that the monotonicity conditions themselves may be relaxed somewhat to get the inequality \eqref{solution bound}. In particular we can strengthen the sub and super solution properties of $\underline u$ and $\bar u$ so that $D^2\underline u \ge A(x,z,D\underline u(x))$ and
\begin{equation}\label{monotone1}
\begin{array}{rll}
F[D^2\underline u-A(x,z,D\underline u(x))] \!\!&\!\!\displaystyle \ge B(x,z,D\underline u(x)), & \quad {\rm in} \ \Omega,\\
G(x,z,D\underline u(x)) \!\!&\!\!\displaystyle \ge 0, & \quad {\rm on}\ \partial\Omega,
\end{array}
\end{equation}
whenever $z<\underline u(x)$, with one of the inequalities in \eqref{monotone1} strict, and
\begin{equation}\label{monotone2}
\begin{array}{rll}
F[D^2\bar u-A(x,z,D\bar u(x))] \!\!&\!\!\displaystyle \le B(x,z,D\bar u(x)), & \quad {\rm in} \ \Omega,\\
G(x,z,D\bar u(x)) \!\!&\!\!\displaystyle \le 0, & \quad {\rm on}\ \partial\Omega,
\end{array}
\end{equation}
whenever $z>\bar u(x)$, $D^2\bar u\ge A(x,z,D\bar u(x))$, with one of \eqref{monotone2} strict. As a special case, if $\underline u=-K$, $\bar u=K$ for some constant $K$, we get $|u|\le K$; (see also \cite{JTX2015}, Section 3). Under this more general hypothesis we can then infer the existence of  admissible solutions in Theorems \ref{Th4.1} and \ref{Th4.2}. 
\end{remark}

\subsection{Examples}\label{subsection4.2}

In this subsection, we present various examples of operators $\mathcal{F}$, matrix functions $A$ and associated oblique boundary operators $\mathcal{G}$ which satisfy our hypotheses.

\vskip8pt

{\it Examples for $\mathcal{F}$.}  As already indicated in Section \ref{Section 1}, our main examples are the $k$-Hessian operators  and their  quotients  $F_{k,l}$ ($0\le l<k\le n$), as considered in Theorem \ref{Th1.4}. For $0\le l<k\le n$, $F_{k,l}$ satisfy F1-F5, F7 in $\Gamma_k$ with $a_0=0$. We remark that $b$ in F5 can be a positive constant or $+\infty$. For $l=0$, the corresponding $k$-Hessian operators $F_k$, $2\le k\le n$ satisfy F1-F4, F5\textsuperscript{+} and F7 in $\Gamma_k$ with $a_0=0$. In this case the operators $F_k$ only satisfy F5\textsuperscript{+} for finite $b$ but not  infinite $b$. w. Note that the normalised Monge-Amp\`ere operator in the form $(\det)^{\frac{1}{n}}$ is also covered by $F_k$ when $k=n$. Another well known concave form of the Monge-Amp\`ere operator is $\log (\det)$, which satisfies F1-F4, F5\textsuperscript{+} in $K^+$ with $a_0=-\infty$. As stated in the introduction,  the $k$-Hessian operators $F_k$ ($k=1,\cdots,n$) and the Hessian quotients $F_{n,l}$ ($1\le l \le n-1$) satisfy F6 in the positive cone $K^+$. If $\mathcal{F}$ is an operator satisfying F1, F2, F3, F5\textsuperscript{+}, with finite $a_0$,  then F6 holds in the positive cone $K^+$, since
\begin{equation}\label{verification of F6}
\begin{array}{rl}
\mathcal{E}_2 \!\!&\!\!\displaystyle \le r\cdot F_r |r|\\
              \!\!&\!\!\displaystyle \le (F(r)-a_0)|r|\\
              \!\!&\!\!\displaystyle = (B-a_0)|r|\\
              \!\!&\!\!\displaystyle \le o(|r|)\mathscr{T}, \quad\quad {\rm as} \  |r|\rightarrow \infty,
\end{array}
\end{equation}
where the property of $K^+$ is used in the first inequality, \eqref{homogeneity} is used in the second inequality, equation \eqref{1.1}, finite $a_0$ and F5\textsuperscript{+} are used in the last two lines. Such a property was observed by Urbas \cite{Urbas1995, Urbas2001} for orthogonally invariant  $\mathcal{F}$. Note that the property \eqref{verification of F6} here holds for non-orthogonally invariant $\mathcal{F}$ as well.

Instead of the elementary symmetric functions $S_k$, we may also consider  functions $P_k$, which are products of $k$ sums of eigenvalues, namely
\begin{equation}\label{operator Pk}
P_k[r]:=P_k(\lambda(r)) = \prod\limits_{i_1<\cdots<i_k}\sum_{s=1}^k \lambda_{i_s}(r), \quad k=1, \cdots, n,
\end{equation}
defined in the cones
\begin{equation}\label{cone Pk}
\mathcal{P}_k = \{r\in \mathbb{S}^n \ |\ \sum_{s=1}^k \lambda_{i_s}(r)>0 \},
\end{equation}
where $i_1,\cdots, i_k \subset \{1,\cdots,n\}$, $\lambda(r)=(\lambda_1(r),\cdots,\lambda_n(r))$ denote the eigenvalues of the matrix $r\in \mathcal{P}_k$.  In differential geometry, there is a large amount of literature dealing with $k$-convex hypersurfaces, where the notion $k$-convexity of a hypersurface, originating from \cite{Sha1986, Sha1987}, is that the sum of any $k$-principal curvatures at each point is positive.
Clearly the associated operators  in \eqref{operator Pk}  interpolate between the Laplacian, $k=n$, and the Monge-Amp\`ere operator, $k=1$. We then obtain another group of examples satisfying the hypotheses of Theorem \ref{Th1.2}, namely the normalised functions 
\begin{equation}
\tilde F_k :=(P_k)^{\frac{1}{C^k_n}}, \quad C^k_n=\frac{n!}{k!(n-k)!}, \quad 1\le k\le n, 
\end{equation}
which are homogeneous of degree one and satisfy F1-F5 in $\mathcal{P}_k$ with $a_0=0$.  Note that  the associated operators also interpolate between the Laplacian $\tilde {\mathcal F}_n = \mathcal {F}_1$ and  the normalised Monge-Amp\`ere operator $\tilde {\mathcal  F}_1 = \mathcal {F}_n $ and that the concavity F2 follows from the arithmetic-geometric mean inequality, similarly to the Monge-Amp\`ere case, $k=1$, (or as a consequence by virtue of the general property that concave functions of linear functions are also concave). For $1\le k\le n-1$, the functions $\tilde F_k$ also satisfy F5\textsuperscript{+} in $\mathcal{P}_k$ and F6 in the positive cone $K^+$. Using the property that at most $k-1$ of $\lambda_1(r),\cdots,\lambda_n(r)$ can be negative, we also see that  $\tilde F_k$ satisfies F7 in $\mathcal P_k$. Furthermore from Theorems \ref{Th1.2} and \ref{Th1.3}, it follows that we can substitute $\tilde F_k$ for $F_k$ and $\mathcal P_k$ for $\Gamma_k$ in cases (i) and (ii)(a) of Theorem \ref{Th1.4}. In the next subsection we will also introduce degenerate versions of these operators.

We also have further examples originating from geometric applications, given by functions,
\begin{equation}\label{F - alpha}
F_{k,-\alpha}[r]:= F_{k,-\alpha}(\lambda(r)) = \{\sum\limits_{i_1<\cdots<i_k}[\sum_{s=1}^k \lambda_{i_s}(r)]^{-\alpha}\}^{-\frac{1}{\alpha}}, \quad \alpha> 0,
\end{equation}
also defined in the  cone $\mathcal{P}_k$ for $k=1,\cdots,n$. When $\alpha = k = 1$, $F_{k,-\alpha}$ coincides with the Hessian quotient $F_{n,n-1}$ and if $\kappa = (\kappa_1, \cdots, \kappa_n)$ denotes the principal curvatures of a hypersurface in $\mathbb{R}^{n+1}$, then $F_ {1,- 1}[\kappa]$ is its  harmonic curvature while $F_{1, - 2}[\kappa]$ is the inverse of the length of the second fundamental form; see \cite{G}. The associated operators are homogeneous and satisfy F1-F5 and F7 in $\mathcal{P}_k$ with $a_0=0$ and either finite or infinite $b$ in F5.

The operators in the above examples are all orthogonally invariant. We also have examples of operators
$\mathcal{F}$ which are not orthogonally invariant. For instance, let us consider a set $V=\{Q_1,\cdots, Q_m\}$, where $Q_i$, $i=1,\cdots, m$ are nonsingular matrices and $m>1$ is a finite integer. We can define an operator of the form
\begin{equation}\label{non-orthogonal example}
F_{k,V}[r]=\min_{Q\in V} F_k (Q r Q^{-1}), \quad {\rm for}\ k=1, \cdots, n,
\end{equation}
in the cone $\Gamma_{k,V}= \{r\in \mathbb{S}^n \ |\ F_j(Q r Q^{-1})>0, \ \forall Q\in V, \ j=1,\cdots,k\}$. Then the operator in \eqref{non-orthogonal example} provides an example, which is non-orthogonally invariant, but still satisfies our assumptions F1-F5\textsuperscript{+} and F7. Note that since $F_{k,V}$ is a concave function in $r$, it has first and second order derivatives almost everywhere in $\Gamma_{k,V}$ so that the differential inequalities in  \eqref{F1 inequality}, \eqref{F2 inequality}, as well as condition F5\textsuperscript{+}, hold in this sense. We can also consider the case of infinite $V$ and replace $F_k$ by other functions.  The resulting Bellman type augmented Hessian operators can then be treated by smooth approximation as in the $k= 1$ case, (see for example \cite {GTbook}); and we would obtain the existence of $C^{2,\alpha} (\bar\Omega)$ solutions, for some $\alpha >0$ in Theorems \ref{Th1.4} and  \ref{Th4.1}. More generally if we drop the smoothness condition $F \in C^2 (\Gamma)$, then we still obtain existence of $C^{2,\alpha} (\bar\Omega)$ solutions, for some $\alpha >0$, in Theorems \ref{Th1.4}, \ref{Th4.1} and Remark \ref{Remark4.1}. Here we need the more general $C^{2,\alpha} (\bar\Omega)$ estimate for concave fully nonlinear uniformly elliptic equations from \cite {Tru1984}.

\vskip8pt

{\it Examples for $A$.} Examples of strictly regular matrix functions arising in optimal transportation and geometric optics can be found for example in \cite{MTW2005, TruWang2009, Tru2014, JT2014, LT2016}.   Typically there is not a natural association with oblique boundary operators, except for those coming from the second boundary value problem to prescribe the images of the associated mappings, so that second derivative estimates may depend on gradient restrictions in accordance with Remark \ref{Remark1.2}. Moreover the relevant equations typically  involve constraints so that we are also in the situation of  Remark \ref{Remark 1.3}. Both these situations will be further examined in ensuing work.
However we will give some examples satisfying our hypotheses,  where oblique boundary operators arise naturally through our domain convexity conditions.

Our first examples extend those coming from the conformal deformations of manifolds with boundary; (as for example in \cite{SSChen2007, JLL2007}). We introduce a class of matrix functions of the form
\begin{equation}\label{main example}
A(x,z,p) = \frac{1}{2} a_{kl}(x,z)p_kp_l I - a_0(x,z)p\otimes p,
\end{equation}
where $a_{kl}, a_0 \in C^2(\bar\Omega\times \mathbb{R})$ and  the matrix $\{a_{kl}\}>0$ in $\bar\Omega\times \mathbb{R}$.
Clearly for any vectors $\xi, \eta \in \mathbb{R}^n$, we have
\begin{equation}\label{verification of strict regular}
\begin{array}{rl}
\displaystyle A_{ij}^{kl}\xi_i\xi_j\eta_k\eta_l \!\!&\!\! = \displaystyle (a_{st}\delta_{sk}\delta_{tl}\delta_{ij}- 2a_0\delta_{ik}\delta_{jl})\xi_i\xi_j\eta_k\eta_l \\
\!\!&\!\!\displaystyle  =  |\xi|^2 a_{kl} \eta_k\eta_l - 2 a_0 (\xi\cdot \eta)^2\\
\!\!&\!\!\displaystyle  \ge   \lambda_1 |\xi|^2 |\eta|^2 - 2 a_0 (\xi\cdot \eta)^2,
\end{array}
\end{equation}
where $\lambda_1 > 0$ denotes the minimum eigenvalue of $\{a_{kl}\}$, so that $A$ is strictly regular in
$\bar\Omega$. Moreover $A$ is uniformly regular, with \eqref{1.19} satisfied with  $\lambda_0 = \inf \lambda_1$,
$\bar\lambda_0 = 2\sup a^+_0$, where the infimum and supremum are taken over $\Omega\times (-M,M)$.
For $A$ given by \eqref{main example}, the corresponding $A$-curvature matrix on $\partial\Omega$ for \eqref{1.6} is given by
\begin{equation}\label{A curvature main example}
K_A[\partial\Omega](x,z,p) = -\delta\nu(x) + a_{kl}(x,z)p_k\nu_l(x)(I-\nu(x)\otimes\nu(x)),
\end{equation}
where $\nu$ is the unit inner normal to $\partial\Omega$ and $\delta$ denotes the tangential gradient. Consequently the quasilinear boundary operator $\mathcal G$, given by \eqref{quasilinear} with $\beta_k = a_{kl}\nu_l$, will be oblique, satisfying
$\beta\cdot\nu \ge \lambda_0$ on $\partial \Omega$ and $\Omega$ is uniformly $(\Gamma, A, G)$-convex with respect to $u$ if and only if
\begin{equation}\label{uniform convexity in conformal geometry}
-\delta \nu + \varphi(\cdot,u)(I-\nu\otimes\nu) + \mu_0\nu(x)\otimes\nu(x) \in \Gamma,
\end{equation}
for some constant $\mu_0>0$, possibly depending on $u$. Accordingly our uniform convexity condition is independent of the gradient variables. In the orthogonally invariant case, letting
$$\tilde \Gamma = \lambda(\Gamma) = \{ \lambda\in \mathbb{R}^n\ | \ \lambda_1, \cdots, \lambda_n \ {\rm are \ eigenvalues\ of\ some }\ r\in \Gamma\}$$
 denote the corresponding cone to $\Gamma$ in $\mathbb{R}^n$, \eqref{uniform convexity in conformal geometry} is equivalent to $(\tilde\kappa, \mu_0) \in \tilde \Gamma$, where $\tilde\kappa_i = \kappa_i + \varphi$, $i= 1,\cdots, n-1$, and $\kappa = (\kappa_1,\cdots, \kappa_{n-1})$ denotes the principal curvatures of $\partial\Omega$. In particular for the cones $\Gamma_k$, \eqref{uniform convexity in conformal geometry} is equivalent to $\tilde\kappa\in \tilde \Gamma_{k-1}$, that is $S_j(\tilde\kappa) > 0$ for $1\le j\le k-1$.

\vskip8pt

{\it Conformal geometry.}
The application to conformal geometry concerns the special case $a_{ij} = \delta_{ij}$, $a_0 = 1$ in \eqref{main example}, that is
\begin{equation}\label{Euclidean Yamabe}
A(p)=\frac{1}{2}|p|^2 I-p\otimes p,
\end{equation}
with the associated semilinear Neumann condition,
\begin{equation}
D_\nu u =\varphi (x,u),  \quad \mbox{on} \ \partial\Omega,
\end{equation}
and is related to the fully nonlinear Yamabe problem with boundary, where $A_{\tilde g} = e^{2u}M[u]$ is  the Schouten tensor of the conformal deformation $\tilde g =e^{-2u} g_0$ and $g_0$ denotes the standard metric on
$\mathbb{R}^n$. If $F$ is positive homogenous of degree one, satisfying F3 with $a_0 = 0$, $\tilde \varphi$ is a positive function on $\Omega$ and
$\tilde h$ a function on $\partial\Omega$, then the problem of finding a conformal metric $\tilde g$ on $\Omega$ such that $F(A_{\tilde g}) = \tilde \varphi$, with mean curvature $\tilde h$  on $\partial\Omega$, is equivalent to solving the semilinear Neumann problem,
\begin{equation}\label{transformed Yamabe Euc}
\begin{array}{cl}
\displaystyle \mathcal F[u] := F(M[u])= \tilde \varphi e^{-2u}, &   {\rm in} \ \Omega,\\
\displaystyle D_\nu u =\tilde he^{-u} - h_0,                   &  {\rm on}\ \partial \Omega,
\end{array}
\end{equation}
where $h_0$ denotes the mean curvature of $\partial \Omega$ with respect to $g_0$, \cite{JLL2007}.  With $\Omega$, $\tilde \varphi$ and $\tilde h$ sufficiently smooth, \eqref {transformed Yamabe Euc} satisfies the hypotheses of the second derivative estimate, Theorem \ref{Th1.2}, if $F$ also satisfies F1 and F2 and $\Omega$ satisfies \eqref{uniform convexity in conformal geometry} with $\varphi = \tilde h e^{-u} - h_0$. Note that our restriction $r \le$ trace$(r)I$ on $\Gamma$ implies that $\tilde h > 0$. However \eqref{uniform convexity in conformal geometry} does provide some relaxation of the umbilic condition in \cite{JLL2007} and related papers, possibly depending on solution upper bounds, and can be extended to more general Riemannian manifolds with boundary, (taking account of Remarks \ref{Remark 2.1} and \ref{Remark3.6}), as well as to more general boundary curvatures. In particular for the cones, $\Gamma_2$ and $\mathcal P_{n-1}$, the convexity condition \eqref{uniform convexity in conformal geometry} is equivalent to $\tilde h > 0$, since $\delta\cdot\nu = (1-n)h_0$ so no geometric conditions are needed. Note that for the  local gradient bound, Theorem \ref{Th3.1}, we only need $F$ to  satisfy F1 to fulfil the hypotheses of case (i) (and no geometric restrictions on $\Omega$).  Since the functions $B$ and $\varphi$ are not monotone increasing in $z$ we would need though {\it a priori} solution bounds to get existence and this is still an unresolved issue in the non-umbilic case.

\vskip8pt

{\it Optimal transportation and geometric optics.}
In optimal transportation problems, the matrix $A$ is generated by a cost function $c \in C^2(\mathcal D)$, where $\mathcal D$ is a domain in $\mathbb{R}^n\times\mathbb{R}^n$, through the relation
\begin{equation}
A(x,p)= c_{xx}(x,Y(x,p)),
\end{equation}
where the mapping $Y \in C^1(\mathcal U)$, for some domain $\mathcal U \in \mathbb{R}^n\times\mathbb{R}^n$, is given as the unique solution of
\begin{equation}\label{solve Y}
c_x(x,Y)=p.
\end{equation}
Here we assume conditions A1, A2 as in \cite{MTW2005, TruWang2009} to guarantee the unique solvability of $Y$ from \eqref{solve Y}. The strict regularity was introduced as condition A3 in \cite{MTW2005}; (see also \cite{Tru2006}). More generally the matrices $A$ arise from prescribed Jacobian equations \cite{Tru2008} where now the mapping $Y \in C^1(\mathcal U)$ is given for a domain $\mathcal U \in \mathbb{R}^n\times\mathbb{R}\times\mathbb{R}^n$ satisfying det$Y_p \ne 0$ in $\mathcal U$ and the matrix $A$ is given by
\begin{equation}\label{PJE}
A(x,z,p) = Y^{-1}_p(Y_x + Y_z \otimes p).
\end{equation}
Mappings $Y$ in geometric optics can also be unified through a notion of generating function \cite {Tru2014}, which extends that of a cost function to permit the $z$ dependence in $Y$  and provides symmetric matrices in \eqref{PJE}. For further information and particular examples of strictly regular matrices $A$ the reader is referred to \cite{MTW2005, TruWang2009, Tru2014, JT2014, LT2016} and the references therein. As mentioned above in most of these examples there are not natural relationships with semilinear oblique boundary operators so that the situation in Remarks \ref{Remark1.2} and
\ref{Remark 1.3} is applicable. The natural  boundary condition is the prescription of the image $\Omega^*$ of the mapping $T:=Y(\cdot,u,Du)$ on $\Omega$, which implies a boundary condition which is oblique with respect to admissible functions, \cite{TruWang2009, Tru2008}. Once the obliqueness is estimated we are in the situation of Theorem \ref{Th1.2} and moreover our domain convexity conditions there originate from those used in the optimal transportation and more generally; (see \cite{TruWang2009, Tru2008, LT2016}).

Accordingly  we just mention here some examples which fit simply with our hypotheses. First the logarithm cost function, given by $c(x,y) = \frac{1}{2}\log |x-y|$ for $x\ne y$, also generates our example \eqref{Euclidean Yamabe}, \cite{TruWang2009}.
From geometric optics we have the example coming from the reflection of a parallel beam to a flat target,
\cite{Tru2014, LT2016},
\begin{equation}\label{ reflection}
A(x,z,p) = \frac{1}{2z}( |p|^2 - 1) I
\end{equation}
for $z > 0$. Here there is a constraint, namely $u>0$, which is readily handled by taking a logarithm or assuming the subsolution $\underline u > 0$ in $\bar\Omega$.  Then for a semilinear Neumann boundary condition of the form
\begin{equation}
D_\nu u = u\varphi (\cdot, u),
\end{equation}
we obtain again that $\Omega$ is uniformly $(\Gamma, A, G)$-convex with respect to $u$ if and only if
\eqref{uniform convexity in conformal geometry} holds.
\vskip8pt

{\it Admissible functions.}
Quadratic functions of the form $u_0 = c_0 + \frac{1}{2}\epsilon |x-x_0|^2$, will be admissible for the matrices
\eqref{main example} and for arbitrary constants $c_0$, points $x_0 \in \Omega$ and sufficiently small $\epsilon$.
In general for matrices $A$ arising in optimal transportation and geometric optics  the existence of admissible functions is proved in \cite{JT2014}.

\vskip8pt

{\it Nonlinear boundary operators.}
The capillarity type operators, given by
\begin{equation}
G(x,z,p) = p\cdot\nu - \theta(x)\sqrt{1+|p|^2} -\varphi(x,z),
\end{equation}
would satisfy our hypotheses for $0<\theta<1$ on $\partial\Omega$. Furthermore for $A$ in the form \eqref{main example} with $\{a_{ij}\} = I$, condition \eqref{uniform convexity in conformal geometry} would at least imply that that $\Omega$ is uniformly $(\Gamma, A, G)$-convex with respect to $u$. Note that here and quite generally we cannot have $\varphi (\cdot, u) \ge 0$ everywhere on $\partial\Omega$ for an admissible function so the basic capillarity condition is ruled out by our concavity condition which requires $\theta > 0$.

\subsection{Degenerate equations}\label{subsection4.3}
In this subsection, we consider the extension of our results to degenerate elliptic equations and in particular apply the classical existence results, Theorems \ref{Th4.1} and \ref{Th4.2}, to yield the existence of $C^{1,1}$ admissible solutions for the oblique boundary value problems.  We shall use the following assumption, in place of F1, to describe the degenerate ellipticity:
\begin{itemize}
\item[{\bf F1$^-$}:]
$F$ is non-decreasing in $\Gamma$, namely
\begin{equation}\label{F1- inequality}
F_r := F_{r_{ij}} = \left \{ \frac{\partial F}{\partial r_{ij}} \right \} \ge 0, \ {\rm in} \ \Gamma,
\end{equation}
and $\mathscr{T}(r): ={\rm {trace}}(F_r) > 0 $ in $\Gamma$.
\end{itemize}

Then using an elliptic regularisation as in \cite{Tru1990}, we define for  a constant $\epsilon \ge 0$, $F_1(r) = {\rm trace}(r)$, approximating operators and cones,
\begin {equation}\label{regularisation}
F^\epsilon (r) = F(r + \epsilon F_1(r)I), \quad \Gamma^\epsilon = \{r + \epsilon F_1(r)I\in\mathbb{S}^n \ | \ r \in \Gamma\}.
\end{equation}
Clearly $F^\epsilon$ satisfies the ellipticity condition F1 in the cone $\Gamma^\epsilon$, for $\epsilon > 0$ and is also uniformly elliptic there with
\begin{equation}
  \epsilon \mathscr{T}(r)I \le F^\epsilon_r \le (1+\epsilon) \mathscr{T}(r)I.
 \end{equation}
Moreover if $F$ also satisfies any of conditions F2 to F7, then $F^\epsilon$ satisfies the same condition in $\Gamma^\epsilon$ with relevant constants independent of $\epsilon$, as $\epsilon$ tends to $0$. Consequently we may replace $F$ by $F^ \epsilon$  and the operator $\mathcal F$ by  $\mathcal F^\epsilon$,  for sufficiently small $\epsilon \ge 0$ in our Hessian and gradient estimates in Sections \ref{Section 2} and \ref{Section 3}. To get the lower second derivative bounds in Theorem  \ref{Th1.1} we can simply use $\mathscr{T}^\epsilon(r): =  {\rm {trace}}(F^\epsilon_r)  > 0 $
 in $\Gamma_\epsilon$, while for the lower tangential bounds in Theorem \ref{Th1.2} we now have, from our restriction  on $\Gamma$ in the quasilinear case,
$$  M[u] \le (1+n\epsilon) \mathcal F_1[u]I $$
so that we arrive again at an estimate of the form \eqref{full boundary}. Note that we only need sufficiently small
$\epsilon$ for the quasilinear case of Theorem \ref{Th1.2}. Clearly we could have assumed the weaker condition F1$^-$ at the outset for our derivative estimates in Sections \ref{Section 2} and \ref{Section 3} but it is not feasible then to consider solutions with smooth second derivatives.  By approximation we now obtain from Theorem \ref{Th4.1} and \ref{Th4.2} the following extension to $C^{1,1}(\bar\Omega)$ solvability of degenerate equations. Here a function $u\in C^{1,1}(\Omega)$ is admissible
if $M[u] \in \Gamma$ almost everywhere in $\Omega$ and is a solution of equation \eqref{1.1} if it is a solution almost everywhere in $\Omega$.

\begin{Corollary}\label{Cor4.1}
In the hypotheses of Theorems \ref{Th4.1} and \ref{Th4.2} assume that condition F1 is weakened to condition F1$^-$, with $\bar u$,
$\underline u \in C^{1,1}(\bar\Omega)$ and the supersolution condition strengthened  so that
$\bar u$ is a supersolution of \eqref{1.1}-\eqref{1.2} with $B$ replaced by $B -\delta$ for some positive constant
$\delta$. Then there exists an admissible solution $u \in C^{1,1}(\bar \Omega)$ of the boundary value problem, \eqref{1.1}-\eqref{1.2}.
\end{Corollary}

\begin{proof}
We claim that $\bar u$ and $\underline u$ are respectively supersolution and admissible subsolution of the boundary value problem \eqref{1.1}-\eqref{1.2} for $\mathcal F$ replaced by $\mathcal F^\epsilon$ for sufficiently small $\epsilon$, depending on $\bar u$ and $\delta$. To prove this we first define the sets
$$\Omega^\prime_\epsilon = \{ x\in \bar\Omega \ |\  M[\bar u](x) + \epsilon F_1(M[\bar u](x))I\in \Gamma_\epsilon\}, \quad K_\epsilon = \{x\in\bar \Omega^\prime_\epsilon\ | \ \mathcal F^\epsilon[\bar u]\ge a\}, $$
where $a$ is constant satisfying $a_0 < a < B(\cdot,\bar u, D\bar u)$ in $\Omega$. Then $K_\epsilon$ is a decreasing family of compact subsets of $\bar\Omega$ approaching $K_0$ as $\epsilon$ approaches zero. Consequently $K_\epsilon\subset \Omega^\prime_0$ for sufficiently small $\epsilon$. By the concavity F2, we then have, in $K_\epsilon$,
\begin{equation}
\mathcal F^\epsilon[\bar u] \le \mathcal F[\bar u] +\epsilon \mathcal F_1[\bar u] \mathscr{T}(M[\bar u]) \le B(\cdot,\bar u, D\bar u),
\end{equation}
for sufficiently small $\epsilon$ depending on  $\bar u$ and $\delta$. Clearly $\mathcal F^\epsilon[\bar u] \le B(\cdot,\bar u, D\bar u)$ in $ \Omega^\prime_\epsilon - K_\epsilon $ so that $\bar u$ is a supersolution of the equation,
$\mathcal F^\epsilon = B$, for sufficiently small $\epsilon$. Next it follows immediately from the degenerate ellipticity F1$^-$, that $\underline u$ is an admissible subsolution, for any $\epsilon \ge 0$ so that our claim is proved.

From Theorem \ref{Th4.1}, (or Theorem \ref{Th4.2}), and Theorems \ref{Th1.1}, \ref {Th1.2} and \ref{Th1.3}, there exists a unique solution $u_\epsilon\in C^{3,\alpha}(\bar \Omega)\cap C^4(\Omega)$ of the problem \eqref{1.1}-\eqref{1.2} with $\mathcal F = \mathcal F^\epsilon$ for sufficiently small positive $\epsilon$, together with the {\it a priori} estimates
\begin{equation}\label{uniform estimates}
|u_\epsilon|_{2;\Omega} \le C
\end{equation}
with constant $C$  independent of $\epsilon$. Hence there exists a subsequence $u_{\epsilon_k}$ and a function $u\in C^{1,1}(\bar \Omega)$ such that
\begin{equation}
u_{\epsilon_k}\rightarrow u \quad {\rm in}\ C^{1,\alpha}(\bar \Omega),\quad \forall \alpha\in (0,1),\quad {\rm as} \ \epsilon_k\rightarrow 0.
\end{equation}
From the stability property of the theory of viscosity solutions  \cite{CIL1992}, it is readily seen that $u\in C^{1,1}(\bar \Omega)$ is an admissible solution of the problem \eqref{1.1}-\eqref{1.2}.
\end{proof}

To illustrate the application of Corollary \ref{Cor4.1}, we consider the degenerate elliptic operators $\mathfrak{m}_k$,
given  by functions
\begin{equation}\label{m operator}
 \mathfrak{m}_k(r)=\min \{\sum_{s=1}^k \lambda_{i_s}(r)\},
\end{equation}
for $k=1,\cdots, n$, $i_1,\cdots, i_k \subset \{1,\cdots,n\}$, in the cones $\mathcal P_k$ introduced in \eqref{cone Pk}. As for the examples \eqref{non-orthogonal example}, the functions $\mathfrak{m}_k$ for $k <n$ are not $C^2$ but will still satisfy conditions F1$^-$, F2, F3, F4, F5 and F7, with $a_0 = 0$, in $\mathcal P_k$ almost everywhere. As well $\mathfrak{m}_k$ is positive homogeneous of degree one. The operators
$\mathfrak{m}_k$ are also related to our examples \eqref{F - alpha} since $ \mathfrak{m}_k = F_{k,\infty} = \lim\limits_{\alpha \rightarrow \infty} F_{k,-\alpha} $. More explicitly the functions $F_{k,-\alpha}$ are monotone decreasing in $\alpha$
and satisfy the inequalities
$$  \mathfrak{m}_k < F_{k,-\alpha} < (C^k_n)^{-\frac{1}{\alpha}} \mathfrak{m}_k$$
in $\mathcal P_k$. We also note that when $k =n$, $\mathfrak{m}_n$ is the Poisson operator  $\mathcal F_1$.

By suitable approximation of the ``minimum'' function we then obtain from Corollary \ref{Cor4.1} the following analogue of Theorem \ref{Th1.4}.

\begin{Corollary}\label{Cor4.2}
Let $\mathcal F = \mathfrak{m}_k$, for some $k =1, \cdots, n-1$, $\Omega$  a $C^{3,1}$ bounded domain in $\mathbb{R}^n$, $A\in C^2(\bar \Omega\times \mathbb{R}\times \mathbb{R}^n)$ strictly regular in $\bar\Omega$, $B > 0, \in C^2(\bar \Omega\times \mathbb{R}\times \mathbb{R}^n)$,
$\mathcal G$ semilinear and oblique with $G\in C^{2,1}(\partial\Omega\times \mathbb{R}\times\mathbb{R}^n)$ satisfying \eqref{semilinear}. Assume that $\underline u$ and $\bar  u$, $\in C^{1,1}(\Omega)\cap C^1(\bar \Omega)$ are respectively an admissible subsolution of \eqref{1.1}-\eqref{1.2} and supersolution of \eqref{1.1}-\eqref{1.2} with
$B$ replaced by $B -\delta$, for some positive constant $\delta$, and $\Omega$ is uniformly $(\mathcal{P}_k,A,G)$-convex with respect to the interval $\mathcal I = [\underline u, \bar u]$. Assume also that $A$, $B$ and $\varphi$ are non-decreasing in $z$, with at least one of them strictly increasing,  $A$ satisfies the quadratic growth conditions \eqref{quadratic growth} and $B$ is independent of $p$. Then if one of the following further conditions is satisfied:
\begin{itemize}
\item[(i):]  $A$ is uniformly regular;
\item[(ii):]   $\beta=\nu$ and $A = o(|p|^2)$  in \eqref{quadratic growth};
\item[(iii):] $k=1$ and $A\ge O (|p|^2)I$ in place of \eqref{quadratic growth},
\end{itemize}
there exists an admissible solution $u \in C^{1,1}(\bar \Omega)$ of the boundary value problem \eqref{1.1}-\eqref{1.2}.
\end{Corollary}

Furthermore the operator $\mathfrak{m}_1$ satisfies condition F6 in $\mathbb{S}^n$ since, from its orthogonal invariance, we have
\begin{equation}
\mathcal{E}_2 = [\mathfrak{m}_1(r)]^2 \le \max\{a^2,b^2\}\le o(|r|)\mathscr{T}, \quad {\rm as}\ |r|\rightarrow \infty,
\end{equation}
for $r\in \mathbb{S}^n$, $a \le \mathfrak{m}_1(r) \le b$ and  using $\mathscr{T}=1$.
Obviously, $\mathfrak{m}_1$ satisfies F6 in the positive cone $K^+$. Then again by  approximation, from Corollary \ref{Cor4.1}, we have the following existence of $C^{1,1}$ admissible solutions for oblique boundary value problem \eqref{1.1}-\eqref{1.2} with $\mathcal{F}=\mathfrak{m}_1$ and nonlinear $\mathcal{G}$.

\begin{Corollary}\label{Cor4.3}
Let $\mathcal F = \mathfrak{m}_1$, $\Omega$  a $C^{3,1}$ bounded domain in $\mathbb{R}^n$, $A\in C^2(\bar \Omega\times \mathbb{R}\times \mathbb{R}^n)$ strictly regular in $\bar\Omega$, $B > 0, \in C^2(\bar \Omega\times \mathbb{R}\times \mathbb{R}^n)$,
$G\in C^{2,1}(\partial\Omega\times \mathbb{R}\times\mathbb{R}^n)$ is concave with respect to $p$ and uniformly oblique in the sense of \eqref{G uniformly oblique}. Assume that $\underline u$ and $\bar  u$, $\in C^{1,1}(\Omega)\cap C^1(\bar \Omega)$ are respectively an admissible subsolution of \eqref{1.1}-\eqref{1.2} and supersolution of \eqref{1.1}-\eqref{1.2} with
$B$ replaced by $B -\delta$ for some positive constant $\delta$ and $\Omega$ uniformly $(A,G)$-convex with respect to the interval $\mathcal I = [\underline u, \bar u]$. Assume also that $A$, $B$ and $-G$ are non-decreasing in $z$, with at least one of them strictly increasing,  $A$ satisfies the quadratic growth conditions \eqref{quadratic structure of A} and $B$ is independent of $p$. Then there exists an admissible solution $u \in C^{1,1}(\bar \Omega)$ of the boundary value problem \eqref{1.1}-\eqref{1.2}.
\end{Corollary}

We remark that we may also prove Corollaries \ref{Cor4.2} and \ref{Cor4.3} directly from Theorem \ref{Th4.1} by approximating $\mathfrak{m}_k$ by $F_{k,-\alpha}$ for large $\alpha$. Also the solutions in Corollaries \ref{Cor4.1}, \ref{Cor4.2} and \ref{Cor4.3} will be unique if either $A$ or $B$ are strictly increasing and more generally under appropriate barrier conditions, as considered in Section 2 of \cite{JT-oblique-II}.


\subsection{Final remarks}\label{subsection4.4}

The oblique boundary value problem \eqref{1.1}-\eqref{1.2} for augmented Hessian equations is natural in the classical theory of fully nonlinear elliptic equations. In this paper and its sequel \cite{JT-oblique-II}, we have treated this problem in a very general setting. Through  {\it a priori} estimates, we have established the classical existence theorems under appropriate domain convexity hypotheses for both (i) strictly regular $A$ and semilinear or concave  $\mathcal{G}$, and (ii) regular $A$ and uniformly concave $\mathcal{G}$. Our emphasis in this paper is the case (i), since the case (ii) is already known in the context of the second boundary value problems of Monge-Amp\`ere equations \cite{Urbas1998, Urbas2001} and optimal transportation equations \cite{TruWang2009, vonNessi2010}. In case (i), the boundary conditions can be any oblique conditions, including the special case of the Neumann  problem, while the operators embrace a large class including the Monge-Amp\`ere operator, $k$-Hessian operators and their quotients, as well as  degenerate  and non-orthogonally invariant operators.

In part II \cite{JT-oblique-II} we treat the case of regular matrices $A$ which includes the basic Hessian equation case, where $A = 0$ or more generally where $A$ is independent of the gradient variables. A fundamental tool here is the extension of our barrier constructions for Monge-Amp\`ere operators in \cite{JTY2013, JT2014} to general operators; (see Remarks in Section 2 of \cite{JTX2015}). In general as indicated by the Pogorelov example, \cite{Wang1992, Urbas1995}, we cannot expect second derivative estimates for arbitrary linear oblique boundary  conditions  and moreover the strict regularity of $A$ is critical for our second derivative estimates in Section \ref{Section 2}. We remark though that our methods in this paper, as further developed in \cite{JT-oblique-II}, also show that the strict regularity can be replaced by not so natural, strong monotonicity conditions with respect to the solution variable on either the matrix $A$ or the boundary function $\varphi$, that is either $A_z$ or $\varphi_z$ is sufficiently large, and the latter would include the case when $A = 0$, in agreement with the Monge-Amp\`ere case in \cite{Urbas1987, Wang1992, Urbas1995}.   For Monge-Amp\`ere type operators, we are able to derive the second derivative bound for semilinear Neumann boundary value problem when $A$ is just regular, under additional assumption of the existence of an admissible supersolution $\bar u$, satisfying $\det (M[\bar u]) \le B(\cdot, \bar u, D\bar u)$ in $\Omega$ and $D_\nu \bar u= \varphi(\cdot,\bar u)$ on $\partial\Omega$; (see Jiang et al. \cite{JTX2015}). This is an extension of the fundamental result in \cite{LTU1986} for the standard Monge-Amp\` ere operator, ( although the supersolution hypothesis is not needed in \cite{LTU1986} and more generally when $D_{px}A = 0$ \cite{JTX2015}).  For the semilinear oblique problem for standard $k$-Hessian equations, the known results due to Trudinger \cite{Tru1987} and Urbas \cite{Urbas1995}, where the second derivative estimates for Neumann problem in balls, and for oblique problem in general domains in dimension two respectively were studied. Recently the Neumann problem for the standard $k$-Hessian equation has been studied in uniformly convex domains in \cite{MQ}. However, it would be reasonable to expect  there are  corresponding second derivative estimates for admissible solutions of the Neumann problem for $k$-Hessian equations in uniformly $(k-1)$-convex domains. Also, the second derivative estimates for admissible solutions of the Neumann problem of the augmented $k$-Hessian equations with only regular $A$ in uniformly $(\Gamma_k, A, G)$-convex domains is still an open problem.

In Section \ref{Section 3}, we have established the gradient estimate for augmented Hessian equations in the cones $\Gamma_k$ when $k>n/2$ under  structure conditions for $A$ and $B$ corresponding to the natural conditions of Ladyzhenskaya and Ural'tseva for quasilinear elliptic equations \cite{LU1968, GTbook}. The gradient estimate under natural conditions is also known for $k=1$, in \cite{GTbook}. Therefore, it would be interesting to prove gradient estimates,  (interior and global), for both oblique and Dirichlet boundary value problems under natural conditions for the remaining cases for operators in the cones $\Gamma_k$ when $2\le k\le n/2$, and in particular for the basic Hessian operators $F_k$ when $2\le k\le n/2$, which also enjoy  $L^p$ gradient estimates for $p < nk/(n-k)$, \cite{TruWang1999}.  In \cite {JT-new}, we apply our gradient estimates here and general barrier constructions in Section 2 of \cite{JT-oblique-II} to study the classical Dirichlet problem for general augmented Hessian equations with only regular matrix functions $A$. Here as well as our conditions on $F$ in case (ii) of Theorem \ref{Th1.1}, for global second derivative estimates we also need to assume orthogonal invariance and the existence of an appropriate subsolution, as in our previous papers \cite{JTY2013, JTY2014}. Our barrier constructions in \cite{JT-oblique-II} also permit some relaxation of the conditions on $F$ in the regular case, as already indicated in Remark 3.2.

As pointed out in Remark \ref{Remark1.2}, our domain convexity conditions require some relationship between the matrix $A$ and the boundary operator $\mathcal G$. If we drop these from our hypotheses, we can still infer the existence of classical solutions of the equation \eqref{1.1} which are globally Lipschitz continuous and satisfy the boundary condition \eqref{1.2} in a weak viscosity sense \cite{CIL1992} so that our domain convexity conditions should become conditions for boundary regularity. This situation will be further amplified in a future paper, along with examination of the sharpness of our convexity conditions. A preliminary result here for the conformal geometry application is given in \cite{LiNguyen2012}.

\vspace{3mm}


\end{document}